\documentclass[11pt,reqno]{amsart}
\usepackage[brazilian,english]{babel} 
\usepackage[T1]{fontenc}
\usepackage{amsfonts,amsthm,amssymb,amsmath}  
\usepackage{graphicx,color}       
\usepackage{latexsym}       
\usepackage{overpic}        
\usepackage[colorlinks=false, linktocpage=true]{hyperref}
\usepackage{graphicx}
\usepackage{multicol}
\usepackage{color}
\usepackage{multirow} 

\newcommand{\leqnomode}{\tagsleft@true}
\newcommand{\reqnomode}{\tagsleft@false}
\makeatother

\newtheorem{teore}{Theorem}[section]
\newtheorem{obs}[teore]{Remark}
\newtheorem{defi}[teore]{Definition}

\newtheorem{coro}[teore]{Corollary}

\newtheorem{pro}[teore]{Proposition}

\newtheorem{lem}[teore]{Lemma}
\newtheorem{theorem}{}

\newcommand{\R}{\mathbb{R}}
\newcommand{\N}{\mathbb{N}}
\newcommand{\C}{\mathbb{C}}

\newcommand{\psib}{\mbox{\boldmath$\psi$}}
\newcommand{\phib}{\mbox{\boldmath$\phi$}}
\newcommand{\varphib}{\mbox{\boldmath$\varphi$}}

\numberwithin{equation}{section}

\title[NLS systems with quadratic-type nonlinearities]{On a system of Schr\"odinger equations with general quadratic-type nonlinearities}

\author[N. Noguera]{Norman Noguera}
\address{IMECC-UNICAMP, Rua S\'ergio Buarque de Holanda, 651, 13083-859, Cam\-pi\-nas-SP, Bra\-zil}
\email{nnoguera57@gmail.com}

\author[A. Pastor]{Ademir Pastor}
\address{IMECC-UNICAMP, Rua S\'ergio Buarque de Holanda, 651, 13083-859, Cam\-pi\-nas-SP, Bra\-zil}
\email{apastor@ime.unicamp.br}
\date{\today}

\begin{document}

\maketitle

\begin{abstract}
In this work we study a system of Schr\"{o}dinger equations involving  nonlinearities with quadratic growth. We establish sharp criterion concerned with the dichotomy global existence versus blow-up in finite time. Such a criterion is given in terms on the ground state solutions associated with the corresponding elliptic system, which in turn are obtained by applying variational methods. By using the concentration-compactness method we also investigate the nonlinear stability/instability of the ground states.  
\end{abstract}

\tableofcontents

\section{Introduction}

In this paper we  consider the following initial-value problem
\begin{equation}\label{system1}
\begin{cases}
\displaystyle i\alpha_{k}\partial_{t}u_{k}+\gamma_{k}\Delta u_{k}-\beta_{k} u_{k}=-f_{k}(u_{1},\ldots,u_{l})\\
(u_{1}(x,0),\ldots,u_{l}(x,0))=(u_{10},\ldots,u_{l0}),\qquad k=1,\ldots l,
\end{cases}
\end{equation}
where $u_{k}:\R^{n}\times \R\to \C$, $(x,t)\in \R^{n}\times \R$, $\Delta$ is the Laplacian operator, $\alpha_{k}, \gamma_{k}>0$, $\beta_{k}\geq0$ are real constants and  the nonlinearities $f_{k}$ satisfy some suitable condition that will be displayed below. Our main interest here  is to study \eqref{system1} when the nonlinearities  have a quadratic-type growth.

Systems as in \eqref{system1}, with power-like quadratic nonlinearities, appear in several areas in physics such as nonlinear optics, plasma physics, propagation in
nonlinear fibers, among others. In nonlinear optics, for instance, such systems can be derived in view of  the so-called multistep cascading mechanism. In particular, multistep cascading can be achieved by second-order
nonlinear processes such as second harmonic generation (SHG) and sum-frequency
mixing (SFM) (see, for instance, \cite{Kiv}). To cite a few examples, when the propagation of optical beams in a nonlinear dispersive medium
with quadratic response is considered, the following  three-wave interaction models appear (see \cite{Kiv})
\begin{equation}\label{system1A}
\begin{cases}
\displaystyle 2i\partial_{t}w+\Delta w-\beta w=-\frac{1}{2}(u^{2}+v^{2}),\\
\displaystyle i\partial_{t}v+\Delta v-\beta_{1} v=-\chi \overline{v}w,\\
\displaystyle i\partial_{t}u+\Delta u-u=- \overline{u}w,
\end{cases}
\end{equation}
and
\begin{equation}\label{system1B2}
\begin{cases}
\displaystyle i\partial_{t}w+\Delta w- w=-(\overline{w}v+\overline{v}u),\\
\displaystyle 2i\partial_{t}v+\Delta v-\beta v=-\left(\frac{1}{2}w^{2}+ \overline{w}u\right),\\
\displaystyle 3i\partial_{t}u+\Delta u-\beta_{1}u=- \chi vw,
\end{cases}
\end{equation}
where $\beta, \beta_{1},\chi>0$ are real constants. In \cite{Pastor2}, the author studied  \eqref{system1A} and  \eqref{system1B2} in the one-dimensional case. Global well-posedness, existence of ground state solutions and linear stability were analyzed.

In \cite{Hayashi}, the authors considered the system
\begin{equation}\label{system1J}
\begin{cases}
\displaystyle i\partial_{t}u+\Delta u=-2\overline{u}v,\\
\displaystyle i\partial_{t}v+\kappa\Delta v=- u^{2},
\end{cases}
\end{equation}
which appears as a non-relativistic version  of some Klein-Gordon systems, when the speed of light constant tend to infinity. It also can be derived as a model in $\chi^{(2)}$ media (see \cite{Colin2}). 
In \cite{Hayashi}, the authors also established  local and global well-posedness theories in $H^{1}(\R^n)$, $L^{2}(\R^n)$, and in some $L^{2}$-weighted spaces. Among other things, they also proved existence of ground state solutions and a sharp sufficient condition for global solutions in the critical case  ($n=4$). In dimension $n=5$, the dichotomy global well-posedness versus blow up in finite time,  was studied in \cite{Hamano} and \cite{NoPa}, whereas the scattering properties was established in \cite{Hamano} (with mass-resonance condition) and in \cite{Hamano1} (without mass-resonance condition). Also, the scattering below the ground state, in dimension $n=4$, was dealt with in \cite{Inui1}.

Additional properties  of system \eqref{system1J} and additional models of two and three wave systems with quadratic nonlinearities can be found in \cite{Buryak}, \cite{Buryak2}, \cite{Colin}, \cite{Colin2}, \cite{Dinh}, \cite{Hayashi3}, \cite{Yew}, \cite{Yew2}, \cite{Zhao} and references therein. Particularly, in  \cite{Buryak} is presented an extensive overview about models in $\chi^{(2)}$ media; derivation of sets of equations with quadratic nonlinearities from Maxwell's equations is done. Others references in a similar spirit are \cite{Colin2} and \cite{Yew}. 

Inspired by these works we intent to provide sufficient conditions on the interactions terms, $f_{k}$, to study the dynamics of system \eqref{system1}.  General nonlinearities with quadratic interactions were considered for example in references \cite{Li}, \cite{LiHa} and \cite{Zhang}. These works were dedicated to the study the Cauchy problem in two dimensions. Here we  consider system \eqref{system1} in dimensions $1\leq n\leq 6$. Also, it is important to mention that our nonlinearities   include the ones considered in   \cite{Li} and \cite{Zhang}. However, in our work no explicit form  is assumed on the interaction terms. 

Our main purpose in this work is to establish local and global well-posedness theory in spaces $L^{2}(\R^n)$ and  $H^{1}(\R^n)$;  existence of blow-up solutions; and existence and stability of ground state solutions.

Next we will present our assumptions on the nonlinear terms. We will start our results with the local well-posedness ones. To do so, we will assume the following.
\renewcommand\thetheorem{(H1)}
\begin{theorem}\label{H1}
 \begin{align*}
f_{k}(0,\ldots,0)=0, \qquad  k=1,\ldots,l. 
\end{align*}
\end{theorem}

\renewcommand\thetheorem{(H2)}
\begin{theorem}\label{H2}
 There exists a constant $C>0$ such that for $(z_{1},\ldots,z_{l}),(z_{1}',\ldots,z_{l}')\in \C^{l}$ we have 
\begin{equation*}
\begin{split}
\left|\frac{\partial }{\partial z_{m}}[f_{k}(z_{1},\ldots,z_{l})-f_{k}(z_{1}',\ldots,z_{l}')]\right|&\leq C\sum_{j=1}^{l}|z_{j}-z_{j}'|,\qquad k,m=1,\ldots,l;\\
\left|\frac{\partial }{\partial \overline{z}_{m}}[f_{k}(z_{1},\ldots,z_{l})-f_{k}(z_{1}',\ldots,z_{l}')]\right|&\leq C\sum_{j=1}^{l}|z_{j}-z_{j}'|,\qquad k,m=1,\ldots,l.
\end{split}
\end{equation*}
\end{theorem}

Next, to establish the global well-posedness by using the conservation laws, we assume

\renewcommand\thetheorem{(H3)}
\begin{theorem}\label{H3}
There exist a function $F:\C^{l}\to \C$,  such that    
\begin{equation*}
f_{k}(z_{1},\ldots,z_{l})=\frac{\partial F}{\partial \overline{z}_{k}}(z_{1},\ldots,z_{l})+\overline{\frac{\partial F }{\partial z_{k}}}(z_{1},\ldots,z_{l}),\qquad k=1\ldots,l. 
\end{equation*}
\end{theorem}

\renewcommand\thetheorem{(H4)}
\begin{theorem}\label{H4}
 For any $ \theta \in \R$ and $(z_{1},\ldots,z_{l})\in \mathbb{C}^{l}$,
\begin{equation*}
\mathrm{Re}\,F\left(e^{i\frac{\alpha_{1}}{\gamma_{1}}\theta  }z_{1},\ldots,e^{i\frac{\alpha_{l}}{\gamma_{l}}\theta  }z_{l}\right)=\mathrm{Re}\,F(z_{1},\ldots,z_{l}).
\end{equation*}	
\end{theorem}

\renewcommand\thetheorem{(H5)}
\begin{theorem}\label{H6}
Function $F$ is homogeneous of degree 3, that is, for any $\lambda >0$  and $(z_{1},\ldots,z_{l})\in \mathbb{C}^{l}$,
\begin{equation*}
F(\lambda z_{1},\ldots,\lambda z_{l})=\lambda^{3}F(z_{1},\ldots,z_{l}).
\end{equation*}
\end{theorem}

Finally, to deal with ground states and their stability, we assume the following.

\renewcommand\thetheorem{(H6)}
\begin{theorem}\label{H5}
There holds
\begin{equation*}
\left|\mathrm{Re}\int_{\R^{n}} F(u_{1},\ldots,u_{l})\;dx\right|\leq \int_{\R^{n}} F(|u_{1}|,\ldots,|u_{l}|)\;dx.  
\end{equation*}
\end{theorem}

\renewcommand\thetheorem{(H7)}
\begin{theorem}\label{H7}
Function $F$ is real valued on $\R^l$, that is, if $(y_{1},\ldots,y_{l})\in \R^{l}$ then
\begin{equation*}
F(y_{1},\ldots,y_{l})\in \R. 
\end{equation*}
Moreover, functions	$f_k$ are non-negative on the positive cone in $\mathbb{R}^l$, that is, for $y_i\geq0$, $i=1,\ldots,l$,
\begin{equation*}
f_{k}(y_{1},\ldots,y_{l})\geq0.
\end{equation*}

\end{theorem}


\renewcommand\thetheorem{(H8)}
\begin{theorem}\label{H8}
	Function $F=F_1+\cdots+F_m$, where $F_s$, $s=1,\ldots, m$ is super-modular on $\R^d_+$, $1\leq d\leq l$ and vanishes on hyperplanes, that is, for any $i,j\in\{1,\ldots,d\}$, $i\neq j$ and $k,h>0$, we have
	\begin{equation*}
	F_s(y+he_i+ke_j)+F_s(y)\geq F_s(y+he_i)+F_s(y+ke_j), \qquad y\in \R^d_+,
	\end{equation*}
	and $F_s(y_1,\ldots,y_l)=0$ if $y_j=0$ for some $j\in\{1,\ldots,d\}$.
\end{theorem}

We will discuss how assumptions \ref{H1}-\ref{H8} appear along the paper. By now, we only mention that if $F_s$ is $C^2$ then \ref{H8} is equivalent to
$$
\frac{\partial^2 F_s}{\partial{x_j}\partial{x_i}}\geq0.
$$

Even though we do not need all assumption in all results, throughout the paper we assume that \ref{H1}-\ref{H8} hold. However, in section \ref{sec.conseq}, we will specify which assumptions we are using in the results. It is easy to see that, systems \eqref{system1A} and \eqref{system1B2} satisfy \ref{H1}-\ref{H8} with
\begin{equation*}
F(z_1,z_2,z_3)=\frac{1}{2}\overline{z}_1(z_2^2+z_3^2), \qquad F(z_1,z_2,z_3)=\frac{1}{2}z_1^2\overline{z}_2+z_1z_2\overline{z}_3,
\end{equation*}
respectively.

This paper is organized as follows. In section \ref{sec.conseq} we establish some preliminaries results which are consequences of our conditions \ref{H1}-\ref{H8}.  In section \ref{sec.lgt}, we develop the local and global theories of system \eqref{system1} in the spaces $L^{2}$ and  $H^{1}$. For this purpose we use standard techniques for Schr\"odinger-type equations: the Strichartz estimates combined with the contraction mapping principle is sufficient to obtain the local well-posedness. On the other hand, the global results  are obtained in view of an a priori bound of the local  solution in the spaces of interest. In particular solutions are global for any initial data in the subcritical case. In the critical and supercritical cases the solutions are global under some assumptions on the charge and energy of the initial data.
In section \ref{sec.gs} we are interested in the existence of ground state solutions for the associated elliptic system. This is necessary taking into account we want to obtain a sharp Gagliardo-Nirenberg-type inequality, in which case the best constant depends on such solutions. We establish the existence of ground state by minimizing the so called Weinstein functional in an appropriate set. In section \ref{sec.gsbu} we are interested in the dichotomy global well-posedness versus blow up in finite time. In the critical case we establish a sharp result for the existence of global solutions (depending on the parameters of the system). In the supercritical case, we prove that under some suitable balance between the charge and the energy of the initial data (in terms of that of the ground states) the solutions are also global. This result is also sharp.
 Finally, in section \ref{sec.stinst} we study the nonlinear stability/instability of the ground states. To do so, in the subcritical dimensions, by using the concentration-compactness method developed by Lions, we see that the set of ground states can also be obtained by minimizing the energy under the constraint of constant charge. As a result, the set of ground states are stable in dimensions $n=1,2,3$. On the other hand, in dimensions $4$ and $5$, by using a blow up method, we prove that ground states are unstable.

\section{Preliminaries}\label{sec.conseq}

In this section we introduce some notations and give some consequences of our assumptions. We use $C$ to denote several constants that may vary line-by-line.
Given any set $A$, by  $\mathbf{A}$ we denote  the product  $\displaystyle A\times \cdots \times A $ ($l$ times). In particular, if $A$ is a Banach space with norm $\|\cdot\|$ then $\mathbf{A}$ is also a Banach space with the standard norm given by the sum. Given any complex number $z\in\mathbb{C}$, Re$z$ and Im$z$ represents its real and imaginary parts. Also, $\overline{z}$ denotes its complex conjugate. In $\mathbb{C}^l$ we frequently  write $\mathbf{z}$ and $\mathbf{z}'$ instead of  $(z_{1},\ldots,z_{l})$ and  $(z_{1}',\ldots,z_{l}')$. Given $\mathbf{z}=(z_{1},\ldots,z_{l})\in \mathbb{C}^l$, we write $z_m=x_m+iy_m$ where $x_m$ and $y_m$ are, respectively, the real and imaginary parts of $z_m$. As usual,  the operators $\partial/\partial z_m$ and $\partial/\partial \overline{z}_m$ are defined by
$$
\dfrac{\partial}{\partial z_m}=\frac{1}{2}\left(\frac{\partial}{\partial x_m} -i\frac{\partial}{\partial y_m}\right), \qquad\dfrac{\partial}{\partial \overline{z}_m}=\frac{1}{2}\left(\frac{\partial}{\partial x_m} +i\frac{\partial}{\partial y_m}\right).
$$ 
The spaces
 $L^{p}=L^{p}(\R^{n})$, $1\leq p\leq \infty$, and $W^s_p=W^s_p(\R^{n})$ denotes the usual Lebesgue and Sobolev spaces. In the case $p=2$, we use the standard notation $H^s=W^s_2$. The space  $H^{1}_{rd}=H^{1}_{rd}(\R^{n})$ is the subspace of radially symmetric non-increasing functions in $H^{1}$. 

To simplify notation, if no confusion is caused we use $\int f\, dx$ to denote  $\int_{\R^n} f\, dx$.
Given a time interval $I$, the mixed   spaces $L^p(I;L^q(\R^n))$ are endowed with the norm
$$
\|f\|_{L^p(I;L^q)}=\left(\int_I \left(\int_{\R^n}|f(x,t)|^qdx \right)^{\frac{p}{q}} dt \right)^{\frac{1}{p}},
$$
with the obvious modification if either $p=\infty$ or $q=\infty$. When the interval $I$ is implicit and no confusion will be caused we denote  $L^p(I;L^q(\R^n))$ simply by  $L^p(L^q)$ and its norm by $\|\cdot\|_{L^p(L^q)}$. More generally, if $X$ is a Banach space, $L^{p}(I;X)$ represents the $L^p$ space of $X$-valued functions defined on $I$.

Let us now give some useful consequences of our assumptions. 

 \begin{lem}\label{estinth}
 Let be $\theta \in \R$ and $p>0$.  Suppose $h: \C^{l}\to \C$ satisfies 
\begin{equation}\label{limiderivh}
 \begin{split}
 \left|\frac{\partial h }{\partial z_{m}}(\mathbf{z})\right| +
 \left|\frac{\partial h }{\partial \overline{z}_{m}}(\mathbf{z})\right|&\leq C\sum_{j=1}^{l}|z_{j}|^{p},\quad m=1,\ldots,l.
 \end{split}
 \end{equation}
 Then, for any  $\mathbf{z},\mathbf{z}'\in \mathbb{C}^l$,
 \begin{equation*}
 \left|\int_{0}^{1}\frac{\partial h }{\partial z_{m}}(\mathbf{z}'+\theta (\mathbf{z}-\mathbf{z}'))\; d\theta\right|+\left|\int_{0}^{1}\frac{\partial h }{\partial \overline{z}_{m}}(\mathbf{z}'+\theta (\mathbf{z}-\mathbf{z}'))\; d\theta\right| \leq C\sum_{j=1}^{l}(|z_{j}|^{p}+|z_{j}'|^{p}).
 \end{equation*}
 \end{lem}
 
\begin{proof}
 We will prove the estimate only for the first term. For the second one, it follows similarly.
 Using  \eqref{limiderivh} and triangular inequality we have
 \begin{equation*}
 \begin{split}
\left|\int_{0}^{1}\frac{\partial h}{\partial z_{m}}(\mathbf{z}'+\theta (\mathbf{z}-\mathbf{z}'))\; d\theta\right|
&\leq C \sum_{j=1}^{l} \int_{0}^{1}|z_{j}'+\theta(z_{j}-z_{j}')|^{p}\;d\theta \\
&\leq C\sum_{j=1}^{l} \int_{0}^{1}[|z_{j}'|^{p}+\theta(|z_{j}|^{p}+|z_{j}'|^{p})]\;d\theta\\
&\leq C\sum_{j=1}^{l}(|z_{j}|^{p}+|z_{j}'|^{p}),
 \end{split}
 \end{equation*}
which gives the desired.
 \end{proof}

 \begin{lem}\label{estdifh} Under the assumptions of Lemma \ref{estinth},
 \begin{equation*}
\begin{split}
\left|h(\mathbf{z})-h(\mathbf{z}')\right|&\leq C \sum_{m=1}^{l}\sum_{j=1}^{l}(|z_{j}|^{p}+|z_{j}'|^{p})|z_{m}-z_{m}'|.
\end{split}
\end{equation*}
 \end{lem}
 \begin{proof}
 After applying the chain rule, we have
 \begin{equation*}
 \frac{d}{d\theta}h(\mathbf{z}'+\theta (\mathbf{z}-\mathbf{z}'))=\sum_{m=1}^{l}\frac{\partial h }{\partial z_{m}}(\mathbf{z}'+\theta (\mathbf{z}-\mathbf{z}'))(z_{m}-z_{m}')+\sum_{m=1}^{l}\frac{\partial h}{\partial \overline{z}_{m}}(\mathbf{z}'+\theta (\mathbf{z}-\mathbf{z}'))(\overline{z_{m}-z_{m}'})
 \end{equation*}
Integrating on $[0,1]$  and applying the Fundamental Theorem of  Calculus, we get
 \begin{equation*}
 \begin{split}
 \left|h(\mathbf{z})-h(\mathbf{z}')\right|&\leq \sum_{m=1}^{l}\left\{ \left|\int_{0}^{1}\frac{\partial h}{\partial z_{m}}(\mathbf{z}'+\theta (\mathbf{z}-\mathbf{z}'))\; d\theta\right|\right.\\
 &\quad+\left.\left|\int_{0}^{1}\frac{\partial h}{\partial \overline{z}_{m}}(\mathbf{z}'+\theta (\mathbf{z}-\mathbf{z}'))\; d\theta\right| \right\}|z_{m}-z_{m}'|\\
 &\leq C \sum_{m=1}^{l}\sum_{j=1}^{l}(|z_{j}|^{p}+|z_{j}'|^{p})|z_{m}-z_{m}'|,
 \end{split}
 \end{equation*}
 where we have used Lemma \ref{estinth} in the second inequality.
 \end{proof}

\begin{coro}\label{limfk}
If \textnormal{\ref{H1}} and \textnormal{\ref{H2}}  hold, then
 \begin{equation}\label{estdiffkeq}
\begin{split}
\left|f_{k}(\mathbf{z})-f_{k}(\mathbf{z}')\right|&\leq  C \sum_{m=1}^{l}\sum_{j=1}^{l}(|z_{j}|+|z_{j}'|)|z_{m}-z_{m}'|, \qquad k=1\ldots, l.
\end{split}
\end{equation}
and
\begin{equation*}
\left|f_{k}(\mathbf{z})\right|\leq C\sum_{j=1}^{l}|z_{j}|^{2}, \qquad k=1\ldots, l.
\end{equation*}
\end{coro}
 \begin{proof}
 	Inequality \eqref{estdiffkeq} follows immediately from Lemma \ref{estdifh} with $p=1$. For the second part, it suffices to take
 $\mathbf{z}'=0$ in \eqref{estdiffkeq} and apply Young's inequality.
 \end{proof}

 Note that Corollary \ref{limfk} gives us that our nonlinearities has indeed quadratic growth.
 
 \begin{lem}\label{estgraddiffk}
Let assumptions \textnormal{\ref{H1}} and \textnormal{\ref{H2}} hold. Let $\mathbf{u}$ and $\mathbf{u}'$ be complex-valued functions defined on $\R^n$. Then,
 \begin{equation*}
 \begin{split}
 |\nabla[f_{k}(\mathbf{u})-f_{k}(\mathbf{u}'
 )]|&\leq  C\sum_{m=1}^{l}\sum_{j=1}^{l}|u_{j}||\nabla(  u_{m}-u_{m}')|+C\sum_{m=1}^{l}\sum_{j=1}^{l}|u_{j}-u_{j}'||\nabla u_{m}'|.
 \end{split}
 \end{equation*}
\end{lem}

\begin{proof}
 Writing  $x=(x_{1},\ldots,x_{n})\in\R^n$ and recalling the chain rule
\begin{equation*}
\frac{\partial}{\partial x_{j}}f_{k}(\mathbf{u}(x))=\sum_{m=1}^{l}\left(\frac{\partial f_k}{\partial z_{m}}\frac{\partial u_{m}}{\partial x_{j}}+\frac{\partial f_k}{\partial \overline{z}_{m}}\frac{\partial \overline{u}_{m}}{\partial x_{j}}\right)
\end{equation*}
we obtain
\begin{equation*}
\nabla f_{k}(\mathbf{u})=\sum_{m=1}^{l}\frac{\partial f_k}{\partial z_{m}}\nabla  u_{m}+\sum_{m=1}^{l}\frac{\partial f_k}{\partial \overline{z}_{m}}\nabla \overline{u}_{m}.
\end{equation*}
 Thus,
 \begin{equation}\label{graddiffk}
 \begin{split}
 \nabla[f_{k}(\mathbf{u})-f_{k}(\mathbf{u}'
 )]&=\sum_{m=1}^{l}\frac{\partial f_{k} }{\partial z_{m}}(\mathbf{u})\nabla  u_{m}+\sum_{m=1}^{l}\frac{\partial f_{k} }{\partial \overline{z}_{m}}(\mathbf{u})\nabla \overline{u}_{m}\\ 
&\quad-\sum_{m=1}^{l}\frac{\partial f_{k}}{\partial z_{m}}(\mathbf{u}')\nabla  u_{m}'-\sum_{m=1}^{l}\frac{\partial f_{k} }{\partial \overline{z}_{m}}(\mathbf{u}')\nabla \overline{u}_{m}'\\
&=\sum_{m=1}^{l}\frac{\partial f_{k} }{\partial z_{m}}(\mathbf{u})\nabla(  u_{m}-u_{m}')+\sum_{m=1}^{l}\frac{\partial f_{k} }{\partial \overline{z}_{m}}(\mathbf{u})\nabla( \overline{ u_{m}-u_{m}'})\\
&\quad+\sum_{m=1}^{l}\frac{\partial }{\partial z_{m}}[f_{k}(\mathbf{u})-f_{k}(\mathbf{u}')]\nabla u_{m}'+\sum_{m=1}^{l}\frac{\partial }{\partial \overline{z}_{m}}[f_{k}(\mathbf{u})-f_{k}(\mathbf{u}')]\nabla \overline{u}_{m}'.
 \end{split}
 \end{equation}
 Taking into account \ref{H1} and  \ref{H2} we have, for the first and third terms in \eqref{graddiffk},
  \begin{equation*}
 \begin{split}
\left |\sum_{m=1}^{l}\frac{\partial f_{k} }{\partial z_{m}}(\mathbf{u})\nabla(  u_{m}-u_{m}')\right|&\leq C\sum_{m=1}^{l}\left(\sum_{j=1}^{l}|u_{j}|\right)|\nabla(  u_{m}-u_{m}')|= C\sum_{m=1}^{l}\sum_{j=1}^{l}|u_{j}||\nabla(  u_{m}-u_{m}')|
 \end{split}
 \end{equation*}
and
 \begin{equation*}
 \begin{split}
\left |\sum_{m=1}^{l}\frac{\partial }{\partial z_{m}}[f_{k}(\mathbf{u})-f_{k}(\mathbf{u}')]\nabla u_{m}'\right|&\leq C\sum_{m=1}^{l}\left(\sum_{j=1}^{l}|u_{j}-u_{j}'|\right)|\nabla  u_{m}'|= C\sum_{m=1}^{l}\sum_{j=1}^{l}|u_{j}-u_{j}'||\nabla u_{m}'|.
 \end{split}
 \end{equation*}
We obtain similar bounds for the second and fourth terms, which establishes the desired.
\end{proof}

 The next lemma says how we can estimate the gradient of the nonlinearities, $f_{k}$,  in $L^{p}$-spaces.
\begin{lem} \label{gradfkLr}
 Let $1\leq p,q,r\leq \infty$ such that $\frac{1}{r}=\frac{1}{p}+\frac{1}{q}$. Assume that $\mathbf{u}\in \mathbf{L}^{p}(\R^{n}) $ and $\nabla \mathbf{u}\in \mathbf{L}^{q}(\R^{n}) $. Then, for $k=1,\ldots,l$,
 \begin{equation*}
 \|\nabla f_{k}(\mathbf{u})\|_{\mathbf{L}^{r}}\leq C\|\mathbf{u}\|_{\mathbf{L}^{p}}\|\nabla \mathbf{u}\|_{\mathbf{L}^{q}}.
 \end{equation*}
\end{lem}
\begin{proof}
First note that from Lemma \ref{estgraddiffk} (with $\mathbf{u}'=0$) we have
\begin{equation*}
|\nabla f_{k}(\mathbf{u})|\leq  C\sum_{j=1}^{l}\sum_{m=1}^{l}|u_{j}||\nabla u_{m}|,
\end{equation*}
which combined with  H\"older's inequality yields
\begin{equation*}
\begin{split}
    \int|\nabla f_{k}(\mathbf{u})|^{r}\;dx&\leq  C\sum_{j=1}^{l}\sum_{m=1}^{l} \int|u_{j}|^{r}|\nabla u_{m}|^{r}\;dx\\
    &\leq  C\sum_{j=1}^{l}\sum_{m=1}^{l}\|u_{j}\|_{L^{p}}^{r}\|\nabla u_{m}\|_{L^{q}}^{r}<\infty,
\end{split}
\end{equation*}
completing thus the proof of the lemma. 
\end{proof}

 For our next result we start with the following definition. 
\begin{defi}\label{defGC}
We say that functions $f_{k}$ satisfy the \textbf{Gauge condition} if for any $\theta \in \R$,
 \begin{equation}\label{gaugeCon}
f_{k}\left(e^{i\frac{\alpha_{1}}{\gamma_{1}}\theta }z_{1},\ldots,e^{i\frac{\alpha_{l}}{\gamma_{l}}\theta }z_{l}\right)=e^{i\frac{\alpha_{k}}{\gamma_{k}}\theta }f_{k}(z_{1},\ldots,z_{l}),\qquad k=1,\ldots, l.\tag{GC}
\end{equation}
\end{defi}

\begin{obs}\label{fksegexp}
Note that, from the definition of operators  $\partial/\partial z_k$ and $\partial/\partial \overline{z}_k$,  assumption \textnormal{\ref{H3}} can be rewritten as
\begin{equation}\label{Cond3II}
\begin{split}
f_{k}(\mathbf{z})&=\frac{\partial F}{\partial \overline{z}_{k}}(\mathbf{z})+\overline{\frac{\partial F }{\partial z_{k}}}(\mathbf{z})
=\frac{\partial F}{\partial \overline{z}_{k}}(\mathbf{z})+\frac{\partial \overline{F}}{\partial \overline{z}_{k}}(\mathbf{z})
=2\frac{\partial  }{\partial \overline{z}_{k}}\mathrm{Re} \,F(\mathbf{z}).
\end{split}
\end{equation}
\end{obs}

\begin{lem}\label{H34impGC}
Assume that \textnormal{\ref{H3}} and \textnormal{\ref{H4}} hold. Then $f_k, k=1,\ldots, l$, satisfy the Gauge condition \eqref{gaugeCon}.
\end{lem}
\begin{proof}
By setting	$w_{m}:=e^{i\frac{\alpha_{m}}{\gamma_{m}}\theta}z_{m}$, from \ref{H4} we obtain
\begin{equation}\label{deriReF}
\mathrm{Re}\,F\left(\mathbf{w}\right)=\mathrm{Re}\,F(\mathbf{z}).
\end{equation}
Since the functions $w_m$ are holomorphic we have $\partial w_m/\partial \overline{z}_m=0$. Hence, from
\eqref{deriReF} and the chain rule,
 \begin{equation}\label{dev0}
 \frac{\partial \mbox{Re} F}{\partial \overline{z}_{k}}(\mathbf{z})=\frac{\partial \mbox{Re}F}{\partial \overline{w}_{k}}(\mathbf{w})e^{-i\frac{\alpha_{k}}{\gamma_{k}}\theta  }.
 \end{equation}
In view of  \eqref{dev0} and Remark \ref{fksegexp},
\begin{equation*}
\begin{split}
f_{k}(\mathbf{z})=2\frac{\partial \mbox{Re}F}{\partial \overline{w}_{k}}(\mathbf{w})e^{-i\frac{\alpha_{k}}{\gamma_{k}}\theta  }
=e^{-i\frac{\alpha_{k}}{\gamma_{k}}\theta  }f_{k}(\mathbf{w})
=e^{-i\frac{\alpha_{k}}{\gamma_{k}}\theta  }f_{k}\left(e^{i\frac{\alpha_{1}}{\gamma_{1}}\theta  }z_{1},\ldots,e^{i\frac{\alpha_{l}}{\gamma_{l}}\theta  }z_{l}\right),
\end{split}
\end{equation*}
which completes the proof.
\end{proof}

\begin{lem}\label{ReFinvari}
 Assume that  \textsc{\ref{H3}} and \textsc{\ref{H4}} hold. Then,
there exist positive constants $\sigma_{1},\ldots,\sigma_{l}$ such that, for any $\mathbf{z}\in \mathbb{C}^{l}$,
\begin{equation*}
\mathrm{Im}\sum_{k=1}^{l}\sigma_{k}f_{k}(\mathbf{z})\overline{z}_{k}=0.
\end{equation*}
 \end{lem}
 \begin{proof}
 Denote by  $\mathbf{w}$ the vector $(w_{1},\ldots,w_{l}):=\left(e^{i\frac{\alpha_{1}}{\gamma_{1}}\theta  }z_{1},\ldots,e^{i\frac{\alpha_{l}}{\gamma_{l}}\theta  }z_{l}\right)$. By Lemma \ref{H34impGC}, the nonlinearities $f_{k}$  satisfy the Gauge condition \eqref{gaugeCon}. Then
 \begin{equation}\label{invfkw}
\begin{split}
 f_{k}(\mathbf{w})\overline{w}_{k}=f_{k}\left(e^{i\frac{\alpha_{1}}{\gamma_{1}}\theta  }z_{1},\ldots,e^{i\frac{\alpha_{l}}{\gamma_{l}}\theta  }z_{l}\right)e^{-i\frac{\alpha_{k}}{\gamma_{k}}\theta  }\overline{z_{k}}
 =f_{k}(\mathbf{z})\overline{z}_k.
\end{split}
 \end{equation}
 Define $h(\theta):=F\left(\mathbf{w}\right)$. By the chain rule,
 \begin{equation}
 \begin{split}\label{derh}
 \frac{dh}{d\theta}&=\sum_{k=1}^{l}\frac{\partial F}{\partial w_{k}}(\mathbf{w})\frac{\partial w_{k} }{\partial \theta}+\sum_{k=1}^{l}\frac{\partial F }{\partial \overline{w}_{k}}(\mathbf{w})\frac{\partial \overline{w}_{k} }{\partial \theta}\\
 &=\sum_{k=1}^{l}\frac{\partial F}{\partial w_{k}}(\mathbf{w})\left(\frac{\alpha_{k}}{\gamma_{k}} i\right)e^{i\frac{\alpha_{k}}{\gamma_{k}}\theta  }z_{k}+\sum_{k=1}^{l}\frac{\partial F }{\partial \overline{w}_{k}}(\mathbf{w})\left(-\frac{\alpha_{k}}{\gamma_{k}} i\right)\overline{e^{i\frac{\alpha_{k}}{\gamma_{k}}\theta  }z_{k}}\\
 &=\sum_{k=1}^{l}\overline{\overline{\frac{\partial F}{\partial w_{k}}}(\mathbf{w})\left(-\frac{\alpha_{k}}{\gamma_{k}} i\right)\overline{w}_{k}}+\sum_{k=1}^{l}\frac{\partial F }{\partial \overline{w}_{k}}(\mathbf{w})\left(-\frac{\alpha_{k}}{\gamma_{k}} i\right)\overline{w}_{k}.
 \end{split}
 \end{equation}
Taking the real part on both sides  of \eqref{derh}, in view of Remark \ref{fksegexp} and \eqref{invfkw} we obtain
\begin{equation*}
\begin{split}
\mathrm{Re}\frac{dh}{d\theta}
&=\mathrm{Im}\sum_{k=1}^{l}\left(\frac{\alpha_{k}}{\gamma_{k}} \right)f_{k}(\mathbf{z})\overline{z_{k}}.
\end{split}
\end{equation*}
On the other hand, taking the derivative with respect to $\theta$ on both sides of \textsc{\ref{H4}}, we have that  $\displaystyle \mbox{Re}\frac{dh}{d\theta}=0$; thus the conclusion follows by taking $\displaystyle\sigma_{k}=\frac{\alpha_{k}}{\gamma_{k}}$, for $k=1,\ldots,l$.
 \end{proof}

\begin{lem}\label{estdifF} Assume that  \textnormal{\ref{H1}-\ref{H3}} and \textnormal{\ref{H6}} hold. Then,
 \begin{equation}\label{estdifFeq}
\begin{split}
\left|\mathrm{Re}\,F(\mathbf{z})-\mathrm{Re}\,F(\mathbf{z}')\right|&\leq  C \sum_{m=1}^{l}\sum_{j=1}^{l}(|z_{j}|^{2}+|z_{j}'|^{2})|z_{m}-z_{m}'|,
\end{split}
\end{equation}
and
\begin{equation}\label{estdifFeq1}
|\mathrm{Re}\,F(\mathbf{z})|\leq C \sum_{j=1}^{l}|z_{j}|^{3}, 
\end{equation}
 \end{lem}
\begin{proof}
 Since $\displaystyle\overline{\frac{\partial \mbox{Re}F}{\partial z_{k}}} =\frac{\partial \overline{\mbox{Re}F}}{\partial \overline{z}_{k}}=\frac{\partial \mbox{Re}F }{\partial \overline{z}_{k}}$, Corollary \ref{limfk} and    \eqref{Cond3II} lead to
\begin{equation*}
\begin{split}
\left|\frac{\partial \mbox{Re}F }{\partial z_{m}}(\mathbf{z})\right|+\left|\frac{\partial \mbox{Re}F }{\partial \overline{z}_{m}}(\mathbf{z})\right|\leq C\sum_{j=1}^{l}|z_{j}|^{2},\qquad m=1,\ldots,l.
\end{split}
\end{equation*}
Thus,  \eqref{estdifFeq} follows from Lemma \ref{estdifh} with $p=2$. In addition, from \ref{H6} we have $F(\mathbf{0})=0$. So, \eqref{estdifFeq1} follows from \eqref{estdifFeq} and Young's inequality.
 \end{proof}

The next lemma is usefully to construct  Virial-type identities.

\begin{lem}\label{propertiesF}
Assume that \textnormal{\ref{H3}} holds and let $\mathbf{u}$ be a complex-valued function defined on $\R^n$. Then, 
\begin{enumerate}
\item[(i)]
\begin{equation*}
\mathrm{Re}\sum_{k=1}^{l}f_{k}(\mathbf{u})\nabla \overline{u}_{k}=\mathrm{Re}[\nabla F(\mathbf{u})].
\end{equation*}
\item[(ii)] In addition, if assumption \textnormal{\ref{H6}} holds, then
\begin{equation*}
\mathrm{Re}\sum_{k=1}^{l}f_{k}(\mathbf{u})\overline{u}_{k}=\mathrm{Re}[3F(\mathbf{u})].
\end{equation*}
\end{enumerate}
\end{lem}
\begin{proof}
By differentiating $F$ with respect to $x_{j}$ and using the chain rule we obtain
\begin{equation*}
\nabla F(\mathbf{u})=\sum_{k=1}^{l}\frac{\partial F }{\partial z_{k}}(\mathbf{u})\nabla u_{k}+\sum_{k=1}^{l}\frac{\partial F }{\partial \overline{z}_{k} }(\mathbf{u})\nabla \overline{u}_{k} .
\end{equation*}
Taking the real part on both side and using \ref{H3} (or Remark \ref{fksegexp}) we get part (i).

For  (ii) we differentiate  both sides  of \ref{H6} with respect to $\lambda$ and  evaluate at $\lambda=1$ to deduce that
\begin{equation*}
\sum_{k=1}^{l}\frac{\partial F }{\partial z_{k}}(\mathbf{u})u_{k}+\sum_{k=1}^{l}\frac{\partial F}{\partial \overline{z}_{k}}(\mathbf{u})\overline{u}_{k}=3F(\mathbf{u}).
\end{equation*}
Now taking the real part and using \ref{H3} the proof is completed.
\end{proof}

 The next result is a natural consequence of  \ref{H6}. Since $F$ is homogeneous  of degree 3 its derivative is  homogeneous  of degree 2, which means that the nonlinearities $f_{k}$ inherit this property.

\begin{lem}\label{fkhomog2}
 Assumptions \textnormal{\ref{H3}} and \textnormal{\ref{H6}} imply that the nonlinearities $f_{k}$, $k=1,\ldots,l$ are homogeneous  functions of degree 2. 
\end{lem}
\begin{proof}
It suffices to take the derivative on both sides of \ref{H6} and use \ref{H3}.
\end{proof}

\begin{lem}\label{fkreal}
	If  $F$ satisfies \textsc{\ref{H7}} then 
	\begin{equation*}\label{Cond4.2}
	f_{k}(\mathbf{x})=\frac{\partial F}{\partial x_{k}}(\mathbf{x}), \qquad \forall \mathbf{x}\in \R^{l}.
	\end{equation*}
	In addition,   $F$ is positive on the positive cone of $\R^l$.
\end{lem}
\begin{proof}
	The first part is clear from Remark \ref{fksegexp}. For the second part, we use Lemma \ref{propertiesF} (ii).
\end{proof}

We finish this section with a regularity lemma which imply that our system make sense when we consider  the $H^{-1}-H^{1}$ duality product. First we recall the following result  

 \begin{lem}[Sobolev multiplication law]\label{Hmen1prod}
 Let $n\geq 1$. Assume that $s,s_1,s_2$ are real numbers satisfying either
 \begin{enumerate}
 \item[(i)]  $s_{1}+s_{2}\geq 0$, $s_{1},s_{2}\geq s$ and $s_{1}+s_{2}> s+n/2$; or
 \item[(ii)]  $s_{1}+s_{2}> 0$, $s_{1},s_{2}>s$ and $s_{1}+s_{2}\geq  s+n/2$.
 \end{enumerate}
Then there is a continuous multiplication map
 \begin{equation*}
 H^{s_{1}}(\R^{n})\times H^{s_{2}}(\R^{n}) \to H^{s}(\R^{n}),
 \end{equation*}
 taking
 \begin{equation*}
( u,v)\to uv, 
 \end{equation*}
  and satisfying the estimate
 \begin{equation*}
 \|uv\|_{H^{s}(\R^{n})}\leq  C\|u\|_{H^{s_{1}}(\R^{n})}\|v\|_{H^{s_{2}}(\R^{n})}.
 \end{equation*}
 \end{lem}
 \begin{proof}
 See \cite[Corollary 3.16]{Tao}. 
 \end{proof}

 \begin{lem}\label{lemfkcontH}
 Let $1\leq n\leq 6$. Assume that the nonlinearities $f_{k}$ satisfy   \textnormal{\ref{H1}} and \textnormal{\ref{H2}}.  Then, for all  $k=1,\ldots,l$
 we have $f_{k}\in \mathcal{C}(\mathbf{H}^{1}(\R^{n}),H^{-1}(\R^{n}))$.
 \end{lem}
 \begin{proof} 
 Let  $(\mathbf{u}_{i})\subset \mathbf{H}^{1}(\R^{n})$ be  such that $\mathbf{u}_{i} \to \mathbf{u}$ in $\mathbf{H}^{1}(\R^{n})$. In particular, there exist $M>0$ such that $\|\mathbf{u}_{i}\|_{\mathbf{H}^{1}}\leq M$. 
Corollary \ref{limfk} and part ii) in Lemma \ref{Hmen1prod}, with $s_{1}=s_{2}=1$ and $s=-1$,  lead to
 \begin{equation*}
 \begin{split}
 \|f_{k}(\mathbf{u}_{i})&-f_{k}(\mathbf{u})\|_{H^{-1}}\leq C\sum_{m=1}^{l}\sum_{j=1}^{l}\|(|u_{ji}|+|u_{j}|)|u_{mi}-u_{m}|\|_{H^{-1}}\\
  &\leq C\sum_{m=1}^{l}\sum_{j=1}^{l}\|u_{ji}\|_{H^{1}}\|u_{mi}-u_{m}\|_{H^{1}}
  +C\sum_{m=1}^{l}\sum_{j=1}^{l}\|u_{j}\|_{H^{1}}\|u_{mi}-u_{m}\|_{H^{1}}\\
 &\leq CM\sum_{m=1}^{l}\|u_{mi}-u_{m}\|_{H^{1}}
 +C\sum_{m=1}^{l}\sum_{j=1}^{l}\|u_{j}\|_{H^{1}}\|u_{mi}-u_{m}\|_{H^{1}}.
 \end{split}
 \end{equation*}
 We note that the right-hand side goes to $0$ as $i\to \infty$, which implies that $f_{k}(\mathbf{u}_{i})\to f_{k}(\mathbf{u})$ in  $H^{-1}(\R^{n})$.
 \end{proof}

 \section{Local and global well-posedness in $L^{2}$ and $H^{1}$}\label{sec.lgt}

In this section we will study the dynamics of system \eqref{system1} in  $L^{2}$  and   $H^{1}$ frameworks. Since $f_{k}$ are homogeneous functions of degree 2 (see Lemma \ref{fkhomog2}), by using the scaling
 \begin{equation*}
     u_{k}^{\lambda}(x,t)=\lambda^{2}u_{k}( \lambda x,\lambda^{2} t), \qquad k=1,\ldots,l,
 \end{equation*}
 we can see that the Sobolev space $\dot{H}^{n/2-2}$ is critical for system \eqref{system1} (with $\beta_{k}=0$) in the sense that it is invariant by the above scaling. In particular, $L^{2}$ and $\dot{H}^{1}$ are critical for dimensions $n=4$ and $n=6$, respectively. More precisely, we adopt the following regimes: we will say that system \eqref{system1} is
 \begin{equation*}
     L^{2}-
     \begin{cases}
     \mbox{subcritical}, \quad\mbox{if}\quad 1\leq n\leq 3,\\
     \mbox{critical},\quad\mbox{if}\quad n=4,\\
     \mbox{supercritical},\quad\mbox{if}\quad n\geq5;
     \end{cases}\quad \mbox{and} \quad
     H^{1}-
     \begin{cases}
     \mbox{subcritical},\quad\mbox{if}\quad 1\leq n\leq 5\\
     \mbox{critical},\quad\mbox{if}\quad n=6\\
     \mbox{supercritical},\quad\mbox{if}\quad n\geq 7.
     \end{cases}
 \end{equation*}

We are primarily interested in studying the local and global well-posedness for the Cauchy problem \eqref{system1} in the spaces $L^{2}$ and $\dot{H}^{1}$ in subcritical and critical regimes. For that, we will consider the associated system of integral equations
\begin{equation}\label{system2}
\begin{cases}
 u_{k}(t)= \displaystyle U_{k}(t)u_{k0}+i\int_{0}^{t}U_{k}(t-t') \frac{1}{\alpha_{k}}f_{k}(u_{1},\ldots,u_{l})\;dt',\\
(u_{1}(x,0),\ldots,u_{l}(x,0))=(u_{10},\ldots,u_{l0}),
\end{cases}
\end{equation}
where $U_{k}(t)$ is the Schr\"odinger evolution group defined by $\displaystyle U_{k}(t)=e^{i\frac{t}{\alpha_{k}}(\gamma_{k}\Delta-\beta_{k})}$, $ k=1\ldots, l.$

\subsection{Local existence of $L^2$-solutions}
This section is devoted to study the existence of local $L^2$ solutions in the subcritical and critical regimes, that is, we study  (\ref{system2}) in $L^2$ with $1\leq n\leq 4$. The results here follow close the ones in \cite{Hayashi}. For any $u_{10},\ldots,u_{l0}\in L^2$ we solve (\ref{system2}) in the  spaces
\begin{equation*}
X(I)=\begin{cases}
(\mathcal{C}\cap L^{\infty})(I; L^2)\cap L^{12/n}(I;L^{3}),\qquad 1\leq n\leq 3;\\
(\mathcal{C}\cap L^{\infty})(I; L^2)\cap L^{2}(I;L^{2n/(n-2)}), \qquad n\geq 4,
\end{cases}
\end{equation*}
for some time interval $I=[-T,T]$ with $T>0$. The norm in $X(I)$ is defined as
\begin{equation*}
\|f\|_{X(I)}=\begin{cases}
\|f\|_{L^{\infty}( L^2)}+ \|f\|_{L^{12/n}( L^{3})},\qquad 1\leq n\leq 3;\\
\|f\|_{L^{\infty}( L^2)}+ \|f\|_{L^{2}( L^{2n/(n-2)})}, \qquad n\geq 4.
\end{cases}
\end{equation*}

H\"{o}lder's inequality in space and time variables allow us to prove the following lemma.
 \begin{lem}\label{estfg} Let $1\leq n\leq 4$ and $f, g\in X(I)$. 
\begin{enumerate}
\item[(i)] If $1\leq n\leq 3$, then
\begin{equation*}
\|fg\|_{L^{12/(12-n)}(L^{3/2})}\leq C T^{\frac{4-n}{4}}\|f\|_{X(I)}\|g\|_{X(I)}.
\end{equation*}
\item[(ii)] If $n=4$, then 
\begin{equation*}
\|fg\|_{L^{1}(L^2)}\leq C\|f\|_{L^{2}(L^4)}\|g\|_{L^{2}(L^4)}.
\end{equation*}
\end{enumerate}
\end{lem}

Next we recall that a pair $(q,r)$ is called \textit{admissible} if
$$
\frac{2}{q}=n\left(\frac{1}{2}-\frac{1}{r}\right)
$$
and
$$
2\leq r\leq \frac{2n}{n-2} \quad (2\leq r\leq \infty\;\; \mbox{if}\;\; n=1\;\; \mbox{and} \;\; 2\leq r<\infty\;\; \mbox{if}\;\; n=2).
$$
In particular the pair $(\infty,2)$ is always admissible.  Note also that the pair $(12/n,3)$ is admissible if $1\leq n\leq 3$ and $(2,4)$ is admissible if $n=4$. In the following we will use the well known Strichartz inequalities (see, for instance, Theorem 2.3.3 in \cite{Cazenave}).

\begin{pro}[Strichartz's inequality]\label{stricha}
Let $(q_1,r_1)$ and $(q_2,r_2)$ be two admissible pairs and $I=[-T,T]$ for some $T>0$. Then, for $k=1,\ldots,l$,
$$
\|U_k(t)f\|_{L^{q_1}(\R;L^{r_1})}\leq C\|f\|_{L^2}
$$
and
$$
\left\| \int_0^tU_k(t-s)f(\cdot,s)ds \right\|_{L^{q_1}(I;L^{r_1})} \leq C\|f\|_{L^{q_2'}(I;L^{r_2'})},
$$
where $q_2'$ and $r_2'$ are the H\"older conjugate of $q_2$ and $r_2$, respectively.
\end{pro}

 A combination of the last two results gives us the following. 
 
 \begin{lem}\label{estinormdiffk}
 Let $1\leq n\leq 4$ and suppose $\mathbf{u},\mathbf{u}'\in \mathbf{X}(I)$. Assume that \textnormal{\ref{H1}} and \textnormal{\ref{H2}} hold.
 \begin{enumerate}
 \item[(i)] If $1\leq n\leq 3$, then
 \begin{equation*}
\begin{split}
\left\|\int_{0}^{t}U_{k}(t-t')\frac{1}{\alpha_{k}}[ f_{k}(\mathbf{u})-f_{k}(\mathbf{u}')]\;dt'\right\|_{X(I)}&\leq CT^{\frac{4-n}{4}}\left(\|\mathbf{u}\|_{\mathbf{X}(I)}+\|\mathbf{u}'\|_{\mathbf{X}(I)}\right)\|\mathbf{u}-\mathbf{u}'\|_{\mathbf{X}(I)}.
\end{split}
\end{equation*}
\item[(ii)] If $n=4$, then
 \begin{equation*}
\begin{split}
\left\|\int_{0}^{t}U_{k}(t-t')\frac{1}{\alpha_{k}}[ f_{k}(\mathbf{u})-f_{k}(\mathbf{u}')]\;dt'\right\|_{X(I)}&\leq C\left(\|\mathbf{u}\|_{\mathbf{{L}}^{2}(\mathbf{L}^{4})}+\|\mathbf{u}'\|_{\mathbf{{L}}^{2}(\mathbf{L}^{4})}\right)\|\mathbf{u}-\mathbf{u}'\|_{\mathbf{{L}}^{2}(\mathbf{L}^{4})}.
\end{split}
\end{equation*}
 \end{enumerate}
 \end{lem}
 \begin{proof}
 In view of Proposition \ref{stricha}, Corollary \ref{limfk} and  Lemma \ref{estfg} (i) we get
  \begin{equation*}
  \begin{split}
 \bigg\|\int_{0}^{t}U_{k}(t-t')&\frac{1}{\alpha_{k}}[ f_{k}(\mathbf{u})-f_{k}(\mathbf{u}')]\;dt'\bigg\|_{X(I)} \leq C\left\| f_{k}(\mathbf{u})-f_{k}(\mathbf{u}')]\right\|_{L^{12/(12-n)}(L^{3/2})}\\
 &\leq C\sum_{j=1}^{l}\sum_{m=1}^{l}\|(|u_{j}|+|u_{j}'|)|u_{m}-u_{m}'|\|_{L^{12/(12-n)}(L^{3/2})}\\
 &\leq CT^{\frac{4-n}{4}}\sum_{j=1}^{l}\sum_{m=1}^{l}(\|u_{j}\|_{X(I)}+\|u_{j}'\|_{X(I)})\|u_{m}-u_{m}'\|_{X(I)}\\
& \leq CT^{\frac{4-n}{4}} \left(\|\mathbf{u}\|_{\mathbf{X}(I)}+\|\mathbf{u}'\|_{\mathbf{X}(I)}\right)\|\mathbf{u}-\mathbf{u}'\|_{\mathbf{X}(I)}.
  \end{split}
  \end{equation*}
For $n=4$, the proof follows similar steps, taking into account   Lemma \ref{estfg} (ii).
 \end{proof}
 
 Now we are able to prove the existence of local solutions.
 
 \begin{teore}[Existence of local $L^2$-solutions: subcritical case]\label{localexistenceL2} Let $1\leq n\leq 3$. Assume that   \textnormal{\ref{H1}} and \textnormal{\ref{H2}} hold. Then for any $r>0$ there exists $T=T(r)>0$ such that for any $\mathbf{u}_0\in \mathbf{L}^2 $ with $\|\mathbf{u}_0\|_{\mathbf{L}^2}\leq r$, system \eqref{system1}
 has a unique   solution $\mathbf{u}\in \mathbf{X}(I)$ with $I=[-T,T]$.
\end{teore}

\begin{proof}
The proof relies on the contraction mapping principle.  Define the operator
$$\Gamma(\mathbf{u})=(\Phi_{1}(\mathbf{u}),\ldots,\Phi_{l}(\mathbf{u})),$$
where 
\begin{equation*}
\Phi_{k}(\mathbf{u})(t)=\displaystyle U_{k}(t)u_{k0}+i\int_{0}^{t}U_{k}(t-t') \frac{1}{\alpha_{k}}f_{k}(\mathbf{u})\;dt',\qquad k=1,\ldots,l.
\end{equation*}
For some $T>0$ to be determined later, introduce the ball of radius $a$:
	$$B(T,a)=\left\{\mathbf{u}\in \mathbf{X}(I):\|\mathbf{u}\|_{\mathbf{X}(I)}:=\sum_{j=1}^{l}\|u_{j}\|_{X(I)}\leq a \right\}.$$
Using Strichartz estimates and Lemma \ref{estinormdiffk} (with $\mathbf{u}'=0$) we get 
\begin{equation*}
\begin{split}
\|\Gamma(\mathbf{u})\|_{\mathbf{X}(I)}
&\leq C\|\mathbf{u}_0\|_{\mathbf{L}^{2}}+CT^{1-n/4} \|\mathbf{u}\|_{\mathbf{X}(I)}^{2}.
\end{split}
\end{equation*}
Let us choose $a=2Cr$. Thus, if $\mathbf{u}\in B(T,a)$,  
 $$\|\Gamma(\mathbf{u})\|_{\mathbf{X}(I)}\leq \frac{a}{2}+CT^{1-n/4}a^2=\left(\frac{1}{2}+CT^{1-n/4}a\right)a.$$
So,   fixing $T>0$ such that 
$ CT^{1-n/4}a< 1/2$ (which means that $T=T(r)\approx r^{\frac{4}{n-4}}$)
we have $\|\Gamma(\mathbf{u})\|_{\mathbf{X}(I)}\leq a$. Therefore $\Gamma:B(T,a)\to B(T,a)$ is well defined. Moreover, similar arguments
show that $\Gamma$ is a contraction. The result then follows from the contraction mapping principle.
\end{proof}

\begin{teore}[Existence of local $L^2$-solutions: critical case]\label{locexistL2n=4}
    Let $n=4$. Assume that  \textnormal{\ref{H1}} and \textnormal{\ref{H2}} hold. Then for any $\mathbf{u}_0\in \mathbf{L}^2$, there exists $T(\mathbf{u}_0)>0$ (depending on $\mathbf{u}_0$) such that  system \eqref{system1}
 has a unique  solution $\mathbf{u}\in\mathbf{X}(I)$ with $I=[-T(\mathbf{u}_0),T(\mathbf{u}_0)]$.
\end{teore}
 \begin{proof}
 We apply the contraction mapping principle again.
 From Proposition \ref{stricha} we have $(U_{1}(\cdot)u_{10},\ldots,U_{l}(\cdot)u_{l0})\in \mathbf{{L}}^{2}(\R;\mathbf{L}^4)$. Therefore, given any $\epsilon>0$ we can choose $T=T(\mathbf{u}_0)>0$ such that 
 $$
 \|(U_{1}(\cdot)u_{10},\ldots,U_{l}(\cdot)u_{l0})\|_{\mathbf{{L}}^{2}(I;\mathbf{L}^4)}\leq\epsilon,
 $$
where $I=[-T,T]$. Let $\Gamma$ be defined as in the proof of Theorem \ref{localexistenceL2} and define the ball	
$$
\tilde{B}(T,a)=\left\{\mathbf{u}\in \mathbf{{L}}^{2}(I;\mathbf{L}^4):\|\mathbf{u}\|_{ \mathbf{{L}}^{2}(I;\mathbf{L}^4)}:=\sum_{j=1}^{l}\|u_{j}\|_{L^{2}(I;L^{4})}\leq a \right\}.
$$
If $\mathbf{u}\in \tilde{B}(T,a)$, we have from Lemma \ref{estinormdiffk}
\begin{equation*}
\begin{split}
\sum_{k=1}^{l}\|\Phi_{k}(\mathbf{u})\|_{L^{2}(I;L^{4})} 
&\leq  lC\epsilon+lCa^{2}.
\end{split} 
\end{equation*}
Taking $a$ such that $a=2Cl\epsilon$, we have
$$ \|\Gamma(\mathbf{u})\|_{\mathbf{{L}}^{2}(I;\mathbf{L}^4)}\leq \frac{a}{2}+lCa^2=\left(\frac{1}{2}+lCa\right)a.$$
So, fixing $\epsilon$  such that $2 l^{2}C^{2}\epsilon< 1/2$, which means that $ lCa<1/2$,
we conclude that $\|\Gamma(\mathbf{u})\|_{\mathbf{{L}}^{2}(I;\mathbf{L}^4)}\leq a$. Therefore $\Gamma:\tilde{B}(T,a)\to\tilde{B}(T,a)$ is well defined.
A similar argument also shows that $\Gamma$ is a contraction. The contraction mapping principle then gives a unique solution in ${\mathbf{{L}}^{2}(I;\mathbf{L}^4)}$. Too see that such a solution indeed belongs to $\mathbf{X}(I)$ it suffices to use Strichartz's inequality and Lemma \ref{estinormdiffk} in \eqref{system2}. \end{proof}

 \subsection{Local existence of $H^1$-solutions}
Next we will study the existence of local solutions in the $H^1$ subcritical and critical regimes, that is, in dimensions $1\leq n\leq 6$. Thus, we assume that $u_{01},\ldots,u_{0l}\in {H}^1$ and  solve (\ref{system2}) in the following spaces:
\begin{equation*}
Y(I)=\begin{cases}
(\mathcal{C}\cap L^{\infty})(I; H^1)\cap L^{12/n}\left(I;W^{1}_{3}\right),\qquad 1\leq n\leq 3,\\
(\mathcal{C}\cap L^{\infty})(I; H^1)\cap L^{2}(I;W^{1}_{2n/(n-2)}),\qquad n\geq 4.
\end{cases}
\end{equation*}
on the time interval $I=[-T,T]$ with $T>0$. The norm in $Y(I)$ is defined as
\begin{equation*}
\|f\|_{Y(I)}=\begin{cases}
\|f\|_{L^{\infty}( H^1)}+ \|f\|_{L^{12/n}( W^{1}_{3})},\qquad 1\leq n\leq 3;\\
\|f\|_{L^{\infty}( H^1)}+ \|f\|_{L^{2}( W^{1}_{2n/(n-2)})}, \qquad n\geq 4.
\end{cases}
\end{equation*}

\begin{obs}\label{embedXY}
We point out the followings facts about the spaces $X(I)$ and $Y(I)$.
\begin{enumerate}
\item[(i)] $Y(I)\hookrightarrow X(I)$.
\item[(ii)] For $n\geq 1$ we have $\|\nabla f\|_{X(I)}\leq  \|f\|_{Y(I)}$.
\end{enumerate}
\end{obs}

Using Remark \ref{embedXY}, H\"{o}lder inequalities and Sobolev's embedding  we first establish the following.
\begin{lem}\label{estfgY} Assume $1\leq n\leq 6$ and $f,g\in Y(I)$. Define
 \begin{equation*}
 \theta(n)=
 \begin{cases}
 (4-n)/4, \quad\mbox{if}\quad 1\leq n\leq 3;\\
 (6-n)/4,\quad\mbox{if}\quad 4\leq n\leq 6.
 \end{cases}
 \end{equation*}
 Then,
 \begin{equation}\label{estfg1}
 \|fg\|_{L^{12/(12-n)}(L^{3/2})}\leq C T^{\theta(n)}\|f\|_{Y(I)}\|g\|_{Y(I)}, \quad if\quad 1\leq n\leq 3.
 \end{equation}
 and
 \begin{equation}\label{estfg2}
\|fg\|_{L^{4/(8-n)}( L^{n/(n-2)})}\leq
CT^{\theta(n)}\|f\|_{Y(I)}\|g\|_{Y(I)}, \quad\mbox{if}\quad 4\leq n\leq 6.
\end{equation}
 \end{lem}
\begin{proof}
Estimate \eqref{estfg1} follows immediately from Lemma \ref{estfg} taking into account Remark \ref{embedXY}. For \eqref{estfg2}, note that from H\"older and Sobolev's inequalities, 
$$
\|fg\|_{L^{n/(n-2)}}\leq \|f\|_{L^{2n/(n-2)}}\|g\|_{L^{2n/(n-2)}}\leq C\|f\|_{H^1}\|g\|_{L^{2n/(n-2)}}.
$$
Hence,
\[
\begin{split}
\|fg\|_{L^{4/(8-n)}( L^{n/(n-2)})}&\leq CT^{\theta(n)}\|f\|_{L^\infty (H^1)}\|g\|_{L^2(W^1_{2n/(n-2)})}\\
&\leq CT^{\theta(n)}\|f\|_{Y(I)}\|g\|_{Y(I)},
\end{split}
\]
which gives the desired.
\end{proof}

As a consequence of Lemma \ref{estgraddiffk} we have the following estimate for the integral part in system \eqref{system2}.

 \begin{lem}\label{estinormdiffkY}
 Let $1\leq n\leq 6$ and    assume that  \textnormal{\ref{H1}} and \textnormal{\ref{H2}} hold. Let $\theta(n)$ be defined as in Lemma \ref{estfgY}. If $\mathbf{u}, \mathbf{u}'\in \mathbf{Y}(I)$, for some time interval $I$, then
 \begin{equation*}
\left\|\int_{0}^{t}U_{k}(t-t') \frac{1}{\alpha_{k}}[ f_{k}(\mathbf{u})-f_{k}(\mathbf{u}')]\;dt'\right\|_{Y(I)}\leq CT^{\theta(n)}\left(\|\mathbf{u}\|_{\mathbf{Y}(I)}+\|\mathbf{u}'\|_{\mathbf{Y}(I)}\right)\|\mathbf{u}-\mathbf{u}'\|_{\mathbf{Y}(I)}.
\end{equation*}
\end{lem}
\begin{proof}
Note that a similar estimate as in Lemma \ref{estfgY} holds if we replace the product $fg$ by $\nabla(fg)$. Hence, for $1\leq n\leq 3$ the result follows from \eqref{estfg1} combined with Lemma \ref{estgraddiffk}  and Remark \ref{embedXY}. On the other hand, for $4\leq n\leq 6$, note that  $\left(\frac{4}{n-2},\frac{n}{2}\right)$ is an admissible pair with dual $\left(\frac{4}{8-n},\frac{n}{n-2}\right)$. Hence, the result follows as a combination of \eqref{estfg2}, Proposition \ref{stricha}, and Lemma \ref{estgraddiffk}.
\end{proof}

 Similarly to Theorems \ref{localexistenceL2} and \ref{locexistL2n=4}, a combination of the contraction mapping theorem with   Lemma \ref{estinormdiffkY}   allow us to prove the existence of local $H^{1}$ solutions as follows.

 \begin{teore}[Existence of local $H^1$-solutions: subcritical case]\label{localexistenceH1} Let $1\leq n\leq 5$. Assume that \textnormal{\ref{H1}} and \textnormal{\ref{H2}} hold. Then for any $r>0$ there exists $T(r)>0$ such that for any $\mathbf{u}_0\in \mathbf{H}^1 $ with $\|\mathbf{u}_0\|_{\mathbf{H}^1}\leq r$,  system \eqref{system1}
 has a unique  solution $\mathbf{u}\in \mathbf{Y}(I)$ with $I=[-T(r),T(r)]$.
\end{teore}

\begin{teore}[Existence of local $H^1$-solutions: critical case]\label{locexistH1n=6} Let $n=6$. Assume that  \textnormal{\ref{H1}} and \textnormal{\ref{H2}} hold. Then for any  $\mathbf{u}_0:=(u_{10},\ldots,u_{l0})\in \mathbf{H}^1 $ there exists $T(\mathbf{u}_0)>0$ such that  system \eqref{system1}
 has a unique  solution $\mathbf{u}=(u_{1},\ldots,u_{l})\in\mathbf{Y}(I)$ with $I=[-T(\mathbf{u}_0),T(\mathbf{u}_0)]$.
\end{teore}

\subsection{Global solutions}
This subsection is devoted to  extend globally-in-time the solutions given by Theorems \ref{localexistenceL2} and  \ref{localexistenceH1}. Since in such subcritical cases, the existence time depends only on the norm of the initial data, in addition to the conclusion of Theorems \ref{localexistenceL2} and  \ref{localexistenceH1}, a blow up alternative also holds, that is, there exist $T_*,T^*\in(0,\infty]$ such that the local solutions can be extend to the interval $(-T_*,T^*)$; moreover if $T_*<\infty$  (respect. $T^*<\infty$), then
$$
\lim_{t\to -T_*}\|\mathbf{u}(t)\|_{\mathbf{L}^2}=\infty, \qquad (respect.\lim_{t\to T^*}\|\mathbf{u}(t)\|_{\mathbf{L}^2}=\infty  ),
$$
for $L^2$-solutions, and 
$$
\lim_{t\to -T_*}\|\mathbf{u}(t)\|_{\mathbf{H}^1}=\infty, \qquad (respect.\lim_{t\to T^*}\|\mathbf{u}(t)\|_{\mathbf{H}^1}=\infty  ),
$$
for $H^1$-solutions. Thus, the idea to get global solutions is to find an \textit{a priori}  estimate for the local solution in $L^{2}$ and  $H^{1}$ based on the conservation of the charge and the energy.

 \subsubsection{Global existence of $L^2$-solutions}

Here, let us introduce the spaces
\begin{equation*}
X(\R)=\begin{cases}
(\mathcal{C}\cap L^{\infty})(\R; L^2)\cap L_{loc}^{12/n}(\R;L^{3}),\qquad 1\leq n\leq 3;\\
(\mathcal{C}\cap L^{\infty})(\R; L^2)\cap L_{loc}^{2}(\R;L^{2n/(n-2)}), \qquad n\geq 4.
\end{cases}
\end{equation*}
Our goal is to show that the solution indeed belongs to such spaces.  To do so, we need the conservation of the charge. To obtain this, we proceed formally, but the procedure can be made rigorous by
taking sufficient regular solutions and then passing to the limit or using the strategy in \cite{Ozawa}.

\begin{lem}\label{l2cons}
If  \textnormal{\ref{H3}} and \textnormal{\ref{H4}} hold, then the charge of system \eqref{system1} given by 
\begin{equation}\label{conservationcharge}	Q(\mathbf{u}(t)):=\sum_{k=1}^{l}\frac{\alpha_{k}^{2}}{\gamma_{k}}\|u_{k}(t)\|_{L^2}^{2},
	\end{equation}
is a conserved quantity.
\end{lem}
\begin{proof}
 Multiply \eqref{system1} by $\overline{u}_{k}$, integrate on $\R^{n}$ and take the imaginary part. Then summing over $k$ and using Lemma \ref{ReFinvari} the result follows.
\end{proof}

As an immediate consequence of Lemma \ref{l2cons} we have.

\begin{teore}
	Let $1\leq n\leq 3$. Assume that   \textnormal{\ref{H1}-\ref{H4}} hold. Then for any $\mathbf{u}_0 \in \mathbf{L}^2$, system \eqref{system1} has a unique solution $\mathbf{u}\in \mathbf{X}(\R)$. Moreover,
	\begin{equation*}
	Q(\mathbf{u}(t))=Q(\mathbf{u}_0),\qquad\forall t\in \R.
	\end{equation*}
\end{teore}

\subsubsection{Global existence of $H^1$-solutions} 
Similarly to the case of $L^2$ solutions, here we consider
\begin{equation*}
Y(\R)=\begin{cases}
(\mathcal{C}\cap L^{\infty})(\R; H^1)\cap L_{loc}^{12/n}\left(\R;W^{1}_{3}\right),\qquad 1\leq n\leq 3,\\
(\mathcal{C}\cap L^{\infty})(\R; H^1)\cap L^{2}_{loc}(\R;W^{1}_{2n/(n-2)}),\qquad n\geq 4.
\end{cases}
\end{equation*}

Next lemma establishes the conservation of the energy associated with \eqref{system1}.

\begin{lem}\label{lemconservenerg} If  \textnormal{\ref{H3}} holds, then the energy  associated with \eqref{system1} given by 
\begin{equation}
\label{conservationenergy}
	E(\mathbf{u}(t))=\sum_{k=1}^{l}\gamma_{k}\|\nabla u_{k}(t)\|_{L^2}^{2}+\sum_{k=1}^{l}\beta_{k}\|u_{k}(t)\|_{L^2}^{2}
    -2\mathrm{Re}\int F(\mathbf{u}(t))\;dx,
\end{equation}
is a conserved quantity.
\end{lem}
\begin{proof}
As in Lemma \ref{l2cons} we proceed formally, see \cite{Ozawa}. By multiplying \eqref{system1} by $\partial_{t}\overline{u}_{k}$, adding with its complex conjugate, integrating on $\R^{n}$ and  then 
summing over $k$ we see that
\begin{equation*}
\frac{d}{dt}\left(\sum_{k=1}^{l}\gamma_{k}\|\nabla u_{k}\|_{L^{2}}^{2}+\sum_{k=1}^{l}\beta_{k}\| u_{k}\|_{L^{2}}^{2}\right)=2\mathrm{Re}\int \sum_{k=1}^{l}  f_{k}\partial_{t}\overline{u}_{k}\;dx.
\end{equation*}
But in view of \ref{H3},
\begin{equation*}
\begin{split}
\mathrm{Re}\left[\frac{d}{dt}F(\mathbf{u}(t))\right]
&=\mathrm{Re}\left[\sum_{k=1}^{l}f_{k}\partial_{t}\overline{u}_{k}\right],
\end{split}
\end{equation*}
from which the result follows.
\end{proof}

Next, for $\mathbf{u}=(u_{1},\ldots,u_{l})$ we define the functionals
\begin{equation}\label{funcK1}
K(\mathbf{u})=\sum_{k=1}^{l}\gamma_{k}\|\nabla u_{k}\|_{L^2}^{2},
\end{equation}
\begin{equation}\label{funcL1}
L(\mathbf{u})=\sum_{k=1}^{l}\beta_{k}\|u_{k}\|_{L^2}^{2},
\end{equation}
\begin{equation*}
P(\mathbf{u})=\mbox{Re}\int F(\mathbf{u})\;dx,
\end{equation*}
\begin{equation*}
J(\mathbf{u})=\frac{Q(\mathbf{u})^{\frac{3}{2}-\frac{n}{4}}K(\mathbf{u})^{\frac{n}{4}}}{|P(\mathbf{u})|},
\end{equation*}
and the real number
\begin{equation}\label{xi0def}
\xi_{0}=\inf\{J(\mathbf{u});\mathbf{u}\in\mathbf{H}^{1},\quad P(\mathbf{u})\neq 0\},
\end{equation}
where $Q$ is defined in \eqref{conservationcharge}.

\begin{obs}
Using the previous functionals we can express the energy in \eqref{conservationenergy} as
\begin{equation}\label{conserenerfunc}
E(\mathbf{u}(t))=K(\mathbf{u}(t))+L(\mathbf{u}(t))-2P(\mathbf{u}(t)).
\end{equation}
\end{obs}

Let us observe that $\xi_{0}$ is indeed a positive constant.

\begin{lem}
Assume that \textnormal{\ref{H1}-\ref{H3}} and \textnormal{\ref{H6}} hold. Then, $\xi_{0}$ is a positive constant.
\end{lem}
\begin{proof}
First note that from the Gagliardo-Nirenberg inequality, for each $k=1,\ldots,l,$
	\begin{equation}\label{GNIuk}
	\|u_{k}\|_{L^3}^{3}\leq C^3\gamma_{k}^{-\frac{n}{4}}\left(\frac{\alpha_{k}^{2}}{\gamma_{k}}\right)^{\frac{n}{4}-\frac{3}{2}}Q(\mathbf{u})^{\frac{3}{2}-\frac{n}{4}}K(\mathbf{u})^{\frac{n}{4}}.
	\end{equation}

Using  Lemma \ref{estdifF}   and \eqref{GNIuk} we have
\begin{equation*}
\begin{split}
|P(\mathbf{u})|
&\leq \int \left|\mathrm{Re}\,F(\mathbf{u})\right|\;dx\leq C \sum_{k=1}^{l}\int |u_{k}|^{3}\;dx\leq  C_0Q(\mathbf{u})^{\frac{3}{2}-\frac{n}{4}}K(\mathbf{u})^{\frac{n}{4}},
\end{split}
\end{equation*}
where $C_0$ is a positive constant depending on $\alpha_{k}$ and $\gamma_{k}$, for $k=1,\ldots,l.$ Now, if $P(\mathbf{u})\neq 0$, then $|P(\mathbf{u})|>0$. So,
\begin{equation*}
0<\frac{1}{C_{0}}\leq J(\mathbf{u}),
\end{equation*}
and the conclusion follows.
\end{proof}

The above lemma allows us to establish the following Gagliardo-Nirenberg-type inequality:
\begin{equation}\label{GNE2}
|P(\mathbf{u})| \leq \frac{1}{\xi_{0}}Q(\mathbf{u})^{\frac{3}{2}-\frac{n}{4}}K(\mathbf{u})^{\frac{n}{4}}.
\end{equation}

We now  prove the existence of global $H^{1}$-solutions for  (\ref{system1}) in dimensions $1\leq n\leq 5$.

\begin{teore}\label{thm:globalwellposH1}
	Let $1 \leq n\leq 5$. Assume that \textnormal{\ref{H1}-\ref{H6}} hold.
	\begin{enumerate}
		\item[(i)] If $1 \leq n\leq 3$, then for any $\mathbf{u}_0\in \mathbf{H}^{1}$,  system \eqref{system1} has a unique  solution $\mathbf{u}\in \mathbf{Y}(\R)$.
		\item[(ii)] If $n=4$, then for any $\mathbf{u}_0\in \mathbf{H}^{1}$ satisfying
		\begin{equation}\label{smaldatL2norm}
		2Q(\mathbf{u}_0)^{\frac{1}{2}}<\xi_{0},
		\end{equation}
		system \eqref{system1} has a unique  solution $\mathbf{u}\in \mathbf{Y}(\R)$.
        \item[(iii)] $If n=5$, then for any $\mathbf{u}_0\in \mathbf{H}^{1}$ satisfying     
        \begin{equation}\label{smalldataH1}
		Q(\mathbf{u}_0)K(\mathbf{u}_0)<\left(\frac{2}{5}\xi_0\right)^{4}
		\end{equation}
        and
     \begin{equation}\label{smallenergyH1}
		Q(\mathbf{u}_0)E(\mathbf{u}_0)<\frac{1}{5}\left(\frac{2}{5}\xi_0 \right)^{4},
		\end{equation}
 system \eqref{system1} has a unique  solution $\mathbf{u}\in \mathbf{Y}(\R)$.   
	\end{enumerate}
\end{teore}
\begin{proof}
Clearly, it suffices to get an \textit{a priori} bound for $K(\mathbf{u}(t))$.
For  (i), by \eqref{GNE2} and  Young's inequality we can write, for any $\epsilon>0$,
 \begin{equation*}
    2|P(\mathbf{u})|\leq \epsilon K(\mathbf{u})+C_{\epsilon}Q(\mathbf{u})^{\frac{6-n}{4-n}}
    \end{equation*}
    for some constant $C_\epsilon$.
    Using the last inequality, the conservation of the energy and the fact that $-L(\mathbf{u})\leq 0$ we get an \textit{a priori} bound for $K(\mathbf{u})$. Indeed, from \eqref{conserenerfunc},  if $E_{0}=E(\mathbf{u}_0)$ and $Q_0=Q(\mathbf{u}_0)$, we deduce
\begin{equation*}
\begin{split}
K(\mathbf{u})&=E_{0}-L(\mathbf{u})+2P(\mathbf{u})\\
&\leq E_{0}+2\left|P(\mathbf{u})\right|\\
&\leq E_{0}+\epsilon K(\mathbf{u})+C_{\epsilon}Q_0^{\frac{6-n}{4-n}}
\end{split}
    \end{equation*}
Thus,  if $0<\epsilon<1$, then
\begin{equation}\label{GNE2.1}
K(\mathbf{u})\leq (1-\epsilon)^{-1}\left[E_{0}+C_{\epsilon}Q_0^{\frac{6-n}{4-n}}\right],
\end{equation}
as required. 

For (ii), from \eqref{conserenerfunc} and \eqref{GNE2}, we have
\begin{equation*}
K(\mathbf{u})\leq E_{0}+\frac{2}{\xi_{0}}Q_0^{\frac{1}{2}}K(\mathbf{u}),
\end{equation*}
or, equivalently,
\begin{equation}\label{positenerg1}
\left[1-\frac{2}{\xi_{0}}Q_0^{\frac{1}{2}}\right]K(\mathbf{u})\leq E_{0}.
\end{equation}
Hence, if \eqref{smaldatL2norm} holds then
$$K(\mathbf{u})\leq \left[1-\frac{2}{\xi_{0}}Q_0^{\frac{1}{2}}\right]^{-1}E_{0}, $$
as required.

In order to proof (iii), we use the following lemma (see, for instance, \cite{beg}, \cite{Esfahani} or \cite{Pastor} for its proof).

\begin{lem}\label{supercritcalcase}
Let $I$ an open interval with $0\in I$. Let $a\in \R$, $b>0$ and $q>1$. Define $\gamma=(bq)^{-\frac{1}{q-1}}$ and $f(r)=a-r+br^{q}$, for $r\geq 0$. Let $G(t)$ a non-negative continuous  function such that $f\circ G\geq 0$ on $I$. Assume that $a<\left(1-\frac{1}{q}\right)\gamma$.
\begin{enumerate}
\item[(i)] If $G(0)<\gamma$, then $G(t)<\gamma$, $\forall t\in I$.
\item[(ii)] If $G(0)>\gamma$, then $G(t)>\gamma$, $\forall t\in I$.
\end{enumerate}
\end{lem}

To apply Lemma \ref{supercritcalcase} in our case, we first note that
$$
K(\mathbf{u})\leq E_0+\frac{2}{\xi_{0}}Q_0^{\frac{1}{4}}K(\mathbf{u})^{\frac{5}{4}}.
$$
Therefore, we set $a=E_0$, $b=\frac{2}{\xi_{0}}Q_0^{\frac{1}{4}}$, $q=\frac{5}{4}$, and $G(t)=K(\mathbf{u}(t))$. Thus, since
$$
\gamma=(bq)^{-\frac{1}{q-1}}=\left(\frac{2}{5}\xi_{0}\right)^4\frac{1}{Q_0},
$$
it is easy to see that $a<\left(1-\frac{1}{q}\right)\gamma$ is equivalent to \eqref{smallenergyH1} and $G(0)<\gamma$ is equivalent to \eqref{smalldataH1}. Hence, Lemma \ref{supercritcalcase} gives the desired bound and the proof of the theorem is completed. 
\end{proof}

 \section{Existence of ground state solutions}\label{sec.gs}

 In this section we will prove the existence of ground state solutions for \eqref{system1}. Thus, we will assume that  \textnormal{\ref{H1}-\ref{H8}} hold. Recall that a standing wave solution for  \eqref{system1} is a solution of the form
\begin{equation}\label{standing}
u_{k}(x,t)=e^{i\frac{\alpha_{k}}{\gamma_{k}}\omega t}\psi_{k}(x),\qquad k=1,\ldots,l,
\end{equation}
where $\psi_{k}$ are real functions decaying to zero at infinity. Note that under the assumptions of Lemma \ref{H34impGC}, for $k=1,\ldots,l$ and any $\omega\in \R$, we have
\begin{equation*}
f_{k}\left(e^{i\frac{\alpha_{1}}{\gamma_{1}}\omega t}\psi_{1},\ldots,e^{\frac{\alpha_{l}}{\gamma_{l}}\omega it}\psi_{l}\right)=e^{i\frac{\alpha_{k}}{\gamma_{k}}\omega t}f_{k}(\psib),
\end{equation*}
where $\psib=(\psi_{1},\ldots,\psi_{l})$. 
 Thus, by replacing \eqref{standing} into \eqref{system1}, we see that  $\psi_{k}$ must satisfy the following elliptic system
\begin{equation}\label{system3}
\displaystyle -\gamma_{k}\Delta \psi_{k}+\left(\frac{\alpha_{k}^{2}}{\gamma_{k}}\omega+\beta_{k}\right) \psi_{k}=f_{k}(\psib),\qquad k=1,\ldots,l.
\end{equation}

\begin{obs}\label{fkrealremk}
\begin{enumerate}
\item[(i)] It is clear from  Lemma \ref{fkreal} that  $f_{k}$ are real-valued functions, i.e.,  
 $f_{k}(\psib)\in \R,$ $ k=1,\ldots,l$. Thus, system \eqref{system3} makes sense, since the right-hand side of the system is real.
 \item[(ii)]    Observe that $\psib=\mathbf{0}$ is always a solution (trivial solution) of \eqref{system3}. Hence, we  will always be  interested in non-trivial solutions. 
 \item[(iii)]  In order to obtain non-trivial solutions, here we restrict  the values of $\omega$ to those such that $\displaystyle \omega  > -\frac{\beta_{k}\gamma_{k}}{\alpha_{k}^{2}}$.
\end{enumerate}
\end{obs}

 To simplify notation, we note that system (\ref{system3}) can be written as
\begin{equation}\label{systemelip}
\displaystyle -\gamma_{k}\Delta \psi_{k}+b_{k} \psi_{k}=f_{k}(\psib),\quad k=1\ldots,l.
\end{equation}
where
\begin{equation*}\label{bdef}
b_{k}:=\frac{\alpha_{k}^{2}}{\gamma_{k}}\omega+\beta_{k}>0.
\end{equation*} 

Our goal then will be to find ground state solutions for \eqref{systemelip}. The action functional associated to  (\ref{systemelip}) is defined, for $\psib\in \mathbf{H}^{1}$, as
\begin{equation*}
I(\mbox{\boldmath$\psi$}):=\frac{1}{2}\left[\sum_{k=1}^{l}\gamma_{k}\|\nabla \psi_{k}\|_{L^2}^{2}+\sum_{k=1}^{l}b_{k}\| \psi_{k}\|_{L^2}^{2}\right]
-\int F(\psib)\;dx.
\end{equation*}

In addition, on $\mathbf{H}^{1}$, we define
\begin{equation*}
K(\psib):=\sum_{k=1}^{l}\gamma_{k}\|\nabla \psi_{k}\|_{L^2}^{2};
\end{equation*}
\begin{equation*}
\mathcal{Q}(\psib):=\sum_{k=1}^{l}b_{k}\| \psi_{k}\|_{L^2}^{2};
\end{equation*}
\begin{equation*}
P(\psib):=\int F(\psib)\;dx;
\end{equation*}
and
\begin{equation}\label{functionalJ}
J(\psib):=\frac{\mathcal{Q}(\psib)^{\frac{3}{2}-\frac{n}{4}}K(\psib)^{\frac{n}{4}}}{P(\psib)},\quad \quad P(\psib)\neq 0.
\end{equation}
Thus, the action $I$ can be expressed as
\begin{equation}\label{FunctionalI2}
I(\psib)=\frac{1}{2}\left[K(\psib)+\mathcal{Q}(\psib)\right]-P(\psib).
\end{equation}

Note that the functionals $K$, $\mathcal{Q}$, and $P$ are continuous on $ \mathbf{H}^1$ (the continuity of  $P$ follows from Lemma \ref{estdifF}). Next we show that indeed such functionals have Fr\'echet derivatives. In what follows, the primes represent the Fr\'echet derivatives.

\begin{lem}\label{frede}
If $\mathbf{g}\in \mathbf{H}^1$. Then
\begin{equation*}
K'(\psib)(\mathbf{g})=2\sum_{k=1}^{l}\gamma_{k}\int\nabla\psi_{k}\nabla g_{k}\;dx,
\end{equation*}
\begin{equation*}
\mathcal{Q}'(\psib)(\mathbf{g})=2\sum_{k=1}^{l}b_{k}\int \psi_{k} g_{k}\;dx,
\end{equation*}
and
\begin{equation*}
P'(\psib)(\mathbf{g})=\sum_{k=1}^{l}\int f_{k}(\psib)g_{k}\;dx.
\end{equation*}
\end{lem}
\begin{proof}
The proof is quite standard in view of our assumptions. So, we omit the details.
\end{proof}

In particular, Lemma \ref{frede} implies that $I$ has Fr\'echet derivative. The critical points of $I$ are the solutions of \eqref{systemelip}. More precisely,

\begin{defi}
We say that $\psib\in \mathbf{H}^1$ is a (weak) solution  of \eqref{systemelip} if for any $\mathbf{g}\in\mathbf{H}^1$, 
	\begin{equation}\label{infI}
	\displaystyle \gamma_{k} \int\nabla \psi_{k}  \nabla g_{k}\;dx+b_{k}\int \psi_{k} g_{k}\;dx=\int f_{k}(\psib)g_{k}\;dx,\quad k=1,\ldots,l.
	\end{equation}
	\end{defi}

\begin{defi}\label{defgroundstate} Let $\mathcal{C}$  be the set of non-trivial critical points of $I$. We say that $\psib\in \mathbf{H}^1$ is a ground state solution of \eqref{systemelip} if 
\begin{equation*}
I(\psib)=\inf\left\{I(\boldsymbol{\phi}); \boldsymbol{\phi}\in \mathcal{C}\right\}. 
\end{equation*}
We denote by $\mathcal{G}(\omega,\boldsymbol{\beta})$ the set of all ground states for system \eqref{systemelip}, where  $(\omega,\boldsymbol{\beta})$ indicates the dependence on the parameters $\omega$ and $\boldsymbol{\beta}$.
\end{defi}

Now we establish some relations between the functionals $K,\mathcal{Q}, P$ and $I$. This is similar to the well known Pohozaev's identities for elliptic equations.  
\begin{lem}
\label{identitiesfunctionals}
Let $\psib$ be a solution of \eqref{systemelip}. Then,
\begin{equation}
P(\psib)=2I(\psib),\label{b}\\
\end{equation}
\begin{equation}
K(\psib)=nI(\psib),\label{d}\\
\end{equation}
\begin{equation}
\mathcal{Q}(\psib)=(6-n)I(\psib).\label{e}
\end{equation}
\end{lem}
\begin{proof}
We first note that letting $g_{k}=\psi_{k}$, $k=1,\ldots,l$ in \eqref{infI} we have
\begin{equation*}
\displaystyle \gamma_{k}\|\nabla \psi_{k}  \|_{L^{2}}^{2}+b_{k}\| \psi_{k} \|_{L^{2}}^{2}=\int f_{k}(\psib)\psi_{k}\;dx,\quad k=1,\ldots,l.
\end{equation*}
From  Lemma \ref{propertiesF}, Remark \ref{fkrealremk} (i) and assumption \ref{H7}  we deduce
\begin{equation}\label{b1}
\sum_{k=1}^{l}f_{k}(\psib)\psi_{k}=3F(\psib).
\end{equation}
By summing over $k$ and using \eqref{b1}, we then get
\begin{equation}\label{1}
K(\psib)+ \mathcal{Q}(\psib)=3P(\psib).
\end{equation}
Therefore, (\ref{b}) follows from \eqref{FunctionalI2} and \eqref{1}.

In order to show  (\ref{d}) define $(\delta_{\lambda}f)(x)=f(x/\lambda)$. Then the function $\lambda\mapsto h(\lambda)=I(\delta_{\lambda}\psib)$ has a critical point at $\lambda=1$ or equivalently
$$h'(1)=\left.\frac{d}{d\lambda}\right|_{\lambda=1}I(\delta_{\lambda}\psib)=0.$$
But since
\begin{equation*}
\begin{split}
    h'(1)
&=\frac{n-2}{2}K(\psib)+n\left[\frac{1}{2}\mathcal{Q}(\psib)-P(\psib)\right],
\end{split}
\end{equation*}
we obtain
\begin{equation}\label{Pohosaeviden}
\frac{n-2}{2}K(\psib)+n\left[\frac{1}{2}\mathcal{Q}(\psib)-P(\psib)\right]=0.
\end{equation}
which combined with \eqref{FunctionalI2} gives \eqref{d}.

Finally, \eqref{e} follows as a combination of \eqref{Pohosaeviden} and \eqref{1} with \eqref{b}.
\end{proof}

\begin{obs}\label{remkharge}
Since $\mathcal{Q}(\psib)>0$ for any $\psib\neq\mathbf{0}$, it follows from \eqref{e} that \eqref{systemelip} has no  non-trivial solutions if $6-n\leq 0$. In addition, $\mathcal{Q}$ remains constant along $\mathcal{G}(\omega,\boldsymbol{\beta})$ and $\psib$ is a ground state  if and only if  $\mathcal{Q}(\psib)$ is minimal.
\end{obs}

Next we will prove that \eqref{systemelip} has at least one ground state solution. The idea is to minimize the Weinstein-type functional \eqref{functionalJ}. Before that, we need some preliminary results.

\begin{lem}\label{lemma4.4} Assume  $1\leq n\leq 5$ and define the set
$$ \mathcal{P}:=\{\psib\in \mathbf{H}^{1};\, P(\psib)>0\}.$$
Then,
\begin{enumerate}
\item[(i)] $
\mathcal{C}\subset \mathcal{P};
$
\item[(ii)] 
$
\xi_{1}:=\inf\{J(\psib);\;\psib\in \mathcal{P} \}>0.
$
\end{enumerate}
\end{lem}
\begin{proof}
 Statement (i) follows immediately from \eqref{b} and  \eqref{e}. For (ii) it  suffices to show that there exists a positive constant $B$ such that, for any $\psib\in \mathcal{P}$,
\begin{equation}\label{GNIbest}
P(\psib)\leq B \mathcal{Q}(\psib)^{\frac{3}{2}-\frac{n}{4}}K(\psib)^{\frac{n}{4}}.
\end{equation}
Now, from \eqref{GNIuk} we conclude  that, for $k=1\ldots, l$, 
\begin{equation*}
\|\psi_{k}\|_{L^3}^{3}\leq C^3\gamma_{k}^{-\frac{n}{4}}b_{k}^{\frac{n}{4}-\frac{3}{2}}\mathcal{Q}(\psib)^{\frac{3}{2}-\frac{n}{4}}K(\psib)^{\frac{n}{4}}.
\end{equation*}
Using this and \eqref{estdifFeq1} we reach the desired estimate. 
\end{proof}

Next we present a direct relation between functionals $J$ and $I$.
\begin{lem}\label{lemma3}
Assume	$1\leq n\leq 5$.
If $\psib$ is a non-trivial solution of \eqref{systemelip} then
\begin{eqnarray*}
J(\psib)=\frac{n^{\frac{n}{4}}}{2}\left(6-n\right)^{\frac{3}{2}-\frac{n}{4}}I(\psib)^{\frac{1}{2}}.
\end{eqnarray*}
In particular, any non-trivial solution $\psib \in \mathcal{P}$ of \eqref{systemelip}  which is a minimizer of $J$ is a ground state of \eqref{systemelip}.
\end{lem}
\begin{proof}
The result follows by combining \eqref{b}-\eqref{e} in Lemma \ref{identitiesfunctionals} and the definition of $J$. 
\end{proof}

In what follows, given any non-negative function $f\in H^1(\mathbb{R}^n)$ we denote by $f^*$ its symmetric-decreasing rearrangement (see, for instance, \cite{Leoni} or \cite{Lieb}). Also, for any $\lambda>0$, $(\delta_{\lambda}g)(x)=g\left(\frac{x}{\lambda}\right)$. Thus, if $\mathbf{g}=(g_1,\ldots,g_l)\in \mathbf{H}^1$, we set $\mathbf{g}^*=(g_1^*,\ldots,g_l^*)$ and $(\delta_{\lambda}\mathbf{g})(x)=\left(g_1\left(\frac{x}{\lambda}\right), \ldots, g_l\left(\frac{x}{\lambda}\right)\right)$
The functionals introduced in this section satisfy the following  properties about scaling transformations and symmetric-decreasing rearrangement.

\begin{lem}\label{functrans}
Let $n\geq 1$ and $a,\lambda>0$. If  $\psib\in \mathcal{P}$ and $\mathbf{g}\in (\mathcal{C}_{0}^{\infty}(\R^{n}))^{l}$ we have
\begin{enumerate}
\item[(i)] $\mathcal{Q}(a\delta_{\lambda}\psib)=a^{2}\lambda^{n}\mathcal{Q}(\psib)$;
\item[(ii)] $K(a\delta_{\lambda}\psib)=a^{2}\lambda^{n-2}K(\psib)$;
\item[(iii)] $P(a\delta_{\lambda}\psib)=a^{3}\lambda^{n}P(\psib)$;
\item[(iv)] $K'(a\delta_{\lambda}\psib)(\mathbf{g})=a\lambda^{n-2}K'(\psib)(\delta_{\lambda^{-1}}\mathbf{g})$;
\item[(v)] $\mathcal{Q}'(a\delta_{\lambda}\psib)(\mathbf{g})=a\lambda^{n}\mathcal{Q}'(\psib)(\delta_{\lambda^{-1}}\mathbf{g})$;
\item[(vi)] $P'(a\delta_{\lambda}\psib)(\mathbf{g})=a^{2}\lambda^{n}P'(\psib)(\delta_{\lambda^{-1}}\mathbf{g})$.\\
In addition, if $\psi_{k}$ is non-negative, for $k=1,\ldots,l$, then
\item[(vii)] $\mathcal{Q}(\psib^*)= \mathcal{Q}(\psib)$;
\item[(viii)] $K(\psib^*)\leq K(\psib)$;
\item[(ix)] $P(\psib^*)\geq P(\psib)$.
\end{enumerate}
\end{lem}

\begin{proof}
Properties (i)-(vi) follows immediately from the definitions taking into account Lemma \ref{frede}. Properties (vii) and (viii) follows from the facts that (see, for instance, \cite[Theorems 16.10 and 16.17]{Leoni})
\begin{equation*}
\|\psi^{*}\|_{L^{2}}=\|\psi\|_{L^{2}}, \qquad \mbox{and} \qquad\|\nabla\psi^{*}\|_{L^{2}}\leq\|\nabla \psi\|_{L^{2}}. 
\end{equation*}
Property (ix) is a little bit more delicate. Actually, our assumption \ref{H8} is a necessary and sufficient condition for (ix) to hold. See \cite[Theorem 1]{burchard} and \cite[Propostition 3.1]{haj}.
\end{proof}

\begin{lem}\label{lemma2}
Under the assumptions of Lemma \ref{functrans},
\begin{enumerate}
\item[(i)] $J(a\delta_{\lambda}\psib)=J(\psib)$;
\item[(ii)] $J(|\psib|)\leq J(\psib)$, where $|\psib|=(|\psi_{1}|,\ldots,|\psi_{l}|)$;
    \item[(iii)] $J'(a\delta_{\lambda}\psib)=a^{-1}J'(\psib)(\delta_{\lambda^{-1}}\mathbf{g})$.\\
    In addition, if $\psi_{k}$ is non-negative, for $k=1,\ldots,l$, then
    \item[(iv)] $J(\psib^{*})\leq J(\psib)$.
\end{enumerate}
\end{lem}

\begin{proof}
The  proof of (i), (iii) and (iv) are immediate consequences of Lemma \ref{functrans}. For (ii) we must  use assumption \ref{H5}. 
\end{proof}

With the above lemmas in hand, we are able to present our main result concerning ground states. As usual, we will say that a function $\psib\in\mathbf{H}^1$ is positive (non-negative), and write $\psib>0$ ($\psib\geq0$), if each one of its components are positive (non-negative). Also, $\psib$ is radially symmetric if each one of its components are radially symmetric.

\begin{teore}[Existence of ground state solutions]\label{thm:existenceGSJgeral}
 Assume that  \textnormal{\ref{H1}-\ref{H8}} hold. For  $1\leq n\leq 5$, the infimum
\begin{equation}\label{xi1}
\xi_1=\inf\limits_{\psib\in \mathcal{P}}J(\psib),
\end{equation}
introduced in Lemma \ref{lemma4.4}, is attained at a function $\psib_0\in \mathcal{P} $ such that
\begin{enumerate}
\item[(i)] $\psib_0$ is a non-negative and radially symmetric function.
\item[(ii)]   There exist $t_{0}>0$ and $\lambda_{0}>0$ such that $\psib=t_{0}\delta_{\lambda_{0}}\psib_0$ is a positive ground state solution of \eqref{systemelip}.
In addition,
if $\tilde{\psib}$ is any ground state of \eqref{systemelip} then
\begin{equation}\label{inffunctionalJ}
\xi_{1}=\frac{n^{\frac{n}{4}}}{2}\left(6-n \right)^{1-\frac{n}{4}}\mathcal{Q}(\tilde{\psib})^{1/2}.
\end{equation}
\end{enumerate}
\end{teore}
\begin{proof}
Let $(\psib_j)\subset \mathcal{P} $ be a minimizing sequence for \eqref{xi1}, i.e., 
$$\lim_{j\to \infty}J(\psib_j)=\xi_{1}.$$
Replacing $\psib_j$ by $|\psib_j|^*$, from Lemma \ref{lemma2}   we  may assume that $\psib_j$ are radially symmetric and non-increasing  functions in $\mathbf{H}^{1}$. Define $\tilde{\psib}_{j}=t_{j}\delta_{\lambda_{j}}\psib_{j}$,  where
$$t_{j}=\frac{\mathcal{Q}(\psib_j)^{\frac{n}{4}-\frac{1}{2}}}{K(\psib_j)^{\frac{n}{4}}}\qquad\mbox{and}\qquad \lambda_{j}=\frac{K(\psib_j)^{\frac{1}{2}}}{\mathcal{Q}(\psib_j)^{\frac{1}{2}}}.$$
 Lemmas \ref{functrans} and \ref{lemma2},  with $a=t_{j}$ and $\lambda=\lambda _{j}$ give
\begin{equation}\label{KQlimitado}
K(\tilde{\psib}_{j})=\mathcal{Q}(\tilde{\psib}_{j})=1 \quad \mbox{and} \quad J(\tilde{\psib}_j)=J(\psib_j).
\end{equation}
Hence,
\begin{equation}\label{convergenceP3}
\begin{split}
\frac{1}{P(\tilde{\psib}_{j})}=J(\tilde{\psib}_{j})=J(\psib_j)\to \xi_{1}>0.
\end{split}
\end{equation}
In view of \eqref{KQlimitado}, the sequence $(\tilde{\psib}_j)$ is bounded in $\mathbf{H}_{rd}^{1}$. By recalling that the embedding $H^{1}_{rd}(\R^{n}) \hookrightarrow L^{3}(\R^{n})$ is compact for $1\leq n\leq 5$ (see Proposition 1.7.1 in \cite{Cazenave}), there exist a subsequence, still denoted by $(\tilde{\psib}_{j})$, and $\psib_0\in \mathbf{H}_{rd}^{1}$  such that
\begin{equation}\label{strongcnvergenceL3Rn2}
\begin{cases}
\tilde{\psib}_j\rightharpoonup \psib_{0}, \qquad\mbox{in \quad $\mathbf{H}^{1}$},\\
\tilde{\psib}_{j}\to \psib_{0}\qquad\mbox{in\quad $\mathbf{L}^3$},\\
\tilde{\psib}_{j}\to \psib_{0}\qquad\mbox{ a.e \quad in\quad $\R^n$}.
\end{cases}
 \end{equation}
The last convergence in \eqref{strongcnvergenceL3Rn2} implies that $\psib_{0}$ is non-negative and radially symmetric. In addition, since by  Lemma \ref{estdifF},
\begin{equation*}
\begin{split}
\left|P(\tilde{\psib}_{j})-P(\psib_{0})\right|
&\leq \int\left|F(\tilde{\psib}_{j})-F(\psib_{0})\right|\;dx\\
&\leq C \sum_{m=1}^{l}\sum_{k=1}^{l}\int(|\tilde{\psi}_{kj}|^{2}+|\psi_{k0}|^{2})|\tilde{\psi}_{mj}-\psi_{m0}|\;dx \\
&\leq C \sum_{m=1}^{l}\sum_{k=1}^{l}(\|\tilde{\psi}_{kj}\|^{2}_{L^{3}}+\|\psi_{k0}\|^{2}_{L^{3}})\|\tilde{\psi}_{mj}-\psi_{m0}\|_{L^{3}},
\end{split}
\end{equation*}
we deduce from \eqref{strongcnvergenceL3Rn2} and (\ref{convergenceP3}) that
\begin{equation}\label{relationPandalpha2}
P(\psib_{0})=\lim_{j\to \infty}P(\tilde{\psib}_{j})=\xi_{1}^{-1}>0,
\end{equation}
which means that $\psib_0\in \mathcal{P}$.

On the other hand, the lower semi-continuity of the weak convergence  gives 
\begin{equation*}
K(\psib_{0})\leq \liminf_{j}K(\tilde{\psib}_{j})=1\quad
\mbox{and}\quad
\mathcal{Q}(\psib_{0})
\leq\liminf_{j}\mathcal{Q}(\tilde{\psib}_{j})=1.
\end{equation*}
Therefore, (\ref{relationPandalpha2}) yields
\begin{equation}\label{inequJKandQ2}
\xi_{1}\leq  J(\psib_{0})=\frac{\mathcal{Q}(\psib_{0})^{\frac{3}{2}-\frac{n}{4}}K(\psib_{0})^{\frac{n}{4}}}{P(\psib_{0})}\leq \frac{1}{P(\psib_{0})}=\xi_{1}.
\end{equation}
From \eqref{inequJKandQ2} we conclude that 
\begin{equation*}
    J(\psib_{0})=\xi_{1}
\end{equation*}
and
\begin{equation}\label{equal1}
    K(\psib_{0})=\mathcal{Q}(\psib_{0})=1.
\end{equation}
A combination of the last assertion with \eqref{strongcnvergenceL3Rn2} also implies that
 $\tilde{\psib}_{j}\to \psib_{0}$ strongly in $\mathbf{H}^{1}$. Part (i) of the theorem is thus established.
 
  For  part (ii) we  note that for $t$ sufficiently small and $\mathbf{u}\in\mathbf{H}^1$, $(\boldsymbol{\psi}_{0}+t\mathbf{u})\in \mathcal{P}$. Thus, since $\boldsymbol{\psi}_{0} $ is a minimizer of $J$ on $\mathcal{P} $  we have  
\begin{equation*}
\left.\frac{d}{dt}\right|_{t=0}J(\boldsymbol{\psi}_{0}+t\mathbf{u})=0,
\end{equation*}
which in view of Lemma \ref{lemma2} is equivalent to
\begin{multline*}
\frac{\mathcal{Q}(\boldsymbol{\psi}_{0})^{\frac{3}{2}-\frac{n}{4}}K(\boldsymbol{\psi}_{0})^{\frac{n}{4}}}{P(\boldsymbol{\psi}_{0})}\left[\frac{n}{4}\frac{K'(\boldsymbol{\psi}_{0})(\mathbf{u})}{K(\boldsymbol{\psi}_{0})}+\left(\frac{3}{2}-\frac{n}{4}\right)\frac{\mathcal{Q}'(\boldsymbol{\psi}_{0})}{\mathcal{Q}(\boldsymbol{\psi}_{0})}\right]
=\frac{\mathcal{Q}(\boldsymbol{\psi}_{0})^{\frac{3}{2}-\frac{n}{4}}K(\boldsymbol{\psi}_{0})^{\frac{n}{4}}}{P(\boldsymbol{\psi}_{0})^{2}}P'(\boldsymbol{\psi}_{0})(\mathbf{u}).
\end{multline*}
From \eqref{relationPandalpha2} and \eqref{equal1}  this yields
\begin{equation}
K'(\boldsymbol{\psi}_{0})(\mathbf{u})+\frac{6-n}{n}\mathcal{Q}'(\boldsymbol{\psi}_{0})(\mathbf{u})=\frac{4\xi_{1}}{n}P'(\boldsymbol{\psi}_{0})(\mathbf{u}).\label{relationKQandPgeral}
\end{equation}
Next, define  $\psib=t_{0}\delta_{\lambda_{0}}\psib_{0}$ with 
$$t_{0}=\displaystyle \frac{2 \xi_{1} }{6-n}\qquad\mbox{and}\qquad \lambda_{0}=\left(\frac{6-n}{ n}\right)^{1/2}.$$
 We claim that $\psib$ is a solution of (\ref{systemelip}). Indeed,   for any  $\mathbf{u}\in \mathbf{H}^{1}$ in view of Lemma \ref{functrans} and \eqref{relationKQandPgeral},
\begin{equation*}
\begin{split}
\quad I'(\boldsymbol{\psi})(\boldsymbol{u})
&=\frac{1}{2}\left[K'(\boldsymbol{\psi})(\boldsymbol{u})+\mathcal{Q}'(\boldsymbol{\psi})(\boldsymbol{u})\right]-P'(\boldsymbol{\psi})(\boldsymbol{u})\\
&=\frac{1}{2}\left[K'(t_{0}\delta_{\lambda_{0}}\boldsymbol{\psi}_{0})(\boldsymbol{u})+ \mathcal{Q}'(t_{0}\delta_{\lambda_{0}}\boldsymbol{\psi}_{0})(\boldsymbol{u})\right]-P'(t_{0}\delta_{\lambda_{0}}\boldsymbol{\psi}_{0})(\boldsymbol{u})\\
&=\frac{t_{0}}{2}\left[K'(\delta_{\lambda_{0}}\boldsymbol{\psi}_{0})(\boldsymbol{u})+ \mathcal{Q}'(\delta_{\lambda_{0}}\boldsymbol{\psi}_{0})(\boldsymbol{u})\right]-t_{0}^{2}P'(\delta_{\lambda_{0}}\boldsymbol{\psi}_{0})(\boldsymbol{u})\\
&=\frac{t_{0}}{2}\left[\lambda_{0}^{n-2}K'(\boldsymbol{\psi}_{0})(\delta_{\lambda_{0}^{-1}}\boldsymbol{u})+\lambda_{0}^{n}\mathcal{Q}'(\boldsymbol{\psi}_{0})(\delta_{\lambda_{0}^{-1}}\boldsymbol{u})\right]-t_{0}^{2}\lambda_{0}^{n}P'(\boldsymbol{\psi}_{0})(\delta_{\lambda_{0}^{-1}}\boldsymbol{u})\\
&=\frac{t_{0}\lambda_{0}^{n-2}}{2}\left[K'(\boldsymbol{\psi}_{0})(\delta_{\lambda_{0}^{-1}}\boldsymbol{u})+\lambda_{0}^{2} \mathcal{Q}'(\boldsymbol{\psi}_{0})(\delta_{\lambda_{0}^{-1}}\boldsymbol{u})-2t_{0}\lambda_{0}^{2}P'(\boldsymbol{\psi}_{0})(\delta_{\lambda_{0}^{-1}}\boldsymbol{u})\right]\\
&=\frac{t_{0}\lambda_{0}^{n-2}}{2}\left[K'(\boldsymbol{\psi}_{0})(\delta_{\lambda_{0}^{-1}}\boldsymbol{u})+\frac{6-n}{n}\mathcal{Q}'(\boldsymbol{\psi}_{0})(\delta_{\lambda_{0}^{-1}}\boldsymbol{u})-\frac{4\xi_{1} }{n}P'(\boldsymbol{\psi}_{0})(\delta_{\lambda_{0}^{-1}}\boldsymbol{u})\right]\\
&=0.
\end{split}
\end{equation*}

Now from Lemmas \ref{lemma2} and \ref{lemma3}, we have that $\psib$ is also a critical point of $J$ with $J(\psib)=J(\psib_{0})$. Since $\psib_{0}$ is a minimizer of $J$, so is $\psib$. Another application of Lemma \ref{lemma3} gives that $\psib$ is a ground state of (\ref{systemelip}). To see that $\psib$ is positive, we note that
\begin{equation*}
\Delta \psi_{k} -\frac{b_{k}}{\gamma_{k}} \psi_{k}=-\frac{1}{\gamma_{k}}f_{k}(\psib)\leq 0,
\end{equation*}
because $\gamma_{k}>0$,  $\psi_{k}$ are non-negatives  and $f_{k}$ satisfies \ref{H7}. Therefore by the strong maximum principle (see, for instance, \cite[Theorem 3.5]{gil})  we obtain the positiveness of $\psib$.
 
Finally, we will prove   (\ref{inffunctionalJ}). Indeed, if $\psib$ is as in part (ii), Lemmas \ref{identitiesfunctionals} and \ref{lemma3} imply,
\begin{equation*}
\xi_{1}=J(\psib)
=\frac{n^{\frac{n}{4}}}{2}\left(6-n\right)^{\frac{3}{2}-\frac{n}{4}}I(\psib)^{\frac{1}{2}}
=\frac{n^{\frac{n}{4}}}{2}\left(6-n \right)^{1-\frac{n}{4}}\mathcal{Q}(\psib)^{\frac{1}{2}}.
\end{equation*}
Therefore, if $\tilde{\psib}\in\mathcal{G}(\omega,\boldsymbol{\beta})$, from Remark \ref{remkharge} we get
$$\xi_{1}=\frac{n^{\frac{n}{4}}}{2}\left(6-n \right)^{1-\frac{n}{4}}\mathcal{Q}(\tilde{\psib})^{1/2},$$
completing the proof of the theorem.
\end{proof}

As a consequence of Theorem \ref{thm:existenceGSJgeral} we can obtain the sharp constant one can place in \eqref{GNIbest}. More precisely, we have

\begin{coro}\label{corollarybestconstant}
Let $1\leq n\leq 5$. The inequality 
\begin{equation*}
P(\mathbf{u})\leq C_{op}\mathcal{Q}(\mathbf{u})^{\frac{3}{2}-\frac{n}{4}}K(\mathbf{u})^{\frac{n}{4}},
\end{equation*}
holds, for any $\mathbf{u}\in \mathcal{P}$, with
\begin{equation*}
C_{op}=\frac{2\left(6-n\right)^{\frac{n}{4}-1}}{n^{\frac{n}{4}}}\frac{1}{\mathcal{Q}(\psib)^{\frac{1}{2}}},
\end{equation*}
where $\psib \in \mathcal{G}(\omega,\boldsymbol{\beta})$.
\end{coro}

We finish this section with the following  regularity result.

\begin{lem}\label{regellpsys}
Assume $1\leq n\leq5$ and let $\mbox{\boldmath$\varphi$}\in \mathbf{H}^1$ be a solution of 
	\begin{equation*}
	\displaystyle -\Delta \varphi_{k}+c_{k} \varphi_{k}=d_{k}f_{k}(\boldsymbol{\varphi}), \qquad k=1,\ldots,l,
	\end{equation*}
	where $c_{k},d_{k}$ are positive constants. Then,
	\begin{enumerate}
		\item[(i)] $\mbox{\boldmath$\varphi$}\in \mathbf{W}^{3,p}$ for $2\leq p<\infty$. In particular $\varphi_{k}$ is of class $C^2$ and $\sum_{k=1}^{l}|D^{\beta}\varphi_{k}(x)|\to 0$ as $|x|\to \infty$ for $|\beta|\leq 2$.
		\item[(ii)] There exist $\epsilon>0$ such that 
		\begin{equation*}
		e^{\epsilon |x|}\sum_{k=1}^{l}\left(|\varphi_{k}(x)|+|\nabla\varphi_{k}(x)|\right)\in L^{\infty}. 
		\end{equation*}
		In particular $|\cdot|\mbox{\boldmath$\varphi$}\in \mathbf{L}^{2}$.
	\end{enumerate}
\end{lem}
\begin{proof}
The proof is similar to that of Theorem 8.1.1 in \cite{Cazenave}. So, we omit the details.
\end{proof}

 \section{Global solutions versus blow-up}\label{sec.gsbu}
 In this section we establish global and  blow-up results for system \eqref{system1}. We assume that   \ref{H1}-\ref{H8} hold again.
 
 \subsection{Virial Identities}\label{virialindet}
 
Let us start with the following.

\begin{teore}\label{persistL2wheit} Let $1\leq n\leq 6$. Assume $\mathbf{u}_0\in \mathbf{H}^1$ and $x\mathbf{u}_0\in \mathbf{L}^2$. Let $\mathbf{u}$ be the corresponding  local solution given by Theorems \ref{localexistenceH1} and \ref{locexistH1n=6}. Then, the function $t\to |\cdot|\mathbf{u}(\cdot,t)$ belongs to $\mathcal{C}(I,\mathbf{L}^{2})$. Moreover, the function \begin{equation*}
 t\to V(t)=\sum_{k=1}^{l}\frac{\alpha_{k}^{2}}{\gamma_{k}}\|xu_{k}(t)\|_{L^2}^{2}=\sum_{k=1}^{l}\frac{\alpha_{k}^{2}}{\gamma_{k}}\int|x|^2|u_{k}(x,t)|^{2}\;dx
 \end{equation*}
  is in $\mathcal{C}^{2}(I)$, 
 \begin{equation*}
 V'(t)=4\sum_{k=1}^{l}\alpha_{k}\mathrm{Im}\int\nabla u_{k}\cdot x\overline{u}_{k}\,dx,
 \end{equation*}
 and
 \begin{equation}\label{secderV}
 V''(t)=2 nE_0-2 n L(\mathbf{u})+2 (4-n)K(\mathbf{u}), \qquad \mbox{for all}\;\; t\in I,
 \end{equation}
 where $E_0=E(\mathbf{u}_0)$, and $K$ and $L$ are defined in \eqref{funcK1} and \eqref{funcL1}.
\end{teore}
\begin{proof}
	 The proof  can be performed adapting the arguments of Proposition 6.5.1 in \cite{Cazenave}. Nevertheless, we adopt the  technique   presented in   \cite{Corcho}, which  explores  the Hamiltonian structure of the system.  
  The Hamiltonian form of system (\ref{system1}) is given by
\begin{equation*}
\frac{d}{dt}X(t)=JE'(X(t)),
\end{equation*}
where  $X(t)=\mathbf{u}(t)$, $J$ is the skew-adjoint operator given by
\begin{equation*}
J=\left(\begin{array}{cccc}
-i/2\alpha_{1}&\cdots&0&0\\
0&-i/2\alpha_{2} &\cdots&0\\
\vdots&\cdots&\ddots&\vdots\\
0&\cdots&0&-i/2\alpha_{l}
\end{array}\right),
\end{equation*}
and
$ E'$ stands for the Fr\'echet derivative of the energy $E$ in \eqref{conservationenergy}.

Let us now introduce the variance functional
\begin{equation*}
\mathcal{V}(\mathbf{u})=\int|x|^{2}\sum_{k=1}^{l}\frac{\alpha_{k}^{2}}{\gamma_{k}}|u_{k}|^{2}\;dx.
\end{equation*}
Note that 
\begin{equation}\label{Vprima}
\mathcal{V}'(\mathbf{u})=\left(\begin{array}{c}
2\frac{\alpha_{1}^{2}}{\gamma_{1}}|x|^{2}u_{1}\\
\displaystyle \vdots\\
2\frac{\alpha_{l}^{2}}{\gamma_{l}}|x|^{2}u_{l}
\end{array}\right).
\end{equation}
and $V(t)=\mathcal{V}(X(t))$. 
Thus,
\begin{equation*}
V'(t)=\frac{d}{dt}\mathcal{V}(X(t))=\langle \mathcal{V}'(X(t)),\frac{d}{dt}X(t)\rangle=\langle \mathcal{V}'(X(t)),JE'(X(t))\rangle=:\mathcal{P}(X(t)).
\end{equation*}
Thus, in order to determine $V'(t)$, it suffices to determine the functional $\mathcal{P}$. The idea to do that is to use a dual Hamiltonian system. Indeed, given $\tilde{X}_0$, assume the initial-value problem
\begin{equation}\label{auxHam}
\frac{d}{dt}\tilde{X}(t)=J\mathcal{V}'(\tilde{X}(t)), \qquad \tilde{X}(0)=\tilde{X}_0
\end{equation}
is (at least) locally well-posed. Then
\begin{equation*}
\frac{d}{dt}E(\tilde{X}(t))=\langle E'(\tilde{X}(t)),\frac{d}{dt}\tilde{X}(t)\rangle=\langle E'(\tilde{X}(t)),J\mathcal{V}'(\tilde{X}(t))\rangle =-\langle \mathcal{V}'(\tilde{X}(t)),J E'(\tilde{X}(t))\rangle=-\mathcal{P}(\tilde{X}(t)).
\end{equation*}
Evaluating at $t=0$, we deduce
$$
\mathcal{P}(\tilde{X}_0)=-\frac{d}{dt}E(\tilde{X}(t))\Big|_{t=0}.
$$
To summarize, in order to determine  $V'(t)$, it suffices to solve \eqref{auxHam} and then take the derivative of the energy at this solution evaluated at $t=0$.

In our case, in view of \eqref{Vprima}, problem \eqref{auxHam} takes the form
\begin{equation}\label{auxiliarsystem}
\begin{cases}
\partial_{t}\tilde{u}_{1}=-i\frac{\alpha_{1}}{\gamma_{1}}|x|^{2}\tilde{u}_{1},\qquad\tilde{u}_{1}(0)=\tilde{u}_{10},\\
\qquad\quad\vdots\\
\partial_{t}\tilde{u}_{l}=-i\frac{\alpha_{l}}{\gamma_{l}}|x|^{2}\tilde{u}_{l},\qquad\tilde{u}_{l}(0)=\tilde{u}_{l0},
\end{cases}
\end{equation}
with $\tilde{X}_0=(\tilde{u}_{10},\ldots,\tilde{u}_{l0})$.
Integrating (\ref{auxiliarsystem}) we obtain
\begin{equation*}
(\tilde{u}_{1}(x,t),\ldots,\tilde{u}_{l}(x,t))=\left(e^{-i\frac{\alpha_{1}}{\gamma_{1}}|x|^{2}t}\tilde{u}_{10},\ldots,e^{-i\frac{\alpha_{l}}{\gamma_{l}}|x|^{2}t}\tilde{u}_{l0}\right).
\end{equation*}
Since, for $k=1,\ldots,l$,
\begin{equation*}
\begin{split}
|\nabla \tilde{u}_{k}|^{2}
&=|\nabla \tilde{u}_{k0}|^{2}+4t\frac{\alpha_{k}}{\gamma_{k}}\mbox{Re}[ix\cdot \nabla \tilde{u}_{k0} \overline{\tilde{u}}_{k0}]+4t^{2}\frac{\alpha_{k}^{2}}{\gamma_{k}^{2}}|x\tilde{u}_{k0}|^{2},
\end{split}
\end{equation*}
we deduce
\begin{equation*}
\begin{split}
E(\tilde{X}(t))&=\sum_{k=1}^{l}\gamma_{k}\|\nabla \tilde{u}_{k0}\|_{L^2}^{2}+\sum_{k=1}^{l}\beta_{k}\|\tilde{u}_{k0}\|_{L^2}^{2}+4t\sum_{k=1}^{l}\alpha_{k}\int\mathrm{Re}[ix\cdot \nabla \tilde{u}_{k0} \overline{\tilde{u}}_{k0}]\;dx\\
&\quad+4t^{2}\sum_{k=1}^{l}\frac{\alpha_{k}^{2}}{\gamma_{k}}\|x\tilde{u}_{k0}\|^{2}_{L^{2}}-2\mathrm{Re}\int F\left(e^{-i\frac{\alpha_{1}}{\gamma_{1}}|x|^{2}t}\tilde{u}_{10},\ldots,e^{-i\frac{\alpha_{l}}{\gamma_{l}}|x|^{2}t}\tilde{u}_{l0}\right)\;dx.
 \end{split}
\end{equation*}
Now, from   \ref{H4} with  $\theta=\displaystyle -|x|^{2}t$ we infer
\begin{equation*}
\begin{split}
\mbox{Re}\,F\left(e^{-i\frac{\alpha_{1}}{\gamma_{1}}|x|^{2}t}\tilde{u}_{10},\ldots,e^{-i\frac{\alpha_{l}}{\gamma_{l}}|x|^{2}t}\tilde{u}_{l0}\right)
&=\mbox{Re}\,F(\tilde{X}_0),
\end{split}
\end{equation*}
which leads to
\begin{equation*}
\begin{split}
E(\tilde{X}(t))&=\sum_{k=1}^{l}\gamma_{k}\|\nabla \tilde{u}_{k0}\|_{L^2}^{2}+\sum_{k=1}^{l}\beta_{k}\|\tilde{u}_{k0}\|_{L^2}^{2}+4t\sum_{k=1}^{l}\alpha_{k}\mathrm{Re}[ix\cdot \nabla \tilde{u}_{k0} \overline{\tilde{u}}_{k0}]\;dx\\
&\quad+4t^{2}\sum_{k=1}^{l}\frac{\alpha_{k}^{2}}{\gamma_{k}}\|x\tilde{u}_{k0}\|^{2}_{L^{2}}-2\mathrm{Re}\int F(\tilde{X}_0)\;dx.
 \end{split}
\end{equation*}
Taking the derivative with respect to $t$ in the last expression and evaluating at $t=0$ gives
\begin{equation*}
\begin{split}
\frac{d}{dt}E(\tilde{X}(t))\Big|_{t=0} =4\sum_{k=1}^{l}\alpha_{k}\int\mathrm{Re}[ix\cdot \nabla \tilde{u}_{k0} \overline{\tilde{u}}_{k0}]\;dx
&=-4\mathrm{Im}\int\sum_{k=1}^{l}\alpha_{k}x\cdot \nabla \tilde{u}_{k} \overline{\tilde{u}}_{k}\;dx.
\end{split}
\end{equation*}
Therefore, 
\begin{equation*}
V'(t)=4\mathrm{Im}\int\sum_{k=1}^{l}\alpha_{k} x\cdot \nabla \tilde{u}_{k} \overline{\tilde{u}}_{k}\;dx,
\end{equation*}
as desired.

To obtain the second derivative of  $V$ we use the same idea with the functional $\mathcal{V}$ replaced by
\begin{equation*}
\mathcal{G}(\mathbf{u})=4\mathrm{Im}\int\sum_{k=1}^{l}\alpha_{k}x\cdot \nabla u_{k} \overline{u}_{k}\;dx.
\end{equation*}
We start by noticing that if $
\mathcal{F}(f)=\mathrm{Im}\int x\cdot \nabla f \overline{f}\;dx,
$
then the Fr\'echet derivative of $\mathcal{F}$ is given by
 $\mathcal{F}'(f)=-i(2x\cdot \nabla f+n f)$. Thus,
\begin{equation*}
\mathcal{G}'(\mathbf{u})=\left(\begin{array}{c}
-4\alpha_{1}i(2x\cdot \nabla u_{1}+n u_{1})\\
\vdots\\
-4\alpha_{l}i(2x\cdot \nabla u_{l}+n u_{l})
\end{array}\right),
\end{equation*}
and the IVP $\displaystyle 
\frac{d}{dt}\tilde{X}(t)=J\mathcal{G}'(\tilde{X}(t))$, $\tilde{X}(0)=\tilde{X}_{0}
$ takes the form
\begin{equation*}
\begin{cases}
(\tilde{u}_{1})_{t}=-4x\cdot \nabla u_{1}-2n u_{1},\qquad\tilde{u}_{1}(0)=\tilde{u}_{10},\\
\qquad\quad\vdots\\
(\tilde{u}_{l})_{t}=-4x\cdot \nabla u_{l}-2n u_{l},\qquad\tilde{u}_{l}(0)=\tilde{u}_{l0}.\\
\end{cases}
\end{equation*}
The solution of the last system is 
\begin{equation*}
(\tilde{u}_{1}(x,t),\ldots,\tilde{u}_{l}(x,t))=\left(e^{-2 nt}\tilde{u}_{10}\left(e^{-4 t}x\right),\ldots,e^{-2 nt}\tilde{u}_{l0}\left(e^{-4 t}x\right)\right).
\end{equation*}
Since $\displaystyle \nabla \tilde{u}_{k}=e^{-2 nt}e^{-4 t}\nabla \tilde{u}_{k0}\left(e^{-4 t}x\right)$, then
$$|\nabla \tilde{u}_{k}|^{2}=e^{-4 nt}e^{-8t}\left|\nabla \tilde{u}_{k0}\left(e^{-4 t}x\right)\right|^{2}.$$
This yields
\begin{equation}\label{energaux}
\begin{split}
E(\tilde{X}(t))&=\sum_{k=1}^{l}\gamma_{k}\int e^{-4 nt}e^{-8 t}\left|\nabla \tilde{u}_{k0}\left(e^{-4 t}x\right)\right|^{2}\;dx +\sum_{k=1}^{l}\beta_{k}\int e^{-4 nt}\left| \tilde{u}_{k0}\left(e^{-4 t}x\right)\right|^{2}\;dx\\
&\quad-2\mbox{Re}\int F\left(e^{-2nt}\tilde{u}_{10}\left(e^{-4 t}x\right),\ldots,e^{-2 nt}\tilde{u}_{l0}\left(e^{-4 t}x\right)\right)\;dx.
 \end{split}
\end{equation}
Since $F$ is homogeneous of degree 3 (see assumption \ref{H6}),
\begin{equation*}
 F\left(e^{-2 nt}\tilde{u}_{10}\left(e^{-4 t}x\right),\ldots,e^{-2 nt}\tilde{u}_{l0}\left(e^{-4 t}x\right)\right)=e^{-6nt}F\left(\tilde{u}_{10}\left(e^{-4 t}x\right),\ldots,\tilde{u}_{l0}\left(e^{-4 t}x\right)\right)
\end{equation*}
Combining this with the change of variable $e^{-4 t}x=y$, expression \eqref{energaux} reads as
\begin{equation*}
\begin{split}
E(\tilde{X}(t))&=e^{-8 t}\sum_{k=1}^{l}\gamma_{k}\|\nabla \tilde{u}_{k0}\|_{L^2}^{2}+\sum_{k=1}^{l}\beta_{k}\| \tilde{u}_{k0}\|_{L^2}^{2}-2 e^{-2 n t}\mathrm{Re}\int  F(\tilde{X}_{0}\left(y\right))\;dy,\\ 
 \end{split}
\end{equation*}
Consequently,
\begin{equation*}
\begin{split}
-\left.\frac{d}{dt}E(\tilde{X}(t))\right|_{t=0}
&=8 \sum_{k=1}^{l}\gamma_{k}\|\nabla \tilde{u}_{k0}\|_{L^2}^{2}-4 n\mathrm{Re}\int   F(\tilde{X}_{0}\left(y\right))\;dy,
\end{split}
\end{equation*}
and
\begin{equation}\label{viriuse}
V''(t)=8 \sum_{k=1}^{l}\gamma_{k}\|\nabla {u}_{k}(t)\|_{L^2}^{2}-4 n\mathrm{Re}\int F\left(\mathbf{u}(t)\right)\;dy.
\end{equation}
From the expression in the conservation law (\ref{conservationenergy}) we obtain \eqref{secderV}.
This complete the proof of the theorem.
\end{proof}

\begin{obs}\label{weightspc}
We introduce the following space
\begin{equation*}
\Sigma=\{\mathbf{u}\in \mathbf{H}^{1}; \;x\mathbf{u}\in \mathbf{L}^{2}\}.
\end{equation*}
Here the product $x\mathbf{u}$ must be understood as $(xu_1,\ldots,xu_l)$. In particular,
$$
\|x\mathbf{u}\|_{\mathbf{L}^2}=\sum_{k=1}^l\int|x|^2|u_k|^2dx.
$$
We note that $\Sigma$ equipped with the norm
\begin{equation*}
\|\mathbf{u}\|_{\Sigma}=\|\mathbf{u}\|_{\mathbf{H}^{1}}+\||\cdot|\mathbf{u}\|_{\mathbf{L}^{2}},
\end{equation*}
is a Hilbert space.
\end{obs}

As an immediate consequence of \eqref{secderV} we obtain.

\begin{coro}[Virial identity]\label{conservationlawweightedspace}
	Let $1\leq n\leq 6$.  Assume $\mathbf{u}_0\in \mathbf{H}^1$ and let $\mathbf{u}\in \Sigma$ be the corresponding solution given by Theorems \ref{localexistenceH1}, \ref{locexistH1n=6} and \ref{persistL2wheit}. Then
		\begin{multline*}
	Q\left(x\mathbf{u}(t)\right)=Q\left(x\mathbf{u}_0\right)+P_{0}t+nE_{0}t^{2}-2 n\int_{0}^{t}(t-s)L(\mathbf{u}(s))\; ds\\
	+2(4-n)\int_{0}^{t}(t-s)K(\mathbf{u}(s))\;ds,
	\end{multline*}
for all $t\in I$, where
$$P_{0}=4\sum_{k=1}^{l}\alpha_{k}\mathrm{Im}\int\nabla u_{k0}\cdot x\overline{u}_{k0}\;dx.$$
\end{coro}

Next we will pay particular attention to the case where the initial data is radially symmetric. Let us start by recalling the following.

\begin{lem}\label{StraussLemaconsequence}
If $f\in H^{1}(\R^{n})$ is a radially symmetric function, then
\begin{equation*}
\|f\|^{p+1}_{L^{p+1}(|x|\geq R)}\leq \frac{C}{R^{(n-1)(p-1)/2}}\|f\|^{(p+3)/2}_{L^{2}(|x|\geq R)}\|\nabla f\|^{(p-1)/2}_{L^{2}(|x|\geq R)}.
\end{equation*}
In particular, if $n=5$ and $p=2$, then
\begin{equation}\label{StraussLemmaineq}
\|f\|^{3}_{L^{3}(|x|\geq R)}\leq \frac{C}{R^{2}}\|f\|^{5/2}_{L^{2}(|x|\geq R)}\|\nabla f\|^{1/2}_{L^{2}(|x|\geq R)}.
\end{equation}
\end{lem}
\begin{proof}
The proof is a consequence of  Strauss' radial lemma (see also \cite[page 323]{Ogawa}).
\end{proof}

Next we deduce a similar result as in Theorem \ref{persistL2wheit} but with a smooth cut-off function instead of $|x|$.

\begin{teore}\label{Viarialidenityradialcase} Let $1\leq n\leq 6$. Assume $ \mathbf{u}_0 \in \mathbf{H}^{1}$ and  let $\mathbf{u}$ be the corresponding  given by Theorems \ref{localexistenceH1} and \ref{locexistH1n=6}. Assume $\varphi \in C^{\infty}_{0}(\R^{n})$ and define
\begin{equation*}
V(t)=\frac{1}{2}\int \varphi(x)\left(\sum_{k=1}^{l}\frac{\alpha_{k}^{2}}{\gamma_{k}}|u_{k}|^{2}\right)\;dx.
\end{equation*}
Then,
\begin{equation*}
V'(t)=\sum_{k=1}^{l}\alpha_{k}\mathrm{Im}\int\nabla \varphi\cdot \nabla u_{k} \overline{u}_{k}\;dx,
\end{equation*}
and
\begin{equation}\label{secondervgeralcase}
\begin{split}
V''(t)&=2\sum_{1\leq m,j\leq n}\mathrm{Re}\int\frac{\partial^{2}\varphi}{\partial x_{m}\partial x_{j}}\left[\sum_{k=1}^{l}\gamma_{k}\partial x_{j}\overline{u}_{k}\partial x_{m}u_{k}\right]dx\\
&\quad-\frac{1}{2}\int\Delta^{2}\varphi\left(\sum_{k=1}^{l}\gamma_{k}|u_{k}|^{2}\right)\;dx-\mathrm{Re}\int\Delta\varphi F\left(\mathbf{u}\right)\;dx.
\end{split}
\end{equation}
\end{teore}
\begin{proof}
The proof follows the ideas presented in  \cite[Lemma 2.9]{Kavian}. There, it was considered a local virial identity for a single Schr\"{o}dinger equation. An adapted version of it, for a system, can be founded in \cite[Theorem 2.1]{Pastor}. To get the first derivative of $V$ we use Lemma \ref{ReFinvari}. For the second derivative,  we use the consequences of \ref{H3} and \ref{H6} stayed in Lemma \ref{propertiesF}. 
\end{proof}

\begin{coro}
Under the assumptions of Theorem \ref{Viarialidenityradialcase}, if $\varphi$ and $\mathbf{u}_0$ are 
 radially symmetric functions,  we can write \eqref{secondervgeralcase} as
 \begin{equation}\label{secondervradialcase}
 \begin{split}
V''(t)&=2\int \varphi''\left(\sum_{k=1}^{l}\gamma_{k}|\nabla u_{k}|^{2}\right)dx-\frac{1}{2}\int\Delta^{2}\varphi\left(\sum_{k=1}^{l}\gamma_{k}|u_{k}|^{2}\right)dx\\
&\quad-\mathrm{Re}\int\Delta\varphi\, F\left(\mathbf{u}\right)\;dx.
\end{split}
\end{equation}
 \end{coro}

\begin{proof}
The proof follows immediately from Theorem  \ref{Viarialidenityradialcase}, taking into account that if $\mathbf{u}_0$ is radially symmetric so is $\mathbf{u}$.
\end{proof}

We finish this subsection with the following  result.

\begin{lem}\label{lemafunctionchi} Let  $r=|x|$, $x\in \R^{n}$. Take $\chi$ to be a smooth function with
\begin{equation*}
 \chi(r)=\left\{\begin{array}{cc}
r^{2},&0\leq r\leq 1\\
0,& r\geq 3
\end{array}\right.
\end{equation*}
and $\chi''(r)\leq 2$, for any $ r\geq 0$. Let $\chi_{R}(r)=R^{2}\chi\left(r/R\right)$.
 \begin{enumerate}
 \item If $r\leq R$, then $\Delta\chi_{R}(r)=2n$ and $\Delta^{2}\chi_{R}(r)=0$.
 \item If $r\geq R$, then
 \begin{equation*}
 \Delta\chi_{R}(r)\leq C_{1},\;\;\;\; \;\;\;\;\;\;\Delta^{2}\chi_{R}(r)\leq \frac{C_{2}}{R^{2}},
 \end{equation*}
 where $C_{1},C_{2}$ are constant depending on $n$.
 \end{enumerate}
\end{lem}

\begin{proof}
The lemma follows by a straightforward calculation.
\end{proof}

\subsection{Global existence in $H^{1}$}  

 In Theorem \ref{thm:globalwellposH1} we have proved that solutions of  system \eqref{system1} are global   in $H^1(\R^n)$, for $n=4$ and $n=5$, provided that the initial data is sufficiently small. Here we will see how small the initial data must be. To do so, we will use a particular set of ground states  to give sharp sufficient  conditions for the existence of global solutions. The ground states of interest are those with $b_k=\dfrac{\alpha_k^2}{\gamma_k}$;  that is, the ones satisfying the system

 \begin{equation}\label{OB1}
 \displaystyle -\gamma_{k}\Delta \psi_{k}+ \dfrac{\alpha_k^2}{\gamma_k} \psi_{k}=f_{k}(\psib),\quad k=1\ldots,l.
 \end{equation}


  \begin{obs}\label{rembetomega}
 In view of Theorem \ref{thm:existenceGSJgeral}, ground states for \eqref{OB1} do exists. In addition, they can be seen as  elements is the set $\mathcal{G}(1,\boldsymbol{0})$.
 \end{obs}

Our sharp criterion for global well-posedness will be given in terms of such ground states. More precisely,  Theorem \ref{thm:globalwellposH1} (ii)-(iii) can be reformulated as follows.

\begin{teore}[Sufficient condition for global solutions]\label{thm:globalexistencecondn=5}
 Assume $\mathbf{u}_0\in \mathbf{H}^{1}$ and let  $\mathbf{u}$ be the  solution of \eqref{system1} defined in the maximal  existence  interval $I$. Let  $\psib \in \mathcal{G}(1,\boldsymbol{0}) $.
\begin{enumerate}
    \item[(i)] Assume $n=4$. If
		\begin{equation}\label{L2GSCond}
	Q(\mathbf{u}_0)<Q(\psib),	
	\end{equation}
	then the initial value problem \eqref{system1} is globally well-posed in $\mathbf{H}^{1}$.	
\item[(ii)] Assume $n=5$ and in addition that  
\begin{equation}\label{conditionsharp1}
Q(\mathbf{u}_0)E(\mathbf{u}_0)<Q(\psib)\mathcal{E}(\psib),
\end{equation}
where $\mathcal{E}$ is the energy defined in \eqref{conservationenergy} with $\boldsymbol{\beta}=\boldsymbol{0}$.\\
If 
\begin{equation}\label{conditionsharp2}
Q(\mathbf{u}_0)K(\mathbf{u}_0)<Q(\psib)K(\psib),
\end{equation}
 then
\begin{equation*}
Q(\mathbf{u}_0)K(\mathbf{u}(t))<Q(\psib)K(\psib),\qquad t\in I.
\end{equation*}
 In particular the initial value problem \eqref{system1} is globally well-posed in $\mathbf{H}^{1}$.
\end{enumerate}
\end{teore}
\begin{proof}
 Recall that, from Theorem \ref{thm:existenceGSJgeral}, the numbers $\xi_0$ in \eqref{xi0def} and $\xi_1$ in \eqref{xi1} are the same. Moreover, in view of \eqref{inffunctionalJ} and the fact that $Q(\psib)=\mathcal{Q}(\psib)$ (under the assumption $b_k=\dfrac{\alpha_k^2}{\gamma_k}$), \eqref{smaldatL2norm} and \eqref{L2GSCond} are equivalent. So, part (i) follows from Theorem \ref{thm:globalwellposH1}.

For (ii), recall that
 \begin{equation*}
E(\mathbf{u}(t))=K(\mathbf{u}(t)) +L(\mathbf{u}(t))   -2P(\mathbf{u}(t)).
\end{equation*}
Hence, from Lemmas \ref{l2cons} and \ref{lemconservenerg} and Corollary  \ref{corollarybestconstant}, we deduce
\begin{equation}\label{inequalitysupercritical}
\begin{split}
K(\mathbf{u}(t))&=E_0-L(\mathbf{u}(t))+2P(\mathbf{u}(t))\\ 
&\leq E_0+2C_{op}Q_0^{\frac{1}{4}}K(\mathbf{u}(t))^{\frac{5}{4}}.
\end{split}
\end{equation}

Now, as in the proof of Theorem \ref{thm:globalwellposH1}, we apply Lemma \ref{supercritcalcase} with $a=E_0$, $b=2C_{op}Q_0^{\frac{1}{4}}$, $q=\frac{5}{4}$, and $G(t)=K(\mathbf{u}(t))$. It is easily seen that
$$
\gamma:=(bq)^{-\frac{1}{q-1}}=\frac{5Q(\psib)^2}{Q_0}.
$$
In addition, from Lemma \ref{identitiesfunctionals},  with $n=5$, we see that $K(\psib)=5Q(\psib)$ and $\mathcal{E}(\psib)=Q(\psib)$. As a consequence,  $a<\left(1-\frac{1}{q}\right)\gamma$  and $G(0)<\gamma$ are equivalent to (\ref{conditionsharp1})   and (\ref{conditionsharp2}), respectively.
 Lemma \ref{supercritcalcase} then yields the desired and the proof of the theorem is completed.
\end{proof}

\subsection{Blow-up results}
Now we will use the Virial identities stayed in section \ref{virialindet} to construct blow-up solutions.  In particular we will show that, in some cases, the assumptions in Theorem \ref{thm:globalexistencecondn=5} are sharp.

Let us start with the following.

\begin{pro}
Let $\mathbf{u}_0\in \mathbf{H}^1$ satisfy \eqref{L2GSCond} if $n=4$ or \eqref{conditionsharp2} if $n=5$. Then,
$$
E_0:=E(\mathbf{u}_0)>0.
$$
\end{pro}
\begin{proof}
If $n=4$ this follows as in \eqref{positenerg1} taking into account that $\xi_0^{-1}=C_{op}=\frac{1}{2Q(\psib)^{1/2}}$. In a similar fashion, if $n=5$ this follows as in \eqref{inequalitysupercritical} taking into account that $C_{op}=\frac{2}{5^{5/4}Q(\psib)^{1/2}}$ and using \eqref{conditionsharp2}.
\end{proof}

The next theorem  shows that $E(\mathbf{u}_0)>0$ is indeed a necessary condition in order to have global solution.

\begin{teore}\label{blowupZm}
	Let $4\leq n\leq 6$. Assume $\mathbf{u}_0\in \Sigma$ and let  $\mathbf{u}$ be the  solution of \eqref{system1} defined in the maximal  existence  interval, say, $I$. Then $I$ is finite if either
	\begin{enumerate}
		\item[(i)] $E_{0}<0$; or
		\item[(ii)] $E_{0}=0$, $P_{0}<0$,
	\end{enumerate}
	where $E_{0}$ and $P_{0}$ are as in Corollary \ref{conservationlawweightedspace}.
	\end{teore}
\begin{proof}
This result can be proved by using the classical convexity method in a similar fashion as that for  the nonlinear Schr\"odinger equation (see, for instance, \cite{Cazenave} or \cite{Linares}). So we omit the details.
\end{proof}

Next we will prove that, under some assumptions on the coefficients of system \eqref{system1}, conditions \eqref{L2GSCond} and \eqref{conditionsharp2} in Theorem  \ref{thm:globalexistencecondn=5} are sharp. More precisely, we will construct suitable initial data, which does not meet such a conditions and the corresponding solution blows-up in finite time (see also \cite{Hayashi}).

\subsubsection{$L^2$ critical case}
First we study the $L^{2}$ critical case; $n=4$. We start with the invariance of the system \eqref{system1} under the pseudo-conformal transformation. In what follows $SL(2,\R)$ denotes the special linear group of degree 2.
\begin{lem}\label{pseudoconf}
Assume $n=4$  and  let $A=\left(\begin{array}{cc}
      a&b  \\
      c&d 
 \end{array}\right)\in SL(2,\R)$. Define $\mathbf{v}^A=(v_1^A, \ldots,v_l^A)$ by
 \begin{equation*}
     v_{k}^{A}(x,t)=(a+bt)^{-2}e^{i\frac{\alpha_{k}}{\gamma_{k}}\frac{b|x|^{2}}{4(a+bt)}}u_{k}\left(\frac{x}{a+bt},\frac{c+dt}{a+bt}\right),\quad k=1,\ldots,l.
 \end{equation*}
If $\mathbf{u}$ is a solution of system \eqref{system1} with $\beta_{k}=0$, $k=1,\ldots,l$, so is $\mathbf{v}^A$.
\end{lem}
\begin{proof}
First observe that a straightforward but tedious calculation gives 
\begin{equation*}
     \begin{split}
         \left[i\partial_{t}v_{k}^{A}+\frac{\gamma_{k}}{\alpha_{k}}\Delta v_{k}^{A}\right](x,t)&=(a+bt)^{-4}e^{i\frac{\alpha_{k}}{\gamma_{k}}\frac{b|x|^{2}}{4(a+bt)}}\left[i\partial_{t}u_{k}+\frac{\gamma_{k}}{\alpha_{k}}\Delta u_{k}\right]\left(\frac{x}{a+bt},\frac{c+dt}{a+bt}\right),
     \end{split}
 \end{equation*}
 for $ k=1,\ldots,l.$ Moreover, Lemmas \ref{fkhomog2} and \ref{H34impGC} yield
 \begin{equation*}
     \begin{split}
         f_{k}(\mathbf{v}^A(x,t))&=f_{k}\left((a+bt)^{-2}e^{i\frac{\alpha_{1}}{\gamma_{1}}\frac{b|x|^{2}}{4(a+bt)}}u_{1},\ldots,(a+bt)^{-2}e^{i\frac{\alpha_{l}}{\gamma_{l}}\frac{b|x|^{2}}{4(a+bt)}}u_{l}\right)\\
         &=(a+bt)^{-4}e^{i\frac{\alpha_{k}}{\gamma_{k}}\frac{b|x|^{2}}{4(a+bt)}}f_{k}\left(\mathbf{u}\right),
     \end{split}
 \end{equation*}
 where we have omitted the argument $\left(\frac{x}{a+bt},\frac{c+dt}{a+bt}\right)$ in the right-hand side. The result then follows as a combination of the last two identities.
\end{proof}

\begin{obs}\label{remsoluin}
Note that $\mathbf{u}$ is a solution of \eqref{system1} with $\beta_{k}=\frac{\alpha_k^2}{\gamma_k}$, $k=1,\ldots,l$, if and only if $\tilde{\mathbf{u}}$ given by
$$
\tilde{u}_k(x,t)=e^{i\frac{\alpha_k}{\gamma_k}t}u_k(x,t),  \qquad k=1,\ldots,l,
$$
is also  solution of \eqref{system1} but with $\beta_{k}=0$, $k=1,\ldots,l.$
\end{obs}

Now let $\psib \in \mathcal{G}(1,\boldsymbol{0}) $. In particular $\psib$  is a solution of \eqref{system1} with $\beta_{k}=\frac{\alpha_k^2}{\gamma_k}$. Hence, from Remark \ref{remsoluin}, 
$$
\tilde{u}_k(x,t)=e^{i\frac{\alpha_k}{\gamma_k}t}\psib(x),  \qquad k=1,\ldots,l,
$$
is a solution of \eqref{system1} with $\beta_k=0$. Moreover, by Lemma \ref{pseudoconf}, for any $A\in SL(2,\R)$, $\mathbf{v}^A$ defined by
$$
 v_{k}^{A}(x,t)=(a+bt)^{-2}e^{i\frac{\alpha_{k}}{\gamma_{k}}\frac{b|x|^{2}}{4(a+bt)}} e^{i\frac{\alpha_k}{\gamma_k}\frac{c+dt}{a+bt}}\psi_k\left(\frac{x}{a+bt}\right), \qquad k=1,\ldots,l,
$$
is also a solution. With this in hand we are able to establish the following.

\begin{teore}\label{sharpn=4}
Assume $n=4$ and let $\psib \in \mathcal{G}(1,\boldsymbol{0}) $. For any $T>0$ let  $A=\left(\begin{array}{cc}
T&- 1 \\
0&\frac{1}{T}
\end{array}\right)$
in such a way that
\begin{equation*}
     v_{k}^{A}(x,t)=(T-t)^{-2}e^{-i\frac{\alpha_{k}}{\gamma_{k}}\frac{|x|^{2}}{4(T-t)}+i\frac{\alpha_{k}}{\gamma_{k}}\frac{t}{T(T-t)}}\psi_{k}\left(\frac{x}{T-t}\right),\quad k=1,\ldots,l.
 \end{equation*}
Then
 \begin{enumerate}
     \item[(i)] $\mathbf{v}^A$  is a solution of\eqref{system1} with $\beta_{k}=0$ for $k=1,\ldots,l$.
     \item[(ii)] 
     \begin{equation*}
         v_{k}^{A}(x,0)=T^{-2}e^{-i\frac{\alpha_{k}}{\gamma_{k}}\frac{|x|^{2}}{4T}}\psi_{k}\left(\frac{x}{T}\right),\quad k=1,\ldots,l.
     \end{equation*}
     \item[(iii)] $Q(\mathbf{v}^A(0))=Q(\psib)$.
     \item[(iv)] $K(\mathbf{v}^A(t))=O((T-t)^{-2})$ as $t\to T^{-}$.
 \end{enumerate}
\end{teore}

\begin{proof}
Statement (i) is a consequence of the Lemma \ref{pseudoconf}. Statements (ii), (iii) and (iv) follows from a direct calculation.
\end{proof}

\begin{coro}\label{corosharp}
Under the assumptions of Theorem \ref{sharpn=4}, if  $\mathbf{u}^A$ is defined by
$$
u_k^A(x,t)=e^{-i\frac{\alpha_k}{\gamma_k}t}v_k^A(x,t),
$$
then $\mathbf{u}^A$ is a solution of \eqref{system1} with $\beta_{k}=\frac{\alpha_k^2}{\gamma_k}$, $k=1,\ldots,l$,  such that $Q(\mathbf{u}^A(0))=Q(\psib)$  and $\mathbf{u}^A$ blows-up in finite time.
\end{coro}

 Corollary \ref{corosharp} shows that  part (i) in Theorem \ref{thm:globalexistencecondn=5} is sharp under the assumption  $\beta_{k}=\frac{\alpha_k^2}{\gamma_k}$, $k=1,\ldots,l$.

\subsubsection{$L^2$ supercritical case}

Next we analyze the $L^{2}$ supercritical and $H^{1}$ subcritical case; $n=5$.  We will follow the ideas presented in   \cite{Holmer2}, \cite{Pastor} and \cite{Ogawa}.

\begin{teore}[Existence of blow-up solutions]\label{thm:sharpglobalexistencecondn=5} Let $n=5$. Assume $\mathbf{u}_0\in \mathbf{H}^{1}$ and let  $\mathbf{u}$ be the corresponding  solution of \eqref{system1} defined in the maximal  existence  interval, say, $I$. Let   $\psib \in \mathcal{G}(1,\boldsymbol{0})$.  Assume, also that 
\begin{equation}\label{energycondblowup}
Q(\mathbf{u}_0)E(\mathbf{u}_0)<Q(\psib)\mathcal{E}(\psib),
\end{equation}
and
\begin{equation}\label{gradientcondblowup}
Q(\mathbf{u}_0)K(\mathbf{u}_0)>Q(\psib)K(\psib).
\end{equation}
If $x\mathbf{u}_0\in \mathbf{L}^{2}$ or $\mathbf{u}_0$ is radially symmetric, then $I$ is finite.
\end{teore}

Before proving Theorem we recall a slightly modification of part (ii) in Lemma \ref{supercritcalcase}.

\begin{lem}\label{corosupercritcalcase}
	Let $I$ an open interval with $0\in I$. Let $a\in \R$, $b>0$ and $q>1$. Define $\gamma=(bq)^{-\frac{1}{q-1}}$ and $f(r)=a-r+br^{q}$, for $r\geq 0$. Let $G(t)$ a non-negative continuous  function such that $f\circ G\geq 0$ on $I$. Assume that $a<(1-\delta_{1})\left(1-\frac{1}{q}\right)\gamma$, for some small $\delta_{1}>0$.
	
	If $G(0)>\gamma$, then there exist $\delta_{2}=\delta_{2}(\delta_{1})>0$ such that $G(t)>(1+\delta_{2})\gamma$, $\forall t\in I$.
\end{lem}
\begin{proof}
See \cite[Corollary 3.2]{Pastor}
\end{proof}

\begin{proof}[Proof of Theorem~\ref{thm:sharpglobalexistencecondn=5}]
From (\ref{energycondblowup}) it is clear we may obtain  $\delta_{1}>0$ small such that
\begin{equation}\label{conddelta1}
Q(\mathbf{u}_0)E(\mathbf{u}_0)<(1-\delta_{1})Q(\psib)\mathcal{E}(\psib).
\end{equation}
With the same notation of the proof of part (ii) of Theorem \ref{thm:globalexistencecondn=5}, it is easily checked that $G(0)>\gamma$ is equivalent to \eqref{gradientcondblowup} and $a<(1-\delta_{1})\left(1-\frac{1}{q}\right)\gamma$ is equivalent to \eqref{conddelta1}. Hence, by Lemma 
 \ref{corosupercritcalcase} there exist $\delta_{2}>0$ such that
\begin{equation}\label{conddelta2}
Q(\mathbf{u}_0)K(\mathbf{u}(t))>(1+\delta_{2})Q(\psib)K(\psib).
\end{equation}

Let us first assume  $x\mathbf{u}_0\in \mathbf{L}^{2}$. From \eqref{secderV} with  $n=5$, we have
\begin{equation}\label{virialidentity2}
V''(t)=10 E(\mathbf{u}_0)-10L(\mathbf{u}(t))-2K(\mathbf{u}(t)),\qquad t\in I.
\end{equation}

By multiplying both sides of (\ref{virialidentity2})  by $Q(\mathbf{u}_0)$, using \eqref{conddelta1}-\eqref{conddelta2} and the fact that $\mathcal{E}(\psib)=(1/5)K(\psib)$, we obtain, for any $t\in I$,
\begin{equation*}
\begin{split}
V''(t)Q(\mathbf{u}_0)&=10 E(\mathbf{u}_0)Q(\mathbf{u}_0)-10L(\mathbf{u}(t))Q(\mathbf{u}_0)
   -2 K(\mathbf{u}(t))Q(\mathbf{u}_0)\\
&<10(1-\delta_{1})\mathcal{E}(\psib)Q(\psib) -2 (1+\delta_{2})Q(\psib)K(\psib)\\
&=2 (1-\delta_{1})K(\psib)Q(\psib)-2 (1+\delta_{2})Q(\psib)K(\psib)\\
&=-2\left(\delta_{1}+\delta_{2}\right)Q(\psib)K(\psib)\\
&=:-B,
\end{split}
\end{equation*}
where $B$ is a positive constant.
Thus, if we assume that  $I$ is infinite must exist $t^{*}\in I$ such that $V(t^{*})<0$, which is a contradiction,  because $V>0$. Therefore $I$ must be finite.

Now, we assume that $\mathbf{u}_0$ is radially symmetric.  Thus, by taking $\varphi$ as $\chi_{R}$ in (\ref{secondervradialcase}), where $\chi_{R}$ is given in Lemma \ref{lemafunctionchi}, we get
 \begin{equation}\label{secodnderivativeVwithchi}
 \begin{split}
V''(t)&=2\int \chi_{R}''\left(\sum_{k=1}^{l}\gamma_{k}|\nabla u_{k}|^{2}\right)\;dx-\frac{1}{2}\int\Delta^{2}\chi_{R}\left(\sum_{k=1}^{l}\gamma_{k}|u_{k}|^{2}\right)\;dx\\&\quad-\mathrm{Re}\int\Delta\chi_{R} F\left(\mathbf{u}\right)\;dx.
\end{split}
\end{equation}

We will estimate each one of the terms in $V''$. For the first one, using the fact that $\chi_{R}''(r)\leq 2$,  we have
\begin{equation}\label{estim1}
2\int \chi_{R}''\left(\sum_{k=1}^{l}\gamma_{k}|\nabla u_{k}|^{2}\right)\;dx\leq 4\sum_{k=1}^{l}\gamma_{k}\|\nabla u_{k}\|^{2}_{L^{2}}= 4K(\mathbf{u}).
\end{equation}

For the second one, using Lemma \ref{lemafunctionchi} and the conservation of charge we get
\begin{equation}\label{estim2}
\begin{split}
-\frac{1}{2}\int\Delta^{2}\chi_{R}\left(\sum_{k=1}^{l}\gamma_{k}|u_{k}|^{2}\right)dx&=-
\frac{1}{2}\int_{\{|x|\geq R\}}\Delta^{2}\chi_{R}\left(\sum_{k=1}^{l}\gamma_{k}|u_{k}|^{2}\right)dx\\
&\leq\frac{C_{2}}{R^{2}}\int_{\{|x|\geq R\}}\left(\sum_{k=1}^{l}\gamma_{k}|u_{k}|^{2}\right)dx\\
&\leq\frac{C_{2}}{R^{2}}\max_{1\leq j\leq l}\left\{\frac{\gamma_{j}^{2}}{\alpha_{j}^{2}}\right\}\sum_{k=1}^{l}\frac{\alpha_{k}^{2}}{\gamma_{k}}\|u_{k}\|^{2}_{L^{2}}\\
&=\frac{{C}_{2}'}{R^{2}}Q(\mathbf{u}_0),
\end{split}
\end{equation}
for some positive constant $C_2'$.
Finally, the last term in  (\ref{secodnderivativeVwithchi}) is estimated as follows
\begin{equation*}
\begin{split}
-\mbox{Re}\int\Delta\chi_{R} F\left(\mathbf{u}\right)\;dx&=-\mathrm{Re}\int_{\{|x|\leq R\}}\Delta\chi_{R} F\left(\mathbf{u}\right)\;dx-\mathrm{Re}\int_{\{|x|\geq R\}}\Delta\chi_{R} F\left(\mathbf{u}\right)\;dx\\
&\leq-10\;\mathrm{Re}\int_{\{|x|\leq R\}} F\left(\mathbf{u}\right)\;dx+C_{1}\int_{\{|x|\geq R\}}\left|\mathrm{Re}\,F\left(\mathbf{u}\right)\right|\;dx\\
&=-10\;\mathrm{Re}\,\int_{\R^{5}} F\left(\mathbf{u}\right)\;dx+C_{1}'\int_{\{|x|\geq R\}}\left|\mathrm{Re}\,F\left(\mathbf{u}\right)\right|\;dx\\
&= -10 P(\mathbf{u})+C_{1}'\int_{\{|x|\geq R\}}\left|\mathrm{Re}\,F\left(\mathbf{u}\right)\right|\;dx,
\end{split}
\end{equation*}
where we have used   Lemma  \ref{lemafunctionchi} with $n=5$. Here $C_1'$ is also a positive constant.
Now the conservation of the energy  and \eqref{conserenerfunc} imply
$-10 P(\mathbf{u})=5E(\mathbf{u}_0)-5K(\mathbf{u})-5L(\mathbf{u}).$
Thus,
\begin{equation}\label{estim3}
\begin{split}
-\mathrm{Re}\int\Delta\chi_{R} F\left(\mathbf{u}\right)\;dx&\leq 5E(\mathbf{u}_0)-5K(\mathbf{u})-5L(\mathbf{u})+C_{1}'\int_{\{|x|\geq R\}}\left|\mathrm{Re}\, F\left(\mathbf{u}\right)\right|\;dx\\
&\leq 5E(\mathbf{u}_0)-5K(\mathbf{u})+C_{1}'\int_{\{|x|\geq R\}}\left|\mathrm{Re}\,F\left(\mathbf{u}\right)\right|\;dx.
\end{split}
\end{equation}
Gathering together \eqref{secodnderivativeVwithchi}-(\ref{estim3}), we have
\begin{equation}\label{1b}
\begin{split}
V''(t)&\leq 5E(\mathbf{u}_0)-K(\mathbf{u})+\frac{{C'}_{2}}{R^{2}}Q(\mathbf{u}_0)+C_{1}'\int_{\{|x|\geq R\}}\left|\mbox{Re}\,F\left(\mathbf{u}\right)\right|\;dx\\
&\leq 5E(\mathbf{u}_0)-K(\mathbf{u})+\frac{{C'}_{2}}{R^{2}}Q(\mathbf{u}_0)+C_{1}'C\sum_{k=1}^{l}\|u_{k}\|^{3}_{L^{3}(|x|\geq R)},
\end{split}
\end{equation}
where we used Lemma \ref{estdifF}.

Next, by using  (\ref{StraussLemmaineq}) and Young's inequality with $\epsilon$ (small) we conclude that
 \begin{equation*}
 \begin{split}
\sum_{k=1}^{l}\|u_{k}\|^{3}_{L^{3}(|x|\geq R)}
&\leq \frac{\tilde{C}}{R^{2}}\sum_{k=1}^{l}\|u_{k}\|^{5/2}_{L^{2}(|x|\geq R)}\|\nabla u_{k}\|^{1/2}_{L^{2}(|x|\geq R)}\\
&\leq C_{\epsilon}\sum_{k=1}^{l}\left\{R^{-2}\left[\left(\frac{\alpha_{k}^{2}}{\gamma_{k}}\right)^{1/2}\|u_{k}\|_{L^{2}(|x|\geq R)}\right]^{5/2}\right\}^{4/3} \\
&\quad+\epsilon\sum_{k=1}^{l}\left\{\left[\gamma_{k}^{1/2}\|\nabla u_{k}\|_{L^{2}(|x|\geq R)}\right]^{1/2}\right\}^{4}\\
&\leq \frac{\tilde{C}_{\epsilon}}{R^{8/3}}\left(\sum_{k=1}^{l}\frac{\alpha_{k}^{2}}{\gamma_{k}}\|u_{k}\|^{2}_{L^{2}(|x|\geq R)}\right)^{5/3}+\epsilon\sum_{k=1}^{l}\gamma_{k}\|\nabla u_{k}\|^{2}_{L^{2}(|x|\geq R)}\\
&\leq\frac{\tilde{C}_{\epsilon}}{R^{8/3}}Q(\mathbf{u}_0)^{5/3}+\epsilon K(\mathbf{u}).
\end{split}
 \end{equation*}
 Therefore, from \eqref{1b},
\begin{equation}
\begin{split}\label{estimaV2da}
V''(t)&\leq 5E(\mathbf{u}_0)-\left(1-\epsilon\right)K(\mathbf{u})+\frac{C_2'}{R^{2}}Q(\mathbf{u}_0)
+\frac{\tilde{C}_{\epsilon}}{R^{8/3}}Q(\mathbf{u}_0)^{5/3}.
\end{split}
\end{equation}
Multiplying (\ref{estimaV2da}) by $Q(\mathbf{u}_0)$, we obtain
\begin{equation*}
\begin{split}
Q(\mathbf{u}_0)V''(t)&\leq 5E(\mathbf{u}_0)Q(\mathbf{u}_0)-\left(1-\epsilon\right)K(\mathbf{u})Q(\mathbf{u}_0)
+\frac{C_2'}{R^{2}}Q(\mathbf{u}_0)^{2}+\frac{\tilde{C}_{\epsilon}}{R^{8/3}}Q(\mathbf{u}_0)^{8/3}.
\end{split}
\end{equation*}
 
 Using (\ref{conddelta1}), (\ref{conddelta2}) we can write
\begin{equation*}
\begin{split}
Q(\mathbf{u}_0)V''(t)&\leq 5(1-\delta_{1}) Q(\psib)\mathcal{E}(\psib)-\left(1-\epsilon\right)(1+\delta_{2})Q(\psib)K(\psib)
+\frac{C_2'}{R^{2}}Q(\mathbf{u}_0)^{2}+\frac{\tilde{C}_{\epsilon}}{R^{8/3}}Q(\mathbf{u}_0)^{8/3}\\
&=\left[-(\delta_{1}+\delta_{2})+\epsilon(1+\delta_{2})\right]Q(\psib)K(\psib)+
\frac{C_2'}{R^{2}}Q(\mathbf{u}_0)^{2}+\frac{\tilde{C}_{\epsilon}}{R^{8/3}}Q(\mathbf{u}_0)^{8/3},
\end{split}
\end{equation*}
where we used that $\mathcal{E}(\psib)=(1/5)K(\psib)$. Choosing $\epsilon>0$ small enough and  $R>0$ sufficiently large , we can conclude that $V''(t)<-B$, for some constant $B>0$. As above, we then conclude that $I$ must be finite. 
\end{proof}

\section{Stability and instability of standing waves}\label{sec.stinst}

In this section we will establish some stability and instability results for the ground states obtained in Theorem  \ref{thm:existenceGSJgeral}. As we saw in section \ref{sec.gsbu} the ground states solutions of system \eqref{OB1} play a crucial role in the dynamics of \eqref{system1}. So, here we will be interested in studying their stability/instability.

In the  $L^{2}$-subcritical case, $1\leq n\leq 3$, by using the \textit{concentration-compactness method} we will prove stability results. On the other hand, for the $L^{2}$-critical, $n=4$, and $L^{2}$- supercritical, $n=5$, we use the  blowing up solutions to prove the instability of the standing waves.

\subsection{Stability} 

This subsection is devoted to prove our stability results. Throughout the subsection we assume  $1\leq n \leq 3$ and
\begin{equation*}
\omega=1 \quad \mbox{and} \quad \beta_{k}=0,\; k=1,\ldots,l.
\end{equation*}
This means, as we observed in Remark \ref{rembetomega}, we are interested in the stability of the set $\mathcal{G}(1,\boldsymbol{0})$. Our main theorem here reads as follows.

\begin{teore} \label{thm:stabG} Let $1\leq n\leq 3$. 
	Let  $\mathcal{G}(1,\boldsymbol{0})$ be the set of ground states of \eqref{OB1}. Then  $\mathcal{G}(1,\boldsymbol{0})$ is stable in   $\mathbf{H}^{1}$ in the following sense: for every $\epsilon>0$, there exist $\delta>0$ such that if 
	\begin{equation*}
	\inf_{\phib\in \mathcal{G}(1,\boldsymbol{0})}\| \mathbf{u}_0-\phib\|_{\mathbf{H}^{1}}<\delta,
	\end{equation*}
	then the global solution  of  \eqref{system1} given by Theorem \ref{thm:globalwellposH1}, with $\mathbf{u}(0)=\mathbf{u}_0$ satisfies
	\begin{equation*}
	\inf_{\phib\in \mathcal{G}(1,\boldsymbol{0})}\|\mathbf{u}(t) -\phib\|_{\mathbf{H}^{1}}<\epsilon,
	\end{equation*}
	for all $t\in \R$.
\end{teore}

Our goal throughout this subsection is to prove Theorem \ref{thm:stabG}.  To begin with, recall the energy functional in \eqref{conservationenergy} becomes
\begin{equation*}
    E(\phib)=K(\phib)-2P(\phib).
\end{equation*}
where, as before, the functionals $K$ and $P$ are given by
\begin{equation*}
    K(\phib)=\sum_{k=1}^{l}\gamma_{k}\|\nabla \phi_{k}\|^{2}_{L^{2}}\quad \mbox{and}\quad P(\phib)=\int F(\phib)\;dx.
\end{equation*}

For any $\nu>0$, let us consider the subset of $\mathcal{P}$ defined by
\begin{equation*}
\Gamma_{\nu}=\left\{\phib\in \mathcal{P};\,Q(\phib)=\nu\right\},
\end{equation*}
where
\begin{equation} \label{newQ}
    Q(\phib)=\sum_{k=1}^{l}\frac{\alpha_{k}^{2}}{\gamma_{k}}\| \phi_{k}\|^{2}_{L^{2}}.
\end{equation}
Let $A_{\nu}$ be the set  of all solutions of the minimization problem
\begin{equation}\label{problE}
I_{\nu}=\inf\{E(\phib):\phib\in \Gamma_{\nu}\},
\end{equation}
that is,
\begin{equation*}
A_\nu=\{ \phib\in\mathbf{H}^1; \,E(\phib)=I_\nu\;\; \mbox{and}\;\;Q(\phib)=\nu \}.
\end{equation*}

In what follows we will show that such a set is nonempty. As usual, we say that $(\phib_m)$ is a \textit{minimizing sequence} of \eqref{problE} if $\phib_m\in\Gamma_{\nu}$ and $E(\phib_m)$ converges to $I_\nu$.

\begin{obs}
	It is easily seen that if $(\phib_m)$ is a minimizing sequence so is $(|\phib_{m}|)$. In particular, without loss of generality, we can always (and will) assume that minimizing sequences are non-negative. 
\end{obs}

Next, define the sequence of non-decreasing  functions $M_{m}:[0,\infty)\to [0,\nu]$  by
 \begin{equation}\label{defMm}
     M_{m}(r)=\sup_{y\in \R^{n}}\int_{\{|x-y|<r\}}\sum_{k=1}^{l}\frac{\alpha_{k}^{2}}{\gamma_{k}}|\phi_{km}|^{2}\;dx.
 \end{equation}
Being uniformly bounded, this sequences converges (up to a subsequence) to a non-decreasing function $M:[0,\infty)\to [0,\nu]$. By  defining
\begin{equation}\label{defalpha}
    \alpha:=\lim_{r\to \infty}M(r),
\end{equation}
we have three possibilities: $\alpha=0$ (\textit{vanishing}), $0<\alpha<\nu$ (\textit{dichotomy}), and $\alpha=\nu$ (\textit{compactness}). The idea of the  concentration-compactness method is to show that vanishing and dichotomy cannot occur. To do so, we follow closely the arguments in  \cite{Albert} (see also \cite{Lions1} and \cite{Lions2}).

Let us start with some properties of $I_{\nu}$ and the minimizing sequences of \eqref{problE}. The first result states that $I_{\nu}$ is finite and negative.

\begin{lem}\label{Inunegat}
 For any $\nu>0$, we have $-\infty<I_{\nu}<0$. 
\end{lem}
\begin{proof} We divide the proof in three steps.\\ 

\noindent Step 1. $\Gamma_{\nu}\neq \emptyset$. 

In fact, given  $\mathbf{v}\in \mathcal{P}$, define 
\begin{equation*}
    \phib(x)=\sqrt{\frac{\nu}{Q(\mathbf{v})}}\mathbf{v}(x).
\end{equation*}
An application of Lemma \ref{functrans} gives
$$Q(\phib)=\nu\quad\mbox{and}\quad P(\phib)=\left(\frac{\nu}{Q(\mathbf{v})}\right)^{3/2}P(\mathbf{v})>0.$$
Hence, $\phib\in \Gamma_{\nu}$.\\

\noindent Step 2. $I_{\nu}<0$.  

Fix any $\phib\in \Gamma_{\nu}$. For $\lambda>0$ define
\begin{equation*}
    \phib^{\lambda}(x)=\lambda^{n/2}\left(\delta_{\frac{1}{\lambda}}\phib\right)(x).
\end{equation*}
 Lemma  \ref{functrans} again implies
\begin{equation*}
Q(\phib^{\lambda})=Q(\phib)=\nu\qquad\mbox{and}\qquad P(\phib^{\lambda})=\lambda^{n/2}P(\phib)>0,
\end{equation*}
which means that  $\phib^{\lambda}\in\Gamma_{\nu}$, for any $\lambda>0$.

Now it is easy to see that the function 
$$ 
\R^{+}\ni \lambda\mapsto f(\lambda):=E(\phib^{\lambda})=\lambda^{2}K(\phib)-2\lambda^{n/2}P(\phib)
$$ 
attains its unique  minimum  at the point $\displaystyle \lambda_{*}=\left[\frac{2K(\phib)}{nP(\phib)}\right]^{\frac{2}{n-4}}>0$. In particular, $f(\lambda_{*})=\lambda_{*}^{2}\left(\frac{n-4}{n}\right)K(\phib)<0$ and
\begin{equation*}
    I_{\nu}\leq E(\phib^{\lambda_{*}})=f(\lambda_{*})<0.
\end{equation*}
which concludes Step 2.\\

\noindent Step 3. $I_{\nu}>-\infty$. 

Fix any $\phib\in \Gamma_{\nu}$. 
From Gagliardo-Nirenberg inequality (Corollary \ref{corollarybestconstant}) and Young's inequality with $\epsilon$,
\begin{equation*}
    \begin{split}
       P(\phib) \leq CQ(\phib)^{\frac{3}{2}-\frac{n}{4}}K(\phib)^{\frac{n}{4}}
       =C\nu^{\frac{3}{2}-\frac{n}{4}}K(\phib)^{\frac{n}{4}}
       \leq \epsilon K(\phib)+C,\qquad\epsilon>0,
    \end{split}
\end{equation*}
where $C=C(\epsilon,\nu)$. Thus, 
\begin{equation}\label{6.6}
    \begin{split}
       E(\phib)=K(\phib)-2P(\phib)
       \geq (1-2\epsilon)K(\phib)-C
       \geq  -C,
    \end{split}
\end{equation}
provided that $0<\epsilon<1/2$.
Since $\phib$ is arbitrary the claim follows and the proof is completed.  
\end{proof}

Next lemma establishes that, every minimizing sequence of \eqref{problE} is bounded in $\mathbf{H}^{1}$ and the real sequence $(P(\phib_m))$ is bounded from below for $m$ sufficiently large.

\begin{lem}\label{minseqbon}
 If $(\phib_m)$ is a minimizing sequence of \eqref{problE}, then there exist constants $B>0$ and $\delta>0$ such that
 \begin{enumerate}
     \item[(i)] $\|\phib_m\|_{\mathbf{H}^{1}}\leq B$, for all $ m\in \N$, and 
     \item[(ii)] $P(\phi_m)\geq \delta$ for all sufficiently large $m$.
 \end{enumerate}
\end{lem}

\begin{proof}
 Since $(\phib_m)$ is a minimizing sequence  we have
 \begin{equation*}\label{EQmin}
    \lim_{m\to \infty}E(\phib_m)=I_{\nu} \quad \mbox{and} \quad  Q(\phib_m)=\nu,\quad \forall m\in \N.
 \end{equation*}
In particular, $(\phib_m)$  is bounded in  $\mathbf{L}^{2}$. In addition, from \eqref{6.6} there exist positive constants $\tau$ and $C$ such that
 \begin{equation*}
  \tau K(\phib_m) \leq   E(\phib_m) +C,
 \end{equation*}
 This, combined with the fact  that $ (E(\phib_m))$ is a bounded sequence yield (i).
 
Now we prove (ii). Since $I_{\nu}<0$, we have $E(\phib_m)\leq {I_{\nu}}/{2}$, for $m$ large enough.
 Thus, 
 \begin{equation*}
     \begin{split}
         P(\phib_m)=-\frac{1}{2} E(\phib_m)+\frac{1}{2}K(\phib_m)
         \geq -\frac{I_{\nu}}{4}+\frac{1}{2}K(\phib_m)
         \geq  -\frac{I_{\nu}}{4},
     \end{split}
 \end{equation*}
for $m$ large enough. By taking $\delta=-{I_{\nu}}/{4}>0$ we conclude the proof. 
\end{proof}

Next we prove the subadditivity of $I_\nu$. More precisely,

\begin{lem}\label{subadit}
 For all $\nu_{1}, \nu_{2}>0$ we have
 \begin{equation*}
     I_{\nu_{1}+ \nu_{2}}<I_{\nu_{1}}+I_{ \nu_{2}}. 
 \end{equation*}
\end{lem}

\begin{proof} We proceed in three steps.\\
Step 1. If $\theta>0$ and $\nu>0$ then $I_{\theta \nu}=\theta ^{\frac{6-n}{4-n}}I_{\nu}$. \\

In fact, given any $\phib\in\mathbf{H}^1$, define
\begin{equation*}
    \phib^{\theta}(x)=\theta^{\frac{2}{4-n}} \left( \delta_\lambda\phib\right)\left(x\right),\qquad \lambda=\theta^{-\frac{1}{4-n}}.
\end{equation*}
From Lemma \ref{functrans} we then infer
\begin{eqnarray*}
Q(\phib^{\theta})=\theta Q(\phib), \quad
  K(\phib^{\theta})=\theta^{\frac{6-n}{4-n}} K(\phib),\quad P(\phib^{\theta})=\theta^{\frac{6-n}{4-n}} P(\phib),
\end{eqnarray*}
from which we deduce that the sets $\{ E(\phib); \, \phib\in \Gamma_{\theta\nu} \}$ and $\{\theta^{\frac{6-n}{4-n}} E(\phib); \, \phib\in \Gamma_{\nu} \}$ are the same.
Hence,
\begin{equation*}
    \begin{split}
       I_{\theta \nu} =\inf\{ E(\phib); \, \phib\in \Gamma_{\theta\nu} \}
       =\inf \{\theta^{\frac{6-n}{4-n}} E(\phib); \, \phib\in \Gamma_{\nu} \}
       =  \theta^{\frac{6-n}{4-n}}I_{ \nu}.
    \end{split}
\end{equation*}

\noindent Step 2. For $1\leq n\leq3$, we have $\nu_{1}^{\frac{6-n}{4-n}}+\nu_{2}^{\frac{6-n}{4-n}}<(\nu_{1}+\nu_{2})^{\frac{6-n}{4-n}}$. 

Observe that the cases $n=2$ and $n=3$ are immediate. For $n=1$ we have to prove that
\begin{equation} \label{desn=1}
    \nu_{1}^{\frac{5}{3}}+\nu_{2}^{\frac{5}{3}}<(\nu_{1}+\nu_{2})^{\frac{5}{3}}.
\end{equation}
 Without loss of generality we may assume $\nu_{1}< \nu_{2}$. By dividing both sides of \eqref{desn=1} by $\nu_2^{\frac{5}{3}}$ we see that it suffices to prove that 
$$f(x)=\frac{(1+x)^{\frac{5}{3}}}{1+x^\frac{5}{3}}>1, \qquad 0<x<1.$$
Since  $f'(x)>0$ if $0<x<1$, $f$ is an increasing function on $0<x<1$. In particular, $1=f(0)<f(x)$, which is the desired conclusion.\\

\noindent Step 3. $I_{\nu_{1}+\nu_{2}}<I_{\nu_{1}}+I_{\nu_{2}}$.  

Lemma \ref{Inunegat} yields  $I_{1}<0$. Thus, using Steps 1 and 2 above
\begin{equation*}
    I_{\nu_{1}+\nu_{2}}=(\nu_{1}+\nu_{2})^{\frac{6-n}{4-n}}I_{1}<\nu_{1}^{\frac{6-n}{4-n}}I_1+\nu_{2}^{\frac{6-n}{4-n}}I_{1}=I_{\nu_{1}}+I_{\nu_{2}},
\end{equation*}
which completes the proof. 
\end{proof}

\subsubsection{Ruling out vanishing} Here we prove that the case $\alpha =0$ cannot occur. We start with the following property.

\begin{lem}\label{supintF}
Let $B>0$ and $\delta>0$ be given. There exists a constant $\eta=\eta(B,\delta)>0$ such that if $\phib\in \mathbf{H}^{1}$  satisfies $\phib\geq0$, $\|\phib\|_{\mathbf{H}^{1}}\leq B$ and $P(\phib)\geq \delta$, then 
\begin{equation*}
    \sup_{y\in \R^{n}}\int_{\{|x-y|<1/2\}}F(\phib)\;dx\geq \eta.
\end{equation*}
\end{lem}
\begin{proof}
Without loss of generality we assume that $\mathbf{H}^{1}$ is endowed with the equivalent norm
$\|\phib\|_{\mathbf{H}^{1}}^2=K(\phib)+Q(\phib)$.
Let $(Q_{j})_{j\geq 0}$ be a sequence of open cubes in $\R^{n}$,  with side length $\displaystyle \frac{1}{\sqrt{2}}$, such that $Q_{j}\cap Q_{k}=\emptyset$ if $j\neq k$ and $\overline{\bigcup_{j\geq 0}Q_{j}}=\R^{n}$. Denote by $x_{j}$ the center of each cube. It follows that
\begin{equation*}
\begin{split}
\sum_{j=0}^{\infty}\int_{Q_{j}}\sum_{k=1}^{l}\left(\gamma_{k}|\nabla \phi_{k}|^{2}+\frac{\alpha_{k}^{2}}{\gamma_{k}}|\phi_{k}|^{2}\right)\;dx=\|\phib\|_{\mathbf{H}^{1}}^{2}
\leq\frac{B^{2}}{\delta}P(\phib)
=\frac{B^{2}}{\delta}\sum_{j=0}^{\infty}\int_{Q_{j}}F(\phib)\;dx.
\end{split}
\end{equation*}
The last inequality implies that there exist $j_{0}\geq0$ such that
\begin{equation}\label{normH1}
\begin{split}
   \|\phib\|_{\mathbf{H}^{1}(Q_{j_{0}})}^{2}=\int_{Q_{j_{0}}}\sum_{k=1}^{l}\left(\gamma_{k}|\nabla \phi_{k}|^{2}+\frac{\alpha_{k}^{2}}{\gamma_{k}}|\phi_{k}|^{2}\right)\;dx
   \leq \frac{B^{2}}{\delta}\int_{Q_{0}}F(\phib)\;dx.
   \end{split}
\end{equation}
On the other hand, the Gagliardo-Nirenberg inequality on bounded domains (see, for instance, \cite[Theorem 5.8]{Adams}) gives
\begin{equation}\label{normH2}
    \begin{split}
        \int_{Q_{j_{0}}}F(\phib)\;dx\leq C\|\phib\|_{\mathbf{L}^{2}(Q_{j_{0}})}^{3-\frac{n}{2}}\|\phib\|_{\mathbf{H}^{1}(Q_{j_{0}})}^{\frac{n}{2}}
        \leq C \|\phib\|_{\mathbf{H}^{1}(Q_{j_{0}})}^{3}.
    \end{split}
\end{equation}
 Inequalities \eqref{normH1} and \eqref{normH2} show that
\begin{equation*}
 \int_{Q_{j_{0}}}F(\phib)\;dx\leq    \frac{CB^{3}}{\delta^{\frac{3}{2}}} \left[\int_{Q_{j_{0}}}F(\phib)\;dx\right]^{\frac{3}{2}},
\end{equation*}
which leads to
\begin{equation*}
   \int_{Q_{j_{0}}}F(\phib)\;dx\geq \frac{\delta^{3}}{C^{2}B^{6}}.
\end{equation*}
Let $B_{1/2}(x_{j_{0}})$ be the ball centered in $x_{j_{0}}$ and radius $\displaystyle \frac{1}{2}$. Since $Q_{j_{0}} \subset B_{1/2}(x_{j_{0}})$ and $F(\phib)\geq0$ (see Lemma \ref{fkreal}),
\begin{equation*}
    \sup_{y\in \R^{n}}\int_{\{|x-y|<1/2\}}F(\phib)\;dx \geq \int_{B_{1/2}(x_{j_{0}})} F(\phib)\;dx\geq  \int_{Q_{j_{0}}}F(\phib)\;dx \geq \eta,
\end{equation*}
where $\eta= \displaystyle \frac{\delta^{3}}{C^{2}B^{6}} $, which proves the lemma.
\end{proof}

\begin{lem}\label{outvanishi}
 For every minimizing sequence of \eqref{problE} we have $\alpha>0$. In particular,  \textit{vanishing} cannot occur.
\end{lem}
\begin{proof}
From Lemmas \ref{supintF} and \ref{minseqbon} we can find $\eta>0$ and a sequence $(y_{m})\subset \R^{n}$ such that 
\begin{equation*}
   \int_{\{|x-y_{m}|<1/2\}}F(\phib_m)\;dx\geq \eta,\qquad \forall m\in \N.
\end{equation*}
Thus, Lemma \ref{estdifF},  a  change of variables and  the Gagliardo-Nirenberg inequality on bounded domains   give
\begin{equation*}
    \begin{split}
        \eta &\leq \int_{\{|x-y_{m}|<1/2\}}F(\phib_m)\;dx
        \leq C\int_{\{|x-y_{m}|<1/2\}}\sum_{j=1}^{l}|\phi_{jm}(x)|^{3}\;dx\\
         &= C\int_{\{|z|<1/2\}}\sum_{j=1}^{l}|\phi_{jm}(z+y_{m})|^{3}\;dz\\
        &\leq C\left(\int_{\{|z|<1/2\}}\sum_{k=1}^{l}\frac{\alpha_{k}^{2}}{\gamma_{k}}| \phi_{km}(z+y_{m})|^{2}\;dz\right)^{\frac{3}{2}-\frac{n}{4}}\|\phib_m(\cdot+y_m)\|_{\mathbf{H}^1(B_{1/2}(0))}^\frac{n}{2}
    \end{split}
\end{equation*}
where ${C}$ is a constant  depending on the ball $B_{1/2}(0)$ but independent of $m$. Now, by  using Lemma \ref{minseqbon} and the definition of $M_{m}$ in \eqref{defMm} we conclude that
\begin{equation*}
    \begin{split}
      \eta &\leq {C}B^{\frac{n}{2}}\left(\int_{\{|x-y_{m}|<1/2\}}\sum_{k=1}^{l}\frac{\alpha_{k}^{2}}{\gamma_{k}}| \phi_{km}|^{2}\;dx\right)^{\frac{6-n}{4}}
    \leq C\left[M_{m}\left(1/2\right)\right]^{\frac{6-n}{4}},
    \end{split}
\end{equation*}
where $C$ is an  universal constant. 
Taking the limit as $m\to\infty$ in this last inequality we deduce $\eta \leq C\left[M\left(1/2\right)\right]^{\frac{6-n}{4}}$, that is, $ M\left(1/2\right)\geq \left(\frac{\eta}{C}\right)^{\frac{4}{6-n}}$.  Consequently, since $M(r)$ is an increasing function,
\begin{equation*}
    \alpha=\lim_{r\to \infty}M(r)\geq M\left(1/2\right)>0,
\end{equation*}
which is the desired conclusion.
\end{proof}

\subsubsection{Ruling out dichotomy}
Here we show that the case $0<\alpha<\nu$ does not occur.  The main tool to obtain this is the following result.
 
 \begin{lem}\label{decompseq}
Let $(\phib_m)$ be a minimizing sequence of \eqref{problE}. Then, for every $\epsilon>0$, there exist $m_{0}\in \N$ and sequences of functions $(\mathbf{v}_m)_{m\geq m_{0}}$ and $(\mathbf{w}_m)_{m\geq m_{0}}$ in $\mathcal{P}$    such that  for every $m\geq m_{0}$,
\begin{enumerate}
    \item[(i)] $|Q(\mathbf{v}_m)-\alpha|<\epsilon $.
    \item[(ii)] $|Q(\mathbf{w}_m)-(\nu-\alpha)|<\epsilon $.
    \item[(iii)] $E(\phib_m)\geq E(\mathbf{v}_m)+E(\mathbf{w}_m)-\epsilon$. 
\end{enumerate}
\end{lem}

\begin{proof}
Let $\epsilon>0$ be given. Since $\lim_{r\to \infty}M(r)=\alpha$, there exist $r_{1}>1$ such that if $r\geq r_{1}$ then $|M(r)-\alpha|<\frac{\epsilon}{2}$. 
Thus, from the fact that $M$ is non-decreasing we conclude that, for $r\geq r_{1}$,
\begin{equation}\label{desMalp}
    \alpha-\frac{\epsilon}{2}<M(r)\leq M(2r)\leq \alpha.
\end{equation}
Fix some  $r$ satisfying $r\geq r_1$. From the pointwise convergence of $M_{m}$ to $M$ we can find $m_{0}(r)\in \N$ such that if $m\geq m_{0}(r)$ then
\begin{equation}\label{desMm}
    M(r)-\frac{\epsilon}{2}<M_{m}(r)\qquad\mbox{and}\qquad M_{m}(2r)<M(2r)+ \epsilon.
\end{equation}
By combining \eqref{desMalp} and \eqref{desMm} we infer
 \begin{equation*}
     \alpha-\epsilon <M_{m}(r)\leq M_{m}(2r)<\alpha+\epsilon,\qquad \forall m\geq m_{0}(r).
 \end{equation*}
This means that, for each $m\geq m_{0}(r)$, there exists $y_{m}\in\R^n$ such that
  \begin{equation}\label{alp-eps}
        \int_{\{|x-y_{m}|<r\}}\sum_{k=1}^{l}\frac{\alpha_{k}^{2}}{\gamma_{k}}|\phi_{km}|^{2}\;dx >\alpha -\epsilon,
     \end{equation}
and
 \begin{equation}\label{alp+eps}
        \int_{\{|x-y_{m}|<2r\}}\sum_{k=1}^{l}\frac{\alpha_{k}^{2}}{\gamma_{k}}|\phi_{km}|^{2}\;dx <\alpha +\epsilon.
     \end{equation}
     
Now, choose $\varphi\in C_{0}^{\infty}(B_{2}(0))$ such that $\varphi\geq0$ and $\varphi\equiv 1$ on $B_{1}(0)$ and let $\vartheta \in C^{\infty}(\R^{n})$ be such that
\begin{equation}\label{ballphitheta}
    \varphi^{2}+\vartheta^{2}=1.
\end{equation}
 Define
\begin{equation*}
    \mathbf{v}_m(x)=\varphi_{r}\left(x-y_{m}\right)\phib_m(x)\qquad\mbox{and} \qquad \mathbf{w}_m(x)=\vartheta_{r}\left(x-y_{m}\right)\phib_m(x),
\end{equation*}
where  $\varphi_{r}(x)=(\delta_{r}\varphi)(x)$ and  $\vartheta_{r}(x)=(\delta_{r}\vartheta)(x)$.

We are going to prove that $\mathbf{v}_m$ and $\mathbf{w}_m$ satisfy the desired conclusions.  Indeed, by \eqref{alp+eps},
\begin{equation}\label{Qveps+}
 Q(\mathbf{v}_m)=  \int_{\{|x-y_{m}|<2r\}}\sum_{k=1}^{l}\frac{\alpha_{k}^{2}}{\gamma_{k}}|v_{km}|^{2}\;dx
       \leq \int_{\{|x-y_{m}|<2r\}}\sum_{k=1}^{l}\frac{\alpha_{k}^{2}}{\gamma_{k}}|\phi_{km}|^{2}\;dx
       <\alpha+\epsilon.
\end{equation}
On the other hand, by \eqref{alp-eps},
\begin{equation}\label{Qvesp-}
\begin{split}
     Q(\mathbf{v}_m)& \geq \int_{\{|x-y_{m}|<r\}}\sum_{k=1}^{l}\frac{\alpha_{k}^{2}}{\gamma_{k}}|v_{km}|^{2}\;dx
       =\int_{\{|x-y_{m}|<r\}}\sum_{k=1}^{l}\frac{\alpha_{k}^{2}}{\gamma_{k}}|\phi_{km}|^{2}\;dx >\alpha-\epsilon.
\end{split}
\end{equation}
Hence, from \eqref{Qveps+} and \eqref{Qvesp-} we obtain (i). 

To prove (ii) we first note that, by \eqref{ballphitheta}, 
\begin{equation}\label{Qweps-}
\begin{split}
     Q(\mathbf{w}_m)&=\int\sum_{k=1}^{l}\frac{\alpha_{k}^{2}}{\gamma_{k}}|\phi_{km}|^{2}\;dx-\int\sum_{k=1}^{l}\frac{\alpha_{k}^{2}}{\gamma_{k}}|\varphi_{r}|^{2}|\phi_{km}|^{2}\;dx\\
     &=\nu-  \int_{\{|x-y_{m}|<2r\}}\sum_{k=1}^{l}\frac{\alpha_{k}^{2}}{\gamma_{k}}|\varphi_{r}|^{2}|\phi_{km}|^{2}\;dx\\
     &\geq\nu-  \int_{\{|x-y_{m}|<2r\}}\sum_{k=1}^{l}\frac{\alpha_{k}^{2}}{\gamma_{k}}|\phi_{km}|^{2}\;dx\\
     & >\nu-(\alpha+\epsilon),
\end{split}
\end{equation}

where we have used \eqref{alp+eps} in the last inequality. Also, in view of \eqref{alp-eps},
\begin{equation}\label{Qweps+}
\begin{split}
     Q(\mathbf{w}_m)&=\nu-  \int_{\{|x-y_{m}|<2r\}}\sum_{k=1}^{l}\frac{\alpha_{k}^{2}}{\gamma_{k}}|\varphi_{r}|^{2}|\phi_{km}|^{2}\;dx\\
     &\leq\nu-  \int_{\{|x-y_{m}|<r\}}\sum_{k=1}^{l}\frac{\alpha_{k}^{2}}{\gamma_{k}}|\phi_{km}|^{2}\;dx\\
      & <\nu-(\alpha-\epsilon).
\end{split}
\end{equation}
A combination of   \eqref{Qweps-} and \eqref{Qweps+} yields (ii).

It remains to establish (iii). Note that
\begin{equation*}\label{estgradvarp}
\left\|\nabla[\varphi_{r}\left(x-y_{m}\right)]\right\|_{L^{\infty}}= \frac{1}{r}\left\|\nabla\varphi\left(\frac{x-y_{m}}{r}\right)\right\|_{L^{\infty}}=\frac{1}{r}\left\|\nabla\varphi\right\|_{L^{\infty}}
\end{equation*}
and for each component of $\mathbf{v}_m$,
\begin{equation*}
   \nabla v_{km}(x) =\frac{1}{r}\nabla\varphi\left(\frac{x-y_m}{r}\right)\phi_{km}(x)+\varphi_{r}(x-y_m)\nabla \phi_{km}(x).
\end{equation*}
Hence,   by Young's inequality, and the fact that $\frac{1}{r^2}\leq\frac{1}{r}$,
\begin{equation*}
    \begin{split}
         |\nabla v_{km}|^{2}
         &\leq \frac{1}{r^{2}}\left\|\nabla\varphi\right\|^{2}_{L^{\infty}}|\phi_{km}|^{2}+\frac{2}{r}\left\|\nabla\varphi\right\|_{L^{\infty}}\left\|\varphi\right\|_{L^{\infty}}|\phi_{km}||\nabla \phi_{km}|+|\varphi_{r}|^{2}|\nabla\phi_{km}|^{2}\\
         &\leq  \frac{C}{r} \left(|\phi_{km}|^{2}+|\nabla \phi_{km}|^2\right)+|\varphi_{r}|^{2}|\nabla\phi_{km}|^{2},
    \end{split}
\end{equation*}
which implies that
\begin{equation}\label{kkvm}
    \begin{split}
       K(\mathbf{v}_m) =\sum_{k=1}^{l}\gamma_{k}\|\nabla v_{km}\|^{2}_{L^{2}}
       \leq  \frac{C}{r}\|\phib_m\|^2_{\mathbf{H}^1}
+\sum_{k=1}^{l}\gamma_{k}\int|\varphi_{r}|^{2}|\nabla\phi_{km}|^{2}\;dx.
    \end{split}
\end{equation}
By recalling  that any minimizing sequence is bounded in $\mathbf{H}^{1}$ (see Lemma \ref{minseqbon}), \eqref{kkvm}  then  yields
\begin{equation}\label{Kv}
     K(\mathbf{v}_m)\leq \sum_{k=1}^{l}\gamma_{k}\int|\varphi_{r}|^{2}|\nabla\phi_{km}|^{2}\;dx+O\left(\frac{1}{r}\right).
\end{equation}
In a similar fashion
\begin{equation}\label{Kw}
     K(\mathbf{w}_m)\leq \sum_{k=1}^{l}\gamma_{k}\int|\vartheta_{r}|^{2}|\nabla\phi_{km}|^{2}\;dx+O\left(\frac{1}{r}\right). 
\end{equation}
Combining \eqref{Kv}, \eqref{Kw} and \eqref{ballphitheta} we deduce that
\begin{equation}\label{Kv+KW}
    \begin{split}
      K(\mathbf{v}_m)+   K(\mathbf{w}_m)&\leq \sum_{k=1}^{l}\gamma_{k}\int(|\varphi_{r}|^{2}+|\vartheta_{r}|^{2})|\nabla\phi_{km}|^{2}\;dx+O\left(\frac{1}{r}\right)\\
      &\leq K(\phib_m)+O\left(\frac{1}{r}\right).
    \end{split}
\end{equation}

Now, since $F$ is  homogeneous of 3 three, we see that
\begin{equation}\label{Pv}
    P(\mathbf{v}_m)=\int\varphi_{r}^{3}F(\phib_m)\;dx \quad \mbox{and}\quad  P(\mathbf{w}_m)=\int\vartheta_{r}^{3}F(\phib_m)\;dx.
\end{equation}
In particular, Lemma \ref{fkreal} implies that $\mathbf{v}_m,\mathbf{w}_m\in \mathcal{P}$ (recall we are assuming that minimizing sequences are non-negative).

Let $\Omega_{r}=\{r<|x-y_{m}|<2r\}$. The Gagliardo-Nirenberg inequality  gives
\begin{equation}\label{PVV}
\begin{split}
\|\phi_{jm}\|_{L^{3}(\Omega_{r})}^3&\leq C\| \phi_{jm}\|_{H^{1}(\Omega_{r})}^{\frac{n}{2}}\| \phi_{jm}\|_{L^{2}(\Omega_{r})}^{3-\frac{n}{2}}\\
&\leq C\| \phib_{m}\|_{H^{1}(\Omega_{r})}^{\frac{n}{2}} \left[\int_{\Omega_{r}}\sum_{k=1}^{l}\frac{\alpha_{k}^{2}}{\gamma_{k}}| \phi_{km}|^{2}\;dx\right]^{\frac{3}{2}-\frac{n}{4}},
\end{split}
\end{equation}
where $C$ is independent of $m$. Now, taking the sum over $j$ on the left-hand side of \eqref{PVV}, using Lemma \ref{minseqbon} and inequalities \eqref{alp-eps}-\eqref{alp+eps} we get
\begin{equation}\label{GNIBdom}
\begin{split}
    \sum_{j=1}^{l}\|\phi_{jm}\|_{L^{3}(\Omega_{r})}^{3}
    \leq CB^{\frac{n}{2}}(2\epsilon)^{\frac{3}{2}-\frac{n}{4}}\leq C\epsilon^{\frac{3}{2}-\frac{n}{4}}.
\end{split}
\end{equation}
Hence, from Lemma \ref{estdifF}, \eqref{GNIBdom} and \eqref{Pv},
\begin{equation}\label{F-Pv}
    \begin{split}
        \left|\int\varphi_{r}^{2}F(\phib_m)\;dx-P(\mathbf{v}_m)\right|
        \leq \int|\varphi_{r}^{2}-\varphi_{r}^{3}||F(\phib_m)|\;dx
        \leq C\sum_{j=1}^{l}\int_{\Omega_{r}}|\phi_{jm}|^{3}\;dx
        \leq C\epsilon^{\frac{3}{2}-\frac{n}{4}}.
    \end{split}
\end{equation}

A similar argument also shows that
\begin{equation}\label{F-Pw}
   \left|\int\vartheta_{r}^{2}F(\phib_m)\;dx-P(\mathbf{w}_m)\right|\leq    C\epsilon^{\frac{3}{2}-\frac{n}{4}}.
\end{equation}
Therefore, using \eqref{Kv+KW}, \eqref{GNIBdom}, \eqref{F-Pv} and \eqref{F-Pw} we see that 
\begin{equation*}
    \begin{split}
        E(\mathbf{v}_m)+E(\mathbf{w}_m)&=K(\mathbf{v}_m)+K(\mathbf{w}_m)
        -2P(\mathbf{v}_m)-2P(\mathbf{w}_m)\\
        &\leq K(\phib_m)+O\left(\frac{1}{r}\right)
        -2\int\left(\varphi_{r}^{2}+\vartheta_{r}^{2}\right)F(\phib_m)\;dx\\
        &\quad+2\int\varphi_{r}^{2}F(\phib_m)\;dx-2P(\mathbf{v}_m)
        +2\int\vartheta_{r}^{2}F(\phib_m)\;dx-2P(\mathbf{w}_m)\\
        &\leq E(\phib_m)+O\left(\frac{1}{r}\right)\\
        &\quad+ 2\left|\int\varphi_{r}^{2}F(\phib_m)\;dx-P(\mathbf{v}_m)\right|
       +2\left|\int\vartheta_{r}^{2}F(\phib_m)\;dx-P(\mathbf{w}_m)\right|\\
        &\leq E(\phib_m)+O\left(\frac{1}{r}\right)+C\epsilon^{\frac{6-n}{4}}.
    \end{split}
\end{equation*}

We now can take $r$ sufficiently large such that $O\left(\frac{1}{r}\right)<\epsilon^{\frac{6-n}{4}}$. As a consequence,  for $m\geq m_{0}(r)$,
\begin{equation*}
    E(\phib_m)\geq E(\mathbf{v}_m)+E(\mathbf{w}_m) -(C+1)\epsilon^{\frac{6-n}{4}}.
\end{equation*}
By noting  we can take $\tilde{\epsilon}=\min\left\{\epsilon,\left(\frac{\epsilon}{C+1}\right)^{\frac{4}{6-n}}\right\}$ instead of $\epsilon$ at the beginning of the proof we may repeat the same arguments as above and establish (iii).
\end{proof}

 Finally, the fact that dichotomy cannot occur is a consequence of 
 Lemma \ref{subadit} and the following result.
 
\begin{lem}\label{outdichot}
If $0<\alpha<\nu$ then $I_{\nu}\geq I_{\alpha}+I_{\nu-\alpha}$.  In particular,  \textit{dichotomy} cannot occur.
\end{lem}
\begin{proof}
Let us start by fixing some $\epsilon<\frac{\alpha}{2}$. We claim if
 $\phib\in \mathcal{P}$ satisfies $|Q(\phib)-\alpha|<\epsilon$
 then the number  $\displaystyle\beta=\sqrt{\frac{\alpha}{Q(\phib)}}$ satisfies
\begin{equation}\label{Qbetaphi}
  |\beta-1|<C\epsilon,
\end{equation}
where $C$ is a constant independent of $\epsilon$ and $\phib$. In fact, since $\frac{\alpha}{2}< Q(\phib)$, we have
\begin{equation*}
        |\beta-1|=|\beta^{2}-1|\frac{1}{\beta+1}
        <\left|\frac{\alpha}{Q(\phib)}-1\right|
        =|\alpha-Q(\phib)|Q(\phib)^{-1}
        <\frac{2}{\alpha}\epsilon.
\end{equation*}
So, we can take $C=\frac{2}{\alpha}$ and the claim is proved. 

Now since $Q(\beta\phib)=\alpha$ and $P(\beta\phib)=\beta^{3}P(\phib)>0$, we conclude that $I_{\alpha}\leq E(\beta\phib)$. But,
\begin{equation*}\label{Ebeta}
        E(\beta\phib)=\beta^{2}K(\phib)-2\beta^{3}P(\phib)
        =E(\phib)+(\beta^{2}-1)K(\phib)
        +2(1-\beta^{3})P(\phib).
\end{equation*}
By using \eqref{Qbetaphi} and the facts that $K(\phib)\leq C \|\boldsymbol{\phi}\|^{2}_{\mathbf{H}^{1}}$ and  $P(\phib)\leq C \|\boldsymbol{\phi}\|^{3}_{\mathbf{H}^{1}},
$
we infer
\begin{equation*}
\begin{split}
      E(\beta\phib)&
        \leq E(\phib)+C\epsilon(\beta+1) \|\boldsymbol{\phi}\|^{2}_{\mathbf{H}^{1}}+2C\epsilon (1+\beta+\beta^{2}) \|\boldsymbol{\phi}\|^{3}_{\mathbf{H}^{1}}.\\
        &\leq  E(\phib)+C\epsilon,
\end{split}
\end{equation*}
where we used that $\beta<\sqrt{2}$ and $C$ depends only on $\alpha$ and $\|\boldsymbol{\phi}\|_{\mathbf{H}^{1}}$. Therefore,
\begin{equation*}
    E(\phib)\geq  I_{\alpha}-C\epsilon.
\end{equation*}
If we replace $\epsilon$ by $\tilde{\epsilon}=\min\left\{\epsilon,\frac{\epsilon}{C_{2}}\right\}$ in the previous computations we can conclude that
\begin{equation}\label{Ealpaeps}
    E(\phib)\geq  I_{\alpha}-\epsilon.
\end{equation}

By using similar arguments, if we replace the number $\beta$ by $\displaystyle \tilde{\beta}=\sqrt{\frac{\nu-\alpha}{Q(\phib)}}$ we can prove if $\phib\in\mathcal{P}$ satisfies $ |Q(\phib)-(\nu-\alpha)|<\epsilon$ (for $\epsilon$ small) then
\begin{equation}\label{Enualpaeps}
    E(\phib)\geq  I_{\nu-\alpha}-\epsilon.
\end{equation}

Now, let $s\in \N$ and assume $(\phib_m)$ is a minimizing sequence of \eqref{problE}. From Lemma \ref{decompseq} we can find a subsequence, say, $(\phib_{m_s})$  and corresponding sequences $(\mathbf{v}_{m_s})$ and $(\mathbf{w}_{m_s})$ in $\mathcal{P}$ such that
\begin{equation*}
  |Q(\mathbf{v}_{m_s})-\alpha|<\frac{1}{s},\quad |Q(\mathbf{w}_{m_s})-(\nu-\alpha)|<\frac{1}{s} 
\end{equation*}
and
\begin{equation*}\label{Ephivw}
    E(\phib_{m_{s}})\geq E(\mathbf{v}_{m_s})+E(\mathbf{w}_{m_s})-\frac{1}{s}.
\end{equation*}

Thus, \eqref{Ealpaeps} and \eqref{Enualpaeps} implies that, for $s$ large enough,
\begin{equation*}
     E(\phib_{m_{s}})\geq I_{\alpha}+I_{\nu-\alpha}-\frac{3}{s}
\end{equation*}
Letting $s\to \infty$ we obtain the conclusion of the lemma.  
\end{proof}

\subsubsection{Compactness} Taking into account that vanishing and dichotomy cannot occur, the only possibility is that $\alpha=\nu$. In this case we have the followings results.

\begin{lem}\label{cocentcomp}
 Suppose $\alpha=\nu$. Then  there exists a sequence $(y_{m}) \subset \R^{n}$ such that
 \begin{enumerate}
     \item[(i)] For every $z<\alpha$ there exist $r=r(z)$ such that
     \begin{equation*}
        \int_{\{|x-y_{m}|<r\}}\sum_{k=1}^{l}\frac{\alpha_{k}^{2}}{\gamma_{k}}|\phi_{km}|^{2}\;dx >z
     \end{equation*}
     for all sufficiently large $m$.
     \item[(ii)] The sequence $(\tilde{\phib}_{m})$ defined by $ \tilde{\phib}_{m}(x)=\phib_{m}(x+y_{m})$
     has a subsequence which converges strongly in $\mathbf{H}^{1}$ to a function $\phib\in A_{\nu}$. In particular, $A_{\nu}$ is nonempty. 
 \end{enumerate}
\end{lem}
\begin{proof}
 Since $\displaystyle  \lim_{r\to \infty}M(r)=\nu$  and $\displaystyle  \lim_{m\to \infty}M_{m}(r)=M(r) $ we can find $r_0$ and $m_0$ large enough such that
$M_{m}(r_{0})>{\nu}/{2},$ for $m\geq m_0$. Therefore, for each $m\geq m_{0}$ there exist $y_{m}\in\R^{n}$ such that
\begin{equation}\label{nuboun}
       \int_{\left\{\left|x-y_{m}\right|<r_{0}\right\}}\sum_{k=1}^{l}\frac{\alpha_{k}^{2}}{\gamma_{k}}|\phi_{km}|^{2}\;dx >\frac{\nu}{2}.
     \end{equation}
     
Let be $0<z<\nu$. In view of \eqref{nuboun}, without loss of generality, we may assume $\frac{\nu}{2}<z<\nu$. By using a similar argument as above, we can find $r_{0}(z)$ and $m_{0}(z)$ such that if $m\geq m_{0}(z)$ then
\begin{equation}\label{nuboun1}
       \int_{\left\{\left|x-y_{m}(z)\right|<r_{0}(z)\right\}}\sum_{k=1}^{l}\frac{\alpha_{k}^{2}}{\gamma_{k}}|\phi_{km}|^{2}\;dx >z,
     \end{equation}
for some $y_{m}(z)\in \R^{n}$. We claim  that for $m$ large $\left\{\left|x-y_{m}\right|<r_{0}\right\}\cap \left\{\left|x-y_{m}(z)\right|<r_{0}(z)\right\}\neq \emptyset $. In fact, otherwise,
\begin{equation*}
   \begin{split}
        \nu&=\int\sum_{k=1}^{l}\frac{\alpha_{k}^{2}}{\gamma_{k}}|\phi_{km}|^{2}\;dx\\
        &\geq \int_{\left\{\left|x-y_{m}\right|<r_{0}\right\}}\sum_{k=1}^{l}\frac{\alpha_{k}^{2}}{\gamma_{k}}|\phi_{km}|^{2}\;dx+\int_{\left\{\left|x-y_{m}(z)\right|<r_{0}(z)\right\}}\sum_{k=1}^{l}\frac{\alpha_{k}^{2}}{\gamma_{k}}|\phi_{km}|^{2}\;dx \\
        &>\frac{\nu}{2}+z>\nu,
   \end{split}
\end{equation*}
 which is a contradiction. Hence, for all $\frac{\nu}{2}<z<\nu$ and $m$ large, 
$\left|y_{m}-y_{m}(z)\right|\leq r_{0}+r_{0}(z)$. By defining $r(z)=r_{0}+2r_{0}(z)$ we then see that, for $m$ large enough,
$ \left\{\left|x-y_{m}(z)\right|<r_{0}(z)\right\} \subset \left\{\left|x-y_{m}\right|<r(z)\right\}$.
Therefore, for all $\frac{\nu}{2}<z<\nu$ and $m$ large enough, we have from \eqref{nuboun1},
\begin{equation*}
     \int_{\left\{\left|x-y_{m}\right|<r(z)\right\}}\sum_{k=1}^{l}\frac{\alpha_{k}^{2}}{\gamma_{k}}|\phi_{km}|^{2}\;dx \geq  \int_{\left\{\left|x-y_{m}(z)\right|<r_{0}(z)\right\}}\sum_{k=1}^{l}\frac{\alpha_{k}^{2}}{\gamma_{k}}|\phi_{km}|^{2}\;dx >z,
\end{equation*}
which proves (i).

For  (ii), we observe that from (i), for  every $j\in \N$ there exist $r_{j}\in \R$ such that
\begin{equation}\label{boundblw}
      \int_{\left\{|x|<r_{j}\right\}}\sum_{k=1}^{l}\frac{\alpha_{k}^{2}}{\gamma_{k}}|\tilde{\phi}_{km}|^{2}\;dx
      =  \int_{\left\{\left|x-y_{m}\right|<r_{j}\right\}}\sum_{k=1}^{l}\frac{\alpha_{k}^{2}}{\gamma_{k}}|\phi_{km}|^{2}\;dx
     >\nu-\frac{1}{j}.
\end{equation}

By Lemma \ref{minseqbon} we have that $(\tilde{\phib}_{m})$ is a bounded sequence in $\mathbf{H}^{1}$. Then there exist a subsequence, still denoted by $(\tilde{\phib}_{m})$, and a function  $\phib\in \mathbf{H}^{1}$ such that
\begin{equation}\label{wcH1}
  \tilde{\phib}_{m} \rightharpoonup \phib,\quad \mbox{as}\quad m\to \infty,\quad\mbox{in}\quad \mathbf{H}^{1}.  
\end{equation}

On the other hand, for each $j\in \N$, since $(\tilde{\phib}_{m})$ is also bounded  in $\mathbf{H}^{1}(B_{r_j}(0))$, the compact embedding $H^{1}(B_{r_j}(0))\hookrightarrow L^{2}(B_{r_j}(0))$ combined with a standard Cantor diagonalization process yield that, up to a subsequence,  
\begin{equation*}
  \tilde{\phib}_m \to \phib,\quad \mbox{as}\quad m\to \infty,\quad\mbox{in}\quad \mathbf{L}^{2}(B_{r_j}(0)).  
\end{equation*}

Next we claim that this convergence indeed holds in $\mathbf{L}^{2}(\R^{n})$. In fact, from \eqref{wcH1}  we obtain
$Q(\phib)\leq \liminf Q (\tilde{\phib}_{m})=\nu$.
Thus, \eqref{boundblw} gives 
\begin{equation*}
        \nu\geq Q(\phib)\geq\int_{\left\{|x|<r_{j}\right\}}\sum_{k=1}^{l}\frac{\alpha_{k}^{2}}{\gamma_{k}}|{\phi}_{k}|^{2}\;dx
        =\lim_{m\to \infty}\int_{\left\{|x|<r_{j}\right\}}\sum_{k=1}^{l}\frac{\alpha_{k}^{2}}{\gamma_{k}}|\tilde{\phi}_{km}|^{2}\;dx
        >\nu-\frac{1}{j}
\end{equation*}
Therefore, by taking the limit as $j\to\infty$ in the last inequality,
\begin{equation*}
   Q(\phib)=\nu=\lim_{m\to \infty}Q (\tilde{\phib}_{m}),
\end{equation*}
which combined with \eqref{wcH1} establishes the claim.

Now, from the  Gagliardo-Nirenberg inequality, Lemma \ref{minseqbon} and the $\mathbf{L}^2$ convergence we see that $\tilde{\phib}_m \to \phib$ also in $\mathbf{L}^3$. Combining this with 
 Lemma \ref{estdifF}  we have
\begin{equation}\label{limP}
    \lim_{m\to \infty}P(\tilde{\phib}_{m})=P(\phib).
\end{equation}
From the weak convergence in $\mathbf{H}^{1}$ and \eqref{limP} we have
$E(\phib) \leq\liminf_{m}E(\tilde{\phib}_{m})=I_{\nu}$, which shows that
\begin{equation}\label{limE}
  E(\psib)=I_{\nu}=\lim_{m\to \infty} E(\tilde{\phib}_{m}). 
\end{equation}
In particular, this proves that $\phib\in A_{\nu}$ and $\tilde{\phib}_{m} \to \phib$ in $\mathbf{H}^{1}$. The proof of the lemma is thus completed.  
\end{proof}

\begin{teore}\label{precomseq}
If $(\phib_{m})$ is any minimizing sequence for \eqref{problE}, then
\begin{enumerate}
    \item[(i)] there exist a sequence $(y_{m})\subset\R^{n}$ and $\phib\in A_{\nu}$ such that
 $ (\phib_{m}(\cdot+y_{m})),$
    has a subsequence converging strongly in $\mathbf{H}^{1}$ to $\phib$. 
     \item[(ii)]
    \begin{equation}\label{liminf1}
         \lim_{m\to \infty} \inf_{\substack{\phi\in A_{\nu}\\ y\in \R^{n}}}\| \phib_{m}(\cdot+y)-\phib\|_{\mathbf{H}^{1}}=0.
     \end{equation}
      \item[(iii)] 
       \begin{equation*}
         \lim_{m\to \infty} \inf_{\phib\in A_{\nu}}\| \phib_{m}-\phib\|_{\mathbf{H}^{1}}=0.
     \end{equation*}
\end{enumerate}
\end{teore}

\begin{proof}
From Lemmas \ref{outvanishi}  and  \ref{outdichot}  we have that $\alpha=\nu$. Thus, Lemma  \ref{cocentcomp} implies that (i) holds.

For (ii) we proceed by contradiction. If  \eqref{liminf1} does not hold, then there exist a subsequence $(\phib_{m_{s}})$,
  and $\epsilon>0$ such that
 \begin{equation}\label{infsumposi}
   \inf_{\substack{\phi\in A_{\nu}\\ y\in \R^{n}}}\| \phib_{m_{s}}(\cdot+y)-\phib\|_{\mathbf{H}^{1}}\geq \epsilon,\qquad\forall s\in \N.
 \end{equation}
 Note that $(\phib_{m_{s}})$ is also a minimizing sequence for \eqref{problE}. Then, from (i) it follows that there exist   $(y_{s})\subset\R^{n}$ and $\phib_0\in A_{\nu}$ such that 
 \begin{equation*}
    \begin{split}
         \liminf_{s\to \infty}\| \phib_{m_s}(\cdot+y_s)-\phib_0\|_{\mathbf{H}^{1}}=0,
    \end{split}
 \end{equation*}
 which  obviously contradicts \eqref{infsumposi}.

Finally, (iii) follows immediately from (ii) taking into account that $E$ and $Q$ are invariant under translations.
\end{proof}

\begin{coro} \label{stabAnu} 
The set $A_{\nu}$ is stable in   $\mathbf{H}^{1}$ with respect to the flow of \eqref{system1} in the following sense: for every $\epsilon>0$, there exist $\delta>0$ such that if $\mathbf{u}_0\in\mathbf{H}^{1}$ satisfies
\begin{equation*}
    \inf_{\phib\in A_{\nu}}\| \mathbf{u}_0-\phib\|_{\mathbf{H}^{1}}<\delta,
\end{equation*}
then the solution $\mathbf{u}(t)$ of system \eqref{system1}, given by Theorem \ref{thm:globalwellposH1}, with $\mathbf{u}(0)=\mathbf{u}_0$, satisfies
\begin{equation*}
     \inf_{\phib\in A_{\nu}}\| \mathbf{u}(t)-\phib\|_{\mathbf{H}^{1}}<\epsilon,
\end{equation*}
for all $t\in \R$.
\end{coro}
\begin{proof}
Assume by contradiction the result is false. Then there exist $\epsilon>0$ and  sequences $(\phib_m)\subset\mathbf{H}^{1}$ and  $(t_{m})\subset \R$ such that
\begin{equation}\label{infsumeps0}
    \inf_{\phib\in A_{\nu}}\| \phib_m-\phib\|_{\mathbf{H}^{1}}<\frac{1}{m}
\end{equation}
and
\begin{equation}\label{infsumeps}
     \inf_{\phib\in A_{\nu}}\| \mathbf{u}_m(t_{m})-\phib\|_{\mathbf{H}^{1}}\geq\epsilon,
\end{equation}
where $\mathbf{u}_m(t)$ are the solutions of \eqref{system1} with $\mathbf{u}_m(0)=\phib_m$.
Note that \eqref{infsumeps0} means that $(\phib_m)$ converges to the set $A_\nu$, as $m\to\infty$. Consequently, since $E$ and $Q$ are continuous functions on $\mathbf{H}^1$, and $E\equiv I_\nu$ and $Q\equiv \nu$ on $A_\nu$, we deduce that $E(\phib_m)\to I_\nu$ and $Q(\phib_m)\to\nu$.

Now define $\displaystyle a_{m}=\sqrt{\frac{\nu}{Q(\phib_m)}}$ and $\mathbf{v}_m=a_m\mathbf{u}_m(t_m)$.  It is clear that $a_{m}\to 1$, as $m\to \infty$. Moreover, by  Lemma \ref{functrans} and the conservation of $Q$ and $E$,
\begin{equation}\label{Qphim1}
    Q(\mathbf{v}_m)=a_{m}^{2}Q(\mathbf{u}_m(t_m))=a_{m}^{2}Q(\phib_m)=\nu
\end{equation}
and
\begin{equation}\label{Qphim}
    \begin{split}
  E(\mathbf{v}_m)&=a_m^2E(\mathbf{u}_m(t_m))-2(a_m^3-a_m^2)P(\mathbf{u}_m(t_m))\\
  &= a_m^2 E(\phib_m)-2(a_m^3-a_m^2)P(\mathbf{u}_m(t_m)).
    \end{split}
\end{equation}

As in \eqref{GNE2} and \eqref{GNE2.1} we see that $P(\mathbf{u}_m(t_m))$ can be bounded by a quantity depending on $E(\phib_m)$ and $Q(\phib_m)$, which in turn are uniformly bounded with respect to $m$, because these are convergent sequences. So taking the limit, as $m\to\infty$, in \eqref{Qphim}, we obtain
\begin{equation*}
    \begin{split}
        \lim_{m\to \infty}E(\mathbf{v}_m)
        =\lim_{m\to \infty}E(\phib_m)
        =I_{\nu}
    \end{split}
\end{equation*}
which combined with \eqref{Qphim1} gives that $(\mathbf{v}_m)$ is a minimizing sequence of \eqref{problE}.

Part (iii) of Theorem \ref{precomseq} guarantees, for each $m\in\N$, the existence of 
 $\tilde{\phib}_{m}\in A_{\nu}$ such that $\|\mathbf{v}_m-\tilde{\phib}_{m}\|_{\mathbf{H}^{1}}<\frac{\epsilon}{2}$. Hence
 from \eqref{infsumeps},
\begin{equation*}
    \begin{split}
        \epsilon\leq \| \mathbf{u}_m(t_m)-\tilde{\phib}_{m}\|_{\mathbf{H}^{1}}
        &\leq |1-a_m|\|\mathbf{u}_m(t_m)\|_{\mathbf{H}^{1}}+\|\mathbf{v}_m-\tilde{\phib}_{m}\|_{\mathbf{H}^{1}}\\
        &\leq C|1-a_{m}|+\frac{\epsilon}{2},
    \end{split}
\end{equation*}
where we have used  that the $\mathbf{H}^{1}$ norm of the global solutions is uniformly bounded. By taking the limit, as $m\to \infty $, we arrive to a contradiction and the corollary is proved. 
\end{proof}

\subsubsection{Passing from $A_{\nu}$ to $\mathcal{G}(1,\boldsymbol{0})$}
Let us start by recalling that along  $\mathcal{G}(1,\boldsymbol{0})$ the charge $Q$ is constant (see Remark \ref{remkharge}). This means there exists a constant $\mu>0$ such that 
\begin{equation*}
Q(\psib)=\mu, \quad \mbox{for any}\quad \psib\in\mathcal{G}(1,\boldsymbol{0}).
\end{equation*}
We will show that for this constant, the sets $A_\mu$ and $\mathcal{G}(1,\boldsymbol{0})$ are the same. The proof follows the ideas presented in \cite[Lemma 4.2 ]{Corcho}.

\begin{lem} \label{AnuequG} Assume $1\leq n \leq3$.
Then $A_\mu=\mathcal{G}(1,\boldsymbol{0})$.
\end{lem}
\begin{proof} Suppose $\psib\in\mathcal{G}(1,\boldsymbol{0})$ and let us prove that $\psib\in A_{\mu}$. We already know that $\psib\in \mathcal{P}$ and $Q(\psib)=\mu$. So we only need to prove that $E(\psib)=I_\mu$. To do so, take any $\phib\in\Gamma_{\mu}$ and as in Step 2 of the proof of Lemma \ref{Inunegat} define the function
$f(\lambda)=E(\phib^{\lambda})$, $\lambda>0$. As we saw, such a function attains 
its unique  minimum  value at the point $\displaystyle \lambda_{*}=\left[\frac{2K(\phib)}{nP(\phib)}\right]^{\frac{2}{n-4}}>0$. In particular,
\begin{equation*}
\begin{split}
0=f'(\lambda_{*})=2\lambda_{*}K(\phib)-n\lambda_{*}^{n/2-1}P(\phib)
=\frac{2}{\lambda_{*}}K(\phib^{\lambda_{*}})-\frac{n}{\lambda_{*}}P(\phib^{\lambda_{*}}).
\end{split}
\end{equation*}
Thus, 
\begin{equation}\label{Klamb}
K(\phib^{\lambda_{*}})=\frac{n}{2}P(\phib^{\lambda_{*}}) \quad \mbox{and}\quad E(\phib^{\lambda_{*}})=\frac{n-4}{n}K(\phib^{\lambda_{*}}).
\end{equation}

On the other hand, from Lemma \ref{identitiesfunctionals}, we have 
\begin{equation}\label{Kgrond}
K(\psib)=\frac{n}{2}P(\psib) \quad \mbox{and} \quad E(\psib)=\frac{n-4}{n}K(\psib).
\end{equation}
Thus, since $Q(\phib)=\mu=Q(\psib)$, from \eqref{Klamb} and \eqref{Kgrond}  we obtain
 \begin{equation}\label{Jgrond}
 J(\psib)=\frac{Q(\psib)^{\frac{3}{2}-\frac{n}{4}}K(\psib)^{\frac{n}{4}}}{P(\psib)}=\frac{n}{2}\mu^{\frac{3}{2}-\frac{n}{4}}K(\psib)^{\frac{n}{4}-1}
 \end{equation}
 and
 \begin{equation}\label{Jlamb}
 J(\phib^{\lambda_{*}})=\frac{Q(\phib^{\lambda_{*}})^{\frac{3}{2}-\frac{n}{4}}K(\phib^{\lambda_{*}})^{\frac{n}{4}}}{P(\phib^{\lambda_{*}})}=\frac{n}{2}\mu^{\frac{3}{2}-\frac{n}{4}}K(\phib^{\lambda_{*}})^{\frac{n}{4}-1}.
 \end{equation}

Since $\phib^{\lambda_{*}}\in\Gamma_{\mu}\subset\mathcal{P}$ (see Step 2 in Lemma \ref{Inunegat}) and $\psib$ is a minimizer of $J$ on $\mathcal{P}$ we have $J(\psib)\leq J(\phib^{\lambda_{*}})$, which from \eqref{Jgrond} and \eqref{Jlamb} gives $K(\psib)\geq K(\phib^{\lambda_{*}})$. Hence,
from  \eqref{Klamb} and \eqref{Kgrond}
\begin{equation*}
\begin{split}
E(\phib)=f(1)
\geq f(\lambda_{*})
=E(\phib^{\lambda_{*}})
=\frac{n-4}{n}K(\phib^{\lambda_{*}})
\geq \frac{n-4}{n}K(\psib)
= E(\psib),
\end{split}
\end{equation*}
which implies $E(\psib)\leq I_{\mu}$ and shows that $\psib\in A_{\mu}$.
 
Now assume $\phib\in A_{\mu}$ and let us prove that $\phib\in \mathcal{G}(1,\boldsymbol{0})$. For that, we fix $\psib \in \mathcal{G}(1,\boldsymbol{0})$. Following the above notation, we observe that by construction
\begin{equation*}
    f(1)=E(\phib)=I_{\mu}\leq E (\phib^{\lambda_{*}})=f(\lambda_{*}).
\end{equation*}
Thus, from the definition of $\lambda_{*}$ we have $f(\lambda_{*})=f(1)$. Since $\lambda_{*}$ is the unique positive value where $f$ attains its minimum, we must  have $\lambda_{*}=1$, that is, $\phib^{\lambda_{*}}=\phib$ and
\begin{equation*}
    E(\phib)=E(\phib^{\lambda_{*}})\leq E(\psib).
\end{equation*}
This last inequality combined with \eqref{Klamb} and \eqref{Kgrond} leads to $ K(\phib)\geq K(\psib)$. But, as we proved above we always have $ K(\psib)\geq K(\phib^{\lambda_{*}})=K(\phib)$, which means that
\begin{equation}\label{Kphi}
 K(\phib)=K(\psib).
 \end{equation}
Therefore,  \eqref{Klamb} and \eqref{Kgrond} imply that
\begin{equation}\label{Pphi}
    P(\phib)=\frac{2}{n}K(\phib)=\frac{2}{n}K(\psib)=P(\psib).
\end{equation}

Together \eqref{Kphi}, \eqref{Pphi} and the fact that $\phi\in \Gamma_{\mu}$ imply that
 \begin{equation*}
 \begin{split}
 I(\phib)=\frac{1}{2}\left[K(\phib)+Q(\phib)\right]-P(\phib)
 =\frac{1}{2}\left[K(\psib)+Q(\psib)\right]-P(\psib)
 =I(\psib),
 \end{split}
 \end{equation*}
 which means that $\phib$ is also a minimizer of $I$. To complete the proof, it remains to show that $\phib$  is indeed a solution of \eqref{OB1}. But from Lagrange's multiplier theorem there exists some constant $\theta$ such that
 $$
 \gamma_k\int\nabla\phi_k\nabla g_k\;dx-\int f_k(\phib)g_k\;dx=\theta  \frac{\alpha_k^2}{\gamma_k}\int\phi_kg_k\;dx,
 $$
 for any $\mathbf{g}\in \mathbf{H}^1$. By taking $g_k=\phi_k$, summing over $k$ and using Lemma \ref{propertiesF} we infer
 \begin{equation}\label{Pphi1}
 K(\phib)-3P(\phib)=\theta Q(\phib)=\theta\mu.
 \end{equation}
 Note that from \eqref{Kphi}, \eqref{Pphi} and Lemma \ref{identitiesfunctionals} we have
 $$
  K(\phib)-3P(\phib)= K(\psib)-3P(\psib)=\frac{n}{6-n}\mu-\frac{6}{6-n}\mu=-\mu,
 $$
 which compared to \eqref{Pphi1} yields $\theta=-1$, completing the proof of the lemma.
 \end{proof}

\begin{proof}[Proof of Theorem \ref{thm:stabG}]
	It is a direct consequence of Corollary \ref{stabAnu}  and Lemma \ref{AnuequG}
\end{proof}

\begin{obs}
Corollary \ref{stabAnu} is a little bit stronger than Theorem \ref{thm:stabG}. It says that not only $A_\mu$ but all $A_\nu$, $\nu>0$ are stable by the flow of \eqref{system1}.
\end{obs}

\begin{obs}
By replacing the definition of $Q$ in \eqref{newQ} by $$Q(\phib)=\sum_{k=1}^l\frac{\alpha_k^2\omega}{\gamma_k}\|\phi_k\|_{L^2}, \qquad \omega>0,
$$
and repeating similar arguments as the ones presented in this section, actually we can prove the stability of the set $\mathcal{G}(\omega,\boldsymbol{0})$, for any $\omega>0$. Also, the fact that $\boldsymbol{\beta}=\boldsymbol{0}$ was crucial in the proof of Lemma \ref{Inunegat}. Indeed, if $\boldsymbol{\beta}\neq\boldsymbol{0}$ then the term $L(\phib)$, which is invariant under the transformation $\phib\mapsto \phib^\lambda$, also appear in the definition of the energy. In such a  case we do not know if the energy assumes a negative value.
\end{obs}

\subsection{Instability}

This subsection is devoted to prove the instability results. In the $L^2$-critical case, that is, $n=4$, we prove an instability result in the spirit of \cite{Weinstein} (see also \cite[Theorem 8.2.1]{Cazenave}).

\begin{teore}\label{thm:instabn=4}
Assume  $n=4$. Let  $\mathcal{C}$ be the set of non-trivial solutions of \eqref{OB1}. If $\psib\in \mathcal{C}$ then the standing-wave solution
	\begin{equation*}
	u_{k}(x,t)=e^{i\frac{\alpha_{k}}{\gamma_{k}}\omega t}\psi_{k}(x), \quad k=1,\ldots,l,
	\end{equation*}
	is  unstable in $\mathbf{H}^{1}$ in the following sense: for every $\epsilon>0$ there exists $\psib_0^\epsilon\in \mathbf{H}^{1}$ such that
	\begin{equation*}
	\|\psib_0^\epsilon-\psib\|_{\mathbf{H}^{1}}\leq \epsilon,
	\end{equation*}
	and the corresponding solution $\mathbf{u}^\epsilon(t)$ of \eqref{system1} (with $\beta_k=0$), satisfying $\mathbf{u}^\epsilon(0)=\psib_0^\epsilon$,
	blows up  in finite time. 
\end{teore}
\begin{proof}
Since $\psib\in \mathcal{C}$,  Lemmas  \ref{identitiesfunctionals} and \ref{lemma4.4} imply that $\psib\in\mathcal{P}$ and $E(\psib)=\dfrac{n-4}{n}K(\psib)=0$. Let $\epsilon>0$ be given and define
\begin{equation*}
\psib_0^\epsilon(x)=(1+\tilde{\epsilon})\psib(x).
\end{equation*}
where $\tilde{\epsilon}=\min\left\{\epsilon,\frac{\epsilon}{\|\boldsymbol{\psi}\|_{\mathbf{H}^{1}}}\right\}$.	We first note that
	\begin{equation*}
	\|\psib_0^\epsilon-\psib\|_{\mathbf{H}^{1}}
	=\tilde{\epsilon}\|\boldsymbol{\psi}\|_{\mathbf{H}^{1}}\leq \epsilon.
	\end{equation*}
Therefore, since $\psib_0^\epsilon\in\Sigma$ (see Lemma \ref{regellpsys}), where $\Sigma$ is the Hilbert space defined in Remark \ref{weightspc}, in order to prove the theorem, it suffices to show that $E(\psib_0^\epsilon)<0$ (see Theorem \ref{blowupZm}). But
	\begin{equation*}\label{Enegat}
	\begin{split}
	E(\psib^{\epsilon}_{0})&
	=(1+\tilde{\epsilon})^{2}K(\psib)-2(1+\tilde{\epsilon})^{3}P(\psib)\\
	&=(1+\tilde{\epsilon})^{2}\left[E(\psib)-2\tilde{\epsilon}P(\psib)\right]\\
	&=-2(1+\tilde{\epsilon})^{2}\tilde{\epsilon}P(\psib)<0,
	\end{split}
	\end{equation*}
	which is the desired.
\end{proof}

\begin{teore} \label{thm:instabn=5}
	Assume  $n=5$ and let $\mathcal{G}(1,\boldsymbol{0})$ be the set of ground states solutions of \eqref{OB1}. If $\psib\in\mathcal{G}(1,\boldsymbol{0})$, then the standing wave  
	\begin{equation*}
	u_{k}(x,t)=e^{i\frac{\alpha_{k}}{\gamma_{k}}\omega t}\psi_{k}(x), \quad k=1,\ldots,l,
	\end{equation*}
	is  unstable in $\mathbf{H}^{1}$ in the following sense: for every $\epsilon>0$ there exists $\psib_0^\epsilon\in \mathbf{H}^{1}$ such that
\begin{equation*}
\|\psib_0^\epsilon-\psib\|_{\mathbf{H}^{1}}\leq \epsilon,
\end{equation*}
and the corresponding solution $\mathbf{u}^\epsilon(t)$ of \eqref{system1} (with $\beta_k=0$), satisfying $\mathbf{u}^\epsilon(0)=\psib_0^\epsilon$,
blows up  in finite time. 
\end{teore}

To prove Theorem \ref{thm:instabn=5} we use similar arguments as those in the proof of Theorem 8.2.2 in \cite{Cazenave}. In the rest of this section we always assume $n=5$. Let us start by recalling the following virial identity (see \eqref{viriuse})
 \begin{equation}\label{dersegQn=5}
     \frac{1}{8}\frac{d^{2}}{dt^2}Q(x\mathbf{u})=K(\mathbf{u})-\frac{5}{2} P(\mathbf{u}). 
 \end{equation}
This motivates the definition of  the functional
 \begin{equation*}
   \mathcal{T}(\phib)=K(\phib)-\frac{5}{2} P(\phib).
 \end{equation*}
Also, consider the set 
 \begin{equation*}
 M=\{\phib\in\mathcal{P};\;  \mathcal{T}(\phib)=0 \}.
 \end{equation*}

In what follows we give some properties of $\mathcal{T}$ and $M$. 

\begin{lem}\label{philambpro}
 Given $\boldsymbol{\phi}\in \mathcal{P}$ and $\lambda>0$ we set $\phib^\lambda(x)=\lambda^{5/2}(\delta_{\frac{1}{\lambda}}\phib)(x)=\lambda^{5/2}\phib(\lambda x)$. Then the following properties hold.
 \begin{enumerate}
     \item[(i)] There exist a unique $\lambda_{*}(\boldsymbol{\phi})>0$ such that $\phib^{\lambda_{*}(\boldsymbol{\phi})}\in M$.
     \item[(ii)] The function $f:(0,\infty)\to \R$, $f(\lambda)=I(\phib^{\lambda})$, is concave on $(\lambda_{*}(\boldsymbol{\phi}), \infty)$. 
     \item[(iii)] $\lambda_{*}(\boldsymbol{\phi})<1$ if and only if $\mathcal{T}(\phib)<0$.
     \item[(iv)] $\lambda_{*}(\boldsymbol{\phi})=1$ if and only if $\phib\in M$.
     \item[(v)]  $I(\phib^{\lambda})<I(\phib^{\lambda_{*}(\boldsymbol{\phi})})$, $\forall \lambda>0$, $\lambda\neq \lambda_{*}(\boldsymbol{\phi}) $. 
     \item[(vi)] $\frac{d}{d\lambda}I(\phib^{\lambda})=\frac{1}{\lambda}\mathcal{T}(\phib^{\lambda})$, $\forall \lambda>0$. 
     \item[(vii)]  $|\phib^{\lambda}|^{*}=(|\phib|^{*})^{\lambda}$, where, as before $|\phib|=(|\phi_1|,\ldots,|\phi_l|)$ and $^*$ denotes the symmetric-decreasing rearrangement.
     \item[(viii)]   If $\phib_m\rightharpoonup \phib $ in $ \mathbf{H}^{1}$ and $\phib_m\to \phib$ in $ \mathbf{L}^{3}$, then  $\phib_m^{\lambda}\rightharpoonup \phib^{\lambda} $ in $ \mathbf{H}^{1}$ and $\phib_m^{\lambda}\to \phib^{\lambda}$ in $ \mathbf{L}^{3}$. 
 \end{enumerate}
\end{lem}
\begin{proof}
The proof is similar to that of Lemma 8.2.5 in \cite{Cazenave}. So we omit the details.
\end{proof}

\begin{obs}
The notation $\lambda_{*}(\boldsymbol{\phi})$ shows the dependence of $\lambda_{*}$ with respect to $\boldsymbol{\phi}$; for simplicity and as long as there is no confusion we will write $\lambda_{*}$ instead of $\lambda_{*}(\boldsymbol{\phi})$.  
\end{obs}

 \begin{coro}\label{Mneqempty}
  The set $M$ is nonempty. Moreover, if 
  \begin{equation}\label{mdef}
m=\min\limits_{\varphib\in M}I(\varphib),
\end{equation}
then for every $\phib\in \mathcal{P}$ such that $\mathcal{T}(\phib)<0$ we have       $I(\phib)\geq\mathcal{T}(\phib)+m$.
\end{coro}
 \begin{proof}
 By Lemma \ref{philambpro} (i) for any $\phib\in \mathcal{P}$ we have $\phib^{\lambda_{*}}\in M$; so $M\neq \emptyset$. For the second part, from Lemma \ref{philambpro} we have that $\lambda_{*}<1$ and $f$ is concave on $(\lambda_{*},1)$ implying the relation
\begin{equation}\label{mlem}
 f(1)\geq f(\lambda_{*})+f'(1)(1-\lambda_{*}).
\end{equation}
 Since $f(1)=I(\phib)$ and $f'(1)=\mathcal{T}(\phib)<0$ (see Lemma \ref{philambpro} (vi)), from \eqref{mlem}, we obtain
 $$
 I(\phib)\geq f(\lambda_{*})+\mathcal{T}(\phib)\geq m+\mathcal{T}(\phib),
 $$
 which proves the desired.
 \end{proof}

 \begin{lem}\label{minlemma}
 	The minimum in \eqref{mdef} is attained, that is, there exists $\varphib\in M$ such that $m=I(\varphib)$. In this case, we say that $\varphib$ is a minimizer of \eqref{mdef}.
 \end{lem}
 \begin{proof}
 Let $(\mathbf{v}_j)$ be a minimizing sequence for \eqref{mdef}, that is, a sequence in $M$ satisfying $I(\mathbf{v}_j)\to m$. Set $\mathbf{w}_j=|\mathbf{v}_j|^*$ and define 
 $$
 \phib_j:=\mathbf{w}_j^{\lambda_{*}(\mathbf{w}_j)}=(|\mathbf{v}_j|^{*})^{\lambda_{*}(\mathbf{w}_j)} =|\mathbf{v}_j^{\lambda_{*}(\mathbf{w}_j)}|^*.
 $$
 The last equality follows from Lemma \ref{philambpro} (vii).  Also, from
   Lemma \ref{philambpro} (i), 
\begin{equation}\label{phikjinM}
  \phib_j\in M\quad\mbox{i.e.}\quad \phib_j\in\mathcal{P}\quad \mbox{and} \quad K(\phib_j)=\frac{5}{2}P(\phib_j), \quad  \forall j\in \N.
\end{equation} 
 Hence, from Lemma \ref{functrans} and Lemma \ref{philambpro} (iv)-(v) we obtain
 \begin{equation*}
   I(\phib_j)=I(|\mathbf{v}_j^{\lambda_{*}(\mathbf{w}_j)}|^*)\leq I(|\mathbf{v}_j^{\lambda_{*}(\mathbf{w}_j)}|)\leq I(\mathbf{v}_j^{\lambda_{*}(\mathbf{w}_j)}) \leq I(\mathbf{v}_j^{\lambda_{*}(\mathbf{v}_j)}) = I(\mathbf{v}_j).
 \end{equation*}
Taking the limit, as $j\to\infty$, in this last inequality we see that $(\phib_j)$ is also 
 a minimizing sequence of \eqref{mdef} consisting  of non-negatives functions in  $\mathbf{H}_{rd}^{1}$.

 From the definition of functional $I$ and \eqref{phikjinM}  we have
 \begin{equation*}
 I(\phib_j)-\frac{1}{2}Q(\phib_j)=\frac{1}{2}K(\phib_j)-P(\phib_j)=\frac{2}{5}\left(K(\phib_j)-\frac{5}{2} P(\phib_j)\right)+\frac{1}{10}K(\phib_j)=\frac{1}{10}K(\phib_j).
 \end{equation*}
Since $(I(\phib_j))$ is a bounded sequence, the last equality shows that $(\phib_j)$ is  bounded   in $\mathbf{H}^{1}$. In particular there exists $A>0$ such that $Q(\phib_j)\leq A$. Thus, using \eqref{phikjinM} and the Gagliardo-Nirenberg inequality we get
 \begin{equation*}
         K(\phib_j)=\frac{2}{5}P(\phib_j)\leq 
      CQ(\phib_j)^{\frac{1}{4}}K(\phib_j)^{\frac{5}{4}}
         \leq CA^{\frac{1}{4}}K(\phib_j)^{\frac{5}{4}},
 \end{equation*}
which implies that $(K(\phib_j))$ is bonded from below. Combining this with \eqref{phikjinM} we get that there exists $\eta>0$ such that $P(\phib_j)\geq \eta$.

On the other hand, 
since the embedding $H_{rd}^{1}(\R^{5})\hookrightarrow L^{3}(\R^{5})$ is compact, we can find  $\phib\in \mathbf{H}^{1} $ such that, up to a subsequence,
 \begin{equation*}
 \phib_j\rightharpoonup \phib \;\;\mbox{in  $\mathbf{H}^{1}$} \qquad \mbox{and} \qquad  \phib_j\to \phib_j\;\;\mbox{in $\mathbf{L}^3$}.
 \end{equation*}
As in the proof of Theorem \ref{thm:existenceGSJgeral} we conclude that
 \begin{equation*}
     P(\phib)=\lim_{j\to \infty}P(\phib_j)\geq \eta>0.
 \end{equation*}
 In particular, $\phib\in \mathcal{P}$.  
 
 Next, define $\varphib=\phib^{\lambda_{*}(\phib)}$. By Lemma  \ref{philambpro}  (i) and (viii) we see that $\varphib\in M$ and
 \begin{equation*}
\phib_j^{\lambda_{*}(\phib)}\rightharpoonup \varphib \;\;\mbox{in  $\mathbf{H}^{1}$} \qquad \mbox{and} \qquad  \phib_j^{\lambda_{*}(\phib)}\to \varphib\;\;\mbox{in $\mathbf{L}^3$}.
\end{equation*}
We can use these convergences to conclude that
 \begin{equation*}
     \begin{split}
        I(\varphib)
        \leq \liminf_{j\to\infty}I(\phib_j^{\lambda_{*}(\phib)})
        \leq \liminf_{j\to\infty}I( \phib_j^{\lambda_{*}(\phib_j)})
       =\liminf_{j\to\infty}I(\phib_j)
        =m,
     \end{split}
 \end{equation*}
 where we have used Lemma \ref{philambpro} (v) and (iv). This shows that $I(\varphib)=m$ and the proof of the lemma is complete.
 \end{proof}
  
 \begin{lem}\label{phiinC}
 If $\phib$ is a minimizer of \eqref{mdef} then it is a solution of \eqref{OB1}. 
 \end{lem}
 \begin{proof}
For $\sigma>0$ define 
 $\phib^{\sigma}(x)=\sigma^{-2}(\delta_\sigma\phib)(x).$ Since $\phib\in M$ we have 
 \begin{equation*}
         \mathcal{T}(\phib^{\sigma})=K(\phib^{\sigma})-\frac{5}{2}P(\phib^{\sigma})
         =\sigma^{-1}\left[K(\phib)-\frac{5}{2}P(\phib)\right]
         =\sigma^{-1} \mathcal{T}(\phib)
         =0.
 \end{equation*}
 Thus, $\phib^{\sigma}\in M$, for any $\sigma>0$. Using this and  the function $f(\sigma)=I(\phib^{\sigma})$, we conclude
 \begin{equation*}
     f(1)=I(\phib)\leq I(\phib^{\sigma})=f(\sigma), \qquad \sigma>0.
 \end{equation*}
This means that $f$ attains a minimum at $\sigma=1$. In particular, $f'(1)=0$. 

Now, by using the definition of $I$ we see that 
$f'(\sigma)=-\frac{\sigma^{-2}}{2}K(\phib)+\frac{1}{2}Q(\phib)+\sigma^{-2}P(\phib)$. 
Lemma \ref{frede}, \eqref{b1} and the fact  that $\phib\in M$ imply
 \begin{equation*}
     \begin{split}
        0= f'(1)=-\frac{1}{2}K(\phib)+\frac{1}{2}Q(\phib)+P(\phib)
         =\frac{1}{2}\left[K(\phib)+Q(\phib)-3P(\phib)\right]
         =\frac{1}{2}I'(\phib)(\phib).
     \end{split}
 \end{equation*}
 Therefore,
 \begin{equation}\label{Iprime0}
     I'(\phib)(\phib)=0.
 \end{equation}
 
On the other hand, using Lemma \ref{frede} again,
 \begin{equation}\label{Tneg}
         \mathcal{T}'(\phib)(\phib)=K'(\phib)(\phib)
         -\frac{5}{2}P'(\phib)(\phib)
         =2K(\phib)-\frac{15}{2}P(\phib)
         =2K(\phib)-3K(\phib)
        <0.
 \end{equation}
 Since  $\phib$ is a minimizer of \eqref{mdef},  there is a Lagrange multiplier, say, $\Lambda$, such that
 $I'(\phib)=\Lambda \mathcal{T}'(\phib)$.
 Putting this together with \eqref{Iprime0}  we obtain
 \begin{equation*}
         0=I'(\phib)(\phib)
         =\Lambda  \mathcal{T}'(\phib)(\phib).
 \end{equation*}
 Thus,  \eqref{Tneg} implies that $\Lambda=0$ which yields $I'(\phib)=0$. 
 \end{proof}

 \begin{lem}\label{equivIMC}
  A function $\psib\in \mathcal{P} $ belongs to $\mathcal{G}(1,\boldsymbol{0})$ if and only if  it is a minimizer of \eqref{mdef}.
 \end{lem}
 \begin{proof}
 Set 
 \begin{equation*}
\tau=\min\limits_{\varphib\in \mathcal{C}}I(\varphib),
\end{equation*}
where $\mathcal{C}$ is the set of all solution of \eqref{OB1}. In order to prove the lemma, it suffices to show that $m=\tau$.
Take any $\phib\in \mathcal{G}(1,\boldsymbol{0})$ (from Theorem \ref{thm:existenceGSJgeral} we already know that this set in nonempty). Then $I(\phib)=\tau$ and from Lemma \ref{identitiesfunctionals} we have $\phib\in M$.
 Thus, 
 \begin{equation}\label{equivtau}
     m\leq I(\phib)=\tau. 
 \end{equation}
 
On the other hand, let $\phib$ be a minimizer of \eqref{mdef} (from Lemma \ref{minlemma} such a element always exist). By Lemma \ref{phiinC} we have that $\phib\in \mathcal{C}$. Then,
 \begin{equation}\label{equivtau1}
     \tau\leq I(\phib)=m.
 \end{equation}
Inequalities \eqref{equivtau} and \eqref{equivtau1} yield the desired.
 \end{proof}

 With the above constructions in hand we are able to prove Theorem \ref{thm:instabn=5}.
 
 \begin{proof}[Proof of Theorem \ref{thm:instabn=5}]
Take  $\psib\in \mathcal{G}(1,\boldsymbol{0})$. For $\lambda>0$  define $  \psib^{\lambda}(x)=\lambda^{5/2}\psib(\lambda x)$.
 By Lemma \ref{equivIMC}, $\psib\in M$ and it is a minimizer of \eqref{mdef}. In particular,
 \begin{equation}\label{T=0}
     \mathcal{T}(\psib)=0 \quad \mbox{and} \quad I(\psib)=m.
 \end{equation}
 From Lemma \ref{philambpro} (vi) and \eqref{T=0} we have
 \begin{equation*}
    \begin{split}
         \mathcal{T}(\psib^{\lambda})=\lambda \frac{d}{d\lambda}I(\psib^{\lambda})
         =\lambda\left[\lambda K(\psib)-\frac{5}{2}\lambda^{3/2} P(\psib)\right]
         =\lambda^{2}(1-\lambda^{1/2})K(\psib). 
    \end{split}
 \end{equation*}
 Hence,
 \begin{equation}\label{Tlambneg}
     \mathcal{T}(\psib^{\lambda})<0,\quad \mbox{for any}\; \lambda>1.
 \end{equation}
Moreover, from \eqref{T=0} and Lemma \ref{philambpro} (iv) and (v),
 \begin{equation}\label{Iphimenm}
     I(\psib^{\lambda})<I(\psib^{\lambda_{*}(\boldsymbol{\psi})})=I(\psib)=m.
 \end{equation}

From now on, we assume $\lambda>1$. Let $\mathbf{u}^{\lambda}(t)$ be the maximal solution of  \eqref{system1} (with $\beta_k=0$), given by Theorem \ref{localexistenceH1}, corresponding to the initial data $\psib^{\lambda}$. Let $\tilde{I}$ be the maximal existence interval. By the conservation of the energy and the charge we get, for all $t\in \tilde{I}$,
 \begin{equation}\label{Iu=Iphi}
     \begin{split}
         I(\mathbf{u}^{\lambda}(t))=\frac{1}{2}E(\mathbf{u}^{\lambda}(t))+\frac{1}{2}Q(\mathbf{u}^{\lambda}(t))
         =\frac{1}{2}E(\psib^{\lambda})+\frac{1}{2}Q(\psib^{\lambda})
         =I(\psib^{\lambda}),
     \end{split}
 \end{equation}
 
Since the function $g(t)=\mathcal{T}(\mathbf{u}^{\lambda}(t))$, $t\in \tilde{I}$, is  continuous  and, by \eqref{Tlambneg}, $g(0)=\mathcal{T}(\psib^{\lambda})<0$,
 there exist $\delta>0$ such that $(-\delta,\delta)\subset \tilde{I}$ and $g(t)<0$, for all $t\in  (-\delta,\delta)$.
 In particular, from Corollary \ref{Mneqempty}, \eqref{Iphimenm} and \eqref{Iu=Iphi} we obtain, for each $t\in  (-\delta,\delta)$,
 \begin{equation}\label{Tuneg}
    g(t)\leq I(\mathbf{u}^{\lambda}(t))-m=I(\psib^\lambda)-m=:-\eta, \quad \eta>0. 
 \end{equation}
 
 We claim that $g(t)<0$ for all $t\in \tilde{I}$ and then, \eqref{Tuneg} holds for all $t\in \tilde{I}$.  Indeed, if not, there exist $\overline{t}\in \tilde{I}\setminus (-\delta,\delta)$ such that $g(\overline{t})\geq 0$. We assume first $g(\bar{t})>0$. By the intermediate value theorem must exists $\tilde{t}$ such that $g(\tilde{t})=0$, that is, $\mathbf{u}^\lambda(\tilde{t})\in M$. In addition, from \eqref{Iu=Iphi} and \eqref{Iphimenm} we obtain that $ I(\mathbf{u}^{\lambda}(\tilde{t}))<m$, which is a contradiction. Of course the case $g(\bar{t})=0$ cannot occur either.  Hence the claim follows.

  Finally,  since \eqref{dersegQn=5} gives
 \begin{equation*}
      \frac{d^{2}}{dt^2}Q(x\mathbf{u}^{\lambda}(t))=8\left[K(\mathbf{u}^{\lambda}(t))-\frac{5}{2} P(\mathbf{u}^{\lambda}(t)) \right] 
     =8\mathcal{T}(\mathbf{u}^{\lambda}(t))=g(t)
     <-8\eta, \quad \forall t\in \tilde{I},
 \end{equation*}
 and $\psib^\lambda\in\Sigma$, as in the proof of Theorem \ref{thm:sharpglobalexistencecondn=5} we conclude that $\tilde{I}$ must be finite.

  The conclusion of the theorem then follows because $\psib^\lambda\to \psib$ in $\mathbf{H}^1$, as $\lambda\to1^+$.
 \end{proof}

\section*{Acknowledgement}
A.P. is partially supported by CNPq/Brazil grants 402849/2016-7 and 303098/2016-3.

N.N. is partially supported by Universidad de Costa Rica, through the OAICE.

\end{document}